%% file: main.tex
\documentclass{article}
\usepackage[T1]{fontenc}
\usepackage{hyperref}
\usepackage{mathrsfs} 
\usepackage{numprint}
\usepackage{amsthm}
\usepackage{amsmath} 
\usepackage{mathrsfs} 
\usepackage{amssymb} 
\usepackage{amsfonts} 
\usepackage{multicol}
\usepackage{pdfpages}
\usepackage{a4wide}

\usepackage{pythontex}
\usepackage{graphicx} 
\usepackage{wrapfig} 
\usepackage{stmaryrd}
\usepackage{enumitem} 
\usepackage{comment}

\usepackage{multicol}
\usepackage{tikz}
\usepackage{graphicx}
\usepackage{colortbl}
\usepackage{array}
\usepackage[absolute]{textpos} 
\usepackage{color} 
\definecolor{Prune}{RGB}{99,0,60}
\usepackage{mdframed}
\usepackage{fancyhdr}
\usepackage{geometry}
\usepackage{ragged2e}

\input{commandes_latex}

\title{\textbf{Independence of the Diophantine exponents associated with linear subspaces}}
\author{Gaétan GUILLOT}
\date{}

\begin{document}
\newgeometry{top=2.25 cm, bottom=2.25cm, left=2.6cm, right=2.6cm}
\maketitle
\begin{abstract}
 We elaborate on a problem raised by Schmidt in 1967 which generalizes the theory of classical Diophantine approximation to subspaces of $\R^n$. We consider Diophantine exponents for linear subspaces of $\R^n$ which generalize the irrationality measure for real numbers. We prove here that we have no smooth relations among some functions associated to these exponents. To establish this result, we construct subspaces for which we are able to compute the exponents. 
\end{abstract}

\bigskip

\section{Introduction}
In classical Diophantine approximation, one studies how well real numbers (or points of $\R^n$) can be approximated by rational numbers (or rational points). In 1967, Schmidt \cite{Schmidt} stated a generalisation of this problem in which one studies how well can subspaces of $\R^n$ can be approximated by rational subspaces. We define briefly the necessary notions for the study of this problem. We use the definitions and notation from \cite{Schmidt}, \cite{Joseph-these}, \cite{joseph_exposants} and \cite{joseph_spectre}. Additional details regarding the results of this article can be found in \cite{Guillot_these}.

We say that a subspace of $\R^n$ is rational if it has a basis of vectors with rational coordinates. We denote by $\RR_n(e)$ the set of all rational subspaces of dimension $e$ of $\R^n$. To such a rational subspace $B \in \RR_n(e)$ we can associate a point $\eta = (\eta_1, \ldots, \eta_N) \in \P^N(\R)$ with $N = \binom{n}{e}$ called the Grassmann (or Plücker) coordinates of $B$. We can choose a representative vector $\eta$ with coprime integer coordinates and we define:
$$ H(B) = \| \eta \|$$
where $\| \cdot \|$ is the Euclidean norm on $\R^N$. Let us notice that if $X_1, \ldots, X_e$ is a $\Z$-basis of $B \cap \Z^n$, then $H(B) = \| X_1 \wedge \ldots \wedge X_e \|$. More results on the height can be found in \cite{Schmidt} and \cite{Schmidt_book}.

Let us fix $n \in \Nx$. For $d,e \in \llbracket 1, n \rrbracket^2$ we set $g(d,e,n) = \max(0 , d + e -n)$. For $j \in \llbracket 1, 
 \min(d,e) -g(d,e,n) \rrbracket $ we say that a subspace $A$ of dimension $d$ of $\R^n$ is $(e,j)\tir$irrational if 
$$ \forall B \in \RR_n(e), \quad \dim( A \cap B) < j + g(d,e,n).$$
We denote by $\II_n(d,e)_j$ the set of all $(e,j)\tir$irrational subspaces $A$ of dimension $d$ of $\R^n$.

We define now the notion of proximity between two subspaces. For $X, Y \in \R^n \setminus \{ 0 \},$ set 
$$ \omega ( X,Y) = \frac{\| X \wedge Y \|}{ \|X \| \cdot\|Y \| }$$
where $ X \wedge Y $ is the exterior product of $X$ and $Y$. Geometrically, $\omega(X,Y)$ is the absolute value of the sine of the angle between $X$ and $Y$. Let $A$ and $B$ be two subspaces of $\R^n$ of respective dimensions $d$ and $e$. We construct by induction $t = \min(d,e)$ angles between $A$ and $B$. Let us define 
$$ \omega_1(A,B) = \min\limits_{\substack{X \in A \setminus \{ 0 \} \\Y \in B \setminus \{ 0 \}}} \omega(X,Y)$$
and $(X_1, Y_1)\in A \times B$ such that $\omega(X_1, Y_1) = \omega_1(A,B)$. Let us assume that $\omega_1(A,B), \ldots, \omega_j(A,B)$ and $(X_1, Y_1), \ldots, (X_j, Y_j)$ has been constructed for $j \in \llbracket 1, t-1 \rrbracket$. Let $A_j$ and $B_j$ be respectively the orthogonal complements of $\Span ( X_1, \ldots, X_j)$ in $A$ and $\Span (Y_1, \ldots, Y_j)$ in $B$. We define 
$$ \omega_{j+1}(A,B) = \min\limits_{\substack{X \in A_j \setminus \{ 0 \} \\Y \in B_j \setminus \{ 0 \}}} \omega(X,Y) $$
and $(X_{j+1}, Y_{j+1})\in A \times B$ such that $\omega(X_{j+1}, Y_{j+1}) = \omega_{j+1}(A,B)$. We can notice that if $g(d,e,n) > 0$ then $\omega_1(A,B) = \ldots = \omega_{g(d,e,n)}(A,B) = 0 $, so we introduce for $j \in \llbracket 1 , \min(d,e) - g(d,e,n) \rrbracket$: 
$$ \psi_j(A,B) = \omega_{j + g(d,e,n)}(A,B). $$
We now have all the tools to define the Diophantine exponents studied in this paper.
\begin{defi}
 Let $(d,e) \in \llbracket 1,n-1 \rrbracket^2$, $j \in \llbracket 1, \min(d,e) - g(d,e,n) \rrbracket $ and $A \in \II_n(d,e)_j$. We define $\mu_n(A|e)_j$ as the supremum of the set of all $\mu > 0$ such that there exist infinitely many $B \in \RR_n(e)$ such that
 $$ \psi_j(A,B) \leq H(B)^{-\mu}.$$
\end{defi}

The goal here is to study some joint spectrum of these exponents, that is to say determine the values taken by functions of the form 
$$ 
 \left| \begin{array}{ccc}
 \bigcap\limits_{(e,j)\in U} \II_n(d, e)_{j} &\longrightarrow & (\R \cup \{ + \infty\})^{U} \\
 A & \longmapsto & (\mu_n(A|e)_{j})_{(e,j) \in U}
 \end{array}\right. $$
where $d$ is fixed and $U$ is a subset of $V_{d,n} = \lbrace (e,j) \mid e \in \llbracket 1, n-1\rrbracket, j \in \llbracket 1 , \min(d,e) -g(d,e,n)\rrbracket \rbrace$. 

In the case $d =1$ the joint spectrum $\left( \mu_n(\cdot|1)_1, \ldots, \mu_n(\cdot|n-1)_1 \right)$ has been fully described by Roy \cite{Roy} using parametric geometry of numbers and the previous results of Schmidt \cite{Schmidt} and Laurent \cite{Laurent_omega_un}. 

\begin{theo}[Roy 2016] \label{1theo_roy_spec_joint} For any $\mu_1, \ldots, \mu_{n-1} \in [1, + \infty ]$ satisfying $\mu_1~\geq~\frac{n}{n-1} $ and 
\begin{align*}
 \forall e \in \llbracket 2, n-1 \rrbracket, \quad \frac{e \mu_e}{\mu_e + e -1 } \leq \mu_{e-1} \leq \frac{(n-e)\mu_e}{n-e+1},
\end{align*}
there exists $A \in \II_n(1,n-1)_1$ such that 
 $\forall e \in \llbracket 1, n-1 \rrbracket, \quad \mu_n(A| e)_1 = \mu_e. $
\end{theo}

A corollary of this theorem is that are there no smooth relations between these exponents. For $W$ a set and $f_1, \ldots, f_k $ functions from $W$ to $\R \cup \{ + \infty\} $, we say that $f_1, \ldots, f_k $ are smoothly independent on $W$ if there is no submersion $h: \R^{k} \to \R $ such that for all $ w \in W$
$$ (f_1(w), \ldots, f_k(w)) \in \R^k \quad \Longrightarrow \quad h( f_1(w), \ldots, f_k(w)) = 0. $$
In particular, if such a relation were to exist, the image of $(f_1, \ldots, f_k)$ intersected with $\R^k$ would be contained in a hypersurface of $\R^k$.

\begin{cor}\label{1cor_roy_indep}
 The functions $ \mu_n(\cdot|1)_1, \ldots, \mu_n(\cdot|n-1)_1 $ are smoothly independent on $\II_n(1,n-1)_1$.
\end{cor}

By Theorem~\ref{1theo_roy_spec_joint}, the image of the joint spectrum is indeed a subset of $\R^{n-1}$ with non-empty interior ; it is not contained in any hypersurface of $\R^{n-1}$. The goal of this paper is to generalize the result of Corollary~\ref{1cor_roy_indep} with $d \geq 1$. We will show the following theorems. In Theorems~\ref{1theo_premier_angle},\ref{1theo_min} and~\ref{1theo_d}, we consider not only lines but subspaces of any dimension $d$. The first one deals with the first angle while the other two treat the case of the last angle.

\begin{theo}\label{1theo_premier_angle}
Given $d \in \{1, \ldots, n-1 \},$ the functions $\mu_n(\cdot|1)_1, \ldots, \mu_n(\cdot|n-d)_1$ are smoothly independent on $ \bigcap\limits_{e =1}^{n-d} \II_n(d,e)_{1} = \II_n(d,n-d)_{1} $.
\end{theo}

\begin{theo}\label{1theo_min}
 Assume that $d $ divides $n$. The functions $ \mu_n(\cdot|1)_{\min(d,1)}, \ldots, \mu_n(\cdot|n-d)_{\min(d,n-d)} $ are smoothly independent on $ \bigcap\limits_{e =1}^{n-d} \II_n(d,e)_{\min(d,e)}$.
\end{theo}

\begin{theo}\label{1theo_d}
 Assume that $d $ divides $n$. The functions $ \mu_n(\cdot|d)_{d}, \ldots, \mu_n(\cdot|n-1)_{d} $ are smoothly independent on $ \bigcap\limits_{e =d}^{n-1} \II_n(d,e)_{d}$.
\end{theo}

In order to prove these theorems, we construct subspaces of $\R^n$ with Diophantine exponents that can be computed. With these constructions, we can prove that the image of $(\mu_n(\cdot|e)_{j})_{(e,j) \in U}$ contains a subset with non-empty interior which gives Theorems~\ref{1theo_premier_angle}, \ref{1theo_min} and~\ref{1theo_d}. We will herein show the following theorems.

\begin{theo}\label{6theo_principal}
Let $d \in \llbracket 1, n-1 \rrbracket $. There exists an explicit constant $C_d >0$, such that for any $ (\gamma_1 \ldots, \gamma_{n-d}) \in [C_{d}, +\infty[^{n-d}$, there exists $A \in \II_n(d,n-d)_1$ satisfying
\begin{align*}
\forall e \in \llbracket 1, n-d \rrbracket, \quad \mu_n(A|e)_1 = \max\limits_{i \in \llbracket 0, n-d-e \rrbracket } \gamma_{i+1}\ldots\gamma_{i+e}.
\end{align*}
\end{theo}

Theorem~\ref{6theo_principal} focuses on constructing a subspace whose Diophantine exponents corresponding to the first angle are known in the case $d +e \leq n$. Theorem~\ref{1theo_construction} is more general, as it provides a subspace for which we compute many exponents including in the case $d +e > n$.

\begin{theo}\label{1theo_construction} 
Assume that $n = (m+1)d$ with $m \geq 1$. Let $\cons \label{7cons_petite_hyp_theoc2c1} =\left(1+ \frac{1}{m} \right)^{\frac{1}{d}}$ and $1 < \cons \label{7cons_petite_hyp_theoc2} < c_{\ref{7cons_petite_hyp_theoc2c1}} $.
Let $(\beta_{1,1}, \ldots, \beta_{1, m}) \in \mathbb{R}^m$ such that:
\begin{align}\label{7hypothèse_prop_princi1}
\min\limits_{\ell \in \llbracket 1, m \rrbracket}(\beta_{1,\ell}) > {(3d)^{\frac{c_{\ref{7cons_petite_hyp_theoc2}}}{c_{\ref{7cons_petite_hyp_theoc2}}-1}}} \text{ and } \min\limits_{\ell \in \llbracket 1, m \rrbracket}(\beta_{1,\ell})^{c_{\ref{7cons_petite_hyp_theoc2c1}}} > \max\limits_{\ell \in \llbracket 1, m \rrbracket}(\beta_{1,\ell})^{c_{\ref{7cons_petite_hyp_theoc2}}}.
\end{align}
For $i \in \llbracket 2,d \rrbracket$ let $(\beta_{i,1}, \ldots, \beta_{i, m}) \in \mathbb{R}^m$ satisfy for all $i \in \llbracket 1,d-1 \rrbracket$:
\begin{align}
\min\limits_{\ell \in \llbracket 1, m \rrbracket}(\beta_{i,\ell})^{c_{\ref{7cons_petite_hyp_theoc2c1}}} > \max\limits_{\ell \in \llbracket 1, m \rrbracket}(\beta_{i+1,\ell}) \label{7hypothèse_prop_princi2} \\
\text{ and } \min\limits_{\ell \in \llbracket 1, m \rrbracket}(\beta_{i+1,\ell}) > \max\limits_{\ell \in \llbracket 1, m \rrbracket}(\beta_{i,\ell})^{c_{\ref{7cons_petite_hyp_theoc2}}} \label{7hypothèse_prop_princi3}.
\end{align}
There exists a subspace $A$ of dimension $d$ in $\mathbb{R}^{n}$ such that for all $e \in \llbracket 1, n-1\rrbracket$ and $k \in \llbracket 1 + g(d,e,n), \min(d,e) \rrbracket$ satisfying $e < k (m+1)$, we have $A \in \II_n(d,e)_{k-g(d,e,n)}$ and:
\begin{align*}
\mu_n(A|e)_{k-g(d,e,n)}
&= \frac{1}{ \sum\limits_{q = 1 + \max(0, e-mk)}^k \frac{1}{K_{q+d-k,v_q} }}
\end{align*}
where $v_1, \ldots, v_k$ are defined by letting $u$ and $v$ be such that $e = k v + u$ is the Euclidean division of $e$ by $k$, and:
\begin{align}\label{7eq_def_vj}
v_q &= \left\{
\begin{array}{lll}
v + 1 &\text{ if } q \in \llbracket 1, u \rrbracket \\
v &\text{ if } q \in \llbracket u+1, k \rrbracket
\end{array}
\right.
\end{align}
and finally,
\begin{align*}
\forall i \in \llbracket 1, d \rrbracket, \quad \forall v \in \llbracket 1, m \rrbracket, \quad K_{i,v} = \max\limits_{\ell \in \llbracket 0, m-v \rrbracket} \beta_{i, \ell +1} \ldots \beta_{i, \ell + v}.
\end{align*}
\end{theo}

The goal here is to construct a subspace $A$ of dimension $d$ with numerous known Diophantine exponents. It is noteworthy that we manage to compute "intermediate" exponents, namely those where $j \neq 1$ and $j \neq \min(d,e)$. 

\bigskip
Section~\ref{sect_cas_d1} is devoted to the proof of Theorem~\ref{6theo_principal} in the case $d =1$ ; the general case is proven in section~\ref{sect_cas_dd} where we also establish Theorem~\ref{1theo_premier_angle}. Section~\ref{5sect_cas_general} presents the proof of Theorem~\ref{1theo_construction}. In section~\ref{sect_outils} we expose some tools requisite for the construction of subspaces and section~\ref{sect_sum} furnishes a proof of a theorem enabling the computation of Diophantine exponents for a direct sum of lines used in section~\ref{5sect_cas_general}. Finally we deduce Theorems~\ref{1theo_min} and~\ref{1theo_d} from Theorem~\ref{1theo_construction} in section~\ref{sect_indep}

\section{Tools}\label{sect_outils}
This section is devoted to presenting various tools, some more technical than others, that will be used in computing the Diophantine exponents or constructing $(e,j)\tir$irrational subspaces. First we state a few results on $(e,j)$-irrationality.

\begin{prop}\label{prop_ortho_j_irr}
Let $A$ be a subspace of $\R^n$ with dimension $d$, and $e \in \llbracket 1,n-1 \rrbracket$. For any $j \in \llbracket 1, \min(d,e) -g(d,e,n) \rrbracket $, we have:
\begin{align*}
A \text{ is } (e,j)\text{-irrational} \Longleftrightarrow A^\perp \text{ is } (n-e,j)\text{-irrational}
\end{align*}
where $A^\perp$ denotes the orthogonal complement of $A$ in $\R^n$.
\end{prop}

\begin{proof}
 We recall that $g(d,e,n) = \max(0, d + e - n)$, and we observe that $g(n-d,n-e,n) = g(d,e,n) + n - d - e$. On the other hand, $ \dim(A^\perp \cap B^\perp ) = n - \dim(A +B) = n- d-e + \dim (A\cap B)$ for every rational subspace $B$ of dimension $e$. This gives:
\begin{align*}
 \dim(A\cap B) < j + g(d,e,n) \Longleftrightarrow \dim(A^\perp \cap B^\perp) < j + g(n - d,n- e,n).
\end{align*}
Proposition~\ref{prop_ortho_j_irr} directly follows from this equivalence since $(B^\perp)^\perp = B$ and $\dim(B^\perp) = n-e$.

\end{proof}

\begin{prop}\label{prop_j_irr_descend_monte}
Let $A$ be a vector subspace of $\R^n$ with dimension $d$. Then for any $j \in \llbracket 1, \min(d,n-d) \rrbracket$:
\begin{align*}
A \in \II_n(d,n-d)_j \quad \Longrightarrow \quad \forall e \in \llbracket j,n-j \rrbracket, \quad A \in \II_n(d,e)_j.
\end{align*}
\end{prop}

\begin{proof} Let $A$ be $(n-d,j)$-irrational of dimension $d$. \\
 \textbullet \, \underline{First case:} Let $e \in \llbracket j,n-d \rrbracket$. 
Let $B$ be a rational subspace of dimension $e$, then there exists $C$ a rational subspace of dimension $n-d$ such that $B \subset C$. 
We have $\dim(A \cap B) \leq \dim(A \cap C) < j + g(d,n-d,n) = j + g(d,e,n).$
Thus, if $A \in \II_n(d,n-d)_j$ then $A \in \II_n(d,e)_j$ for all $ e \in \llbracket j,n-d \rrbracket$.

\textbullet \, \underline{Second case:} Let $e \in \llbracket n-d + 1,n-j \rrbracket$. According to Proposition~\ref{prop_ortho_j_irr}, $A^\perp$ is $(n-d,j)$-irrational. 
We then use the first part of the proof and obtain $A^\perp \in \II_n(n-d,f)_j$ for any $f \in \llbracket j, n-(n- d) \rrbracket $. In particular, since $n-e \in \llbracket j, d \rrbracket $, we have:
$A^\perp \in \II_n(n-d,n-e)_j.$ By applying again Proposition~\ref{prop_ortho_j_irr}, we have $A \in \II_n(d,e)_j.$

\end{proof}

Now we can state a result enabling us to show that certain specific subspaces are $(e,j)$-irrational.

\begin{lem}\label{3lem_1_irrat}
Let $1 \leq d \leq n -1$. Let $M = \begin{pmatrix} G \\ \Sigma \end{pmatrix} \in \MM_{n,d}(\mathbb{R})$ satisfy:
\begin{enumerate}[label=(\roman*)]
\item $G \in \text{GL}_{d}(\mathbb{R})$ and $\Sigma \in \MM_{n-d,d}(\mathbb{R} )$. \label{condition_1}
\item The entries of $\Sigma$ form an algebraically independent set over $\mathbb{Q}(\FF)$, where $\FF$ is the set of entries of $G$. \label{condition_2}
\end{enumerate}
Then, for all $e \in \llbracket 1, n-1 \rrbracket$, the subspace spanned by the columns of $M$ is $(e,1)$-irrational.
\end{lem}

\begin{proof}
We denote by $Y_1, \ldots, Y_d$ the columns of the matrix $M$ and $A = \Span(Y_1, \ldots, Y_d)$ the space spanned by these columns. We only show that $A$ is $(n-d,1)$-irrational. By Proposition~\ref{prop_j_irr_descend_monte}, we then have that $A$ is $(e,1)$-irrational for all $e \in \llbracket 1, n-1 \rrbracket$.

Let $B$ be a rational space of dimension $n-d$. Suppose by contradiction that $A \cap B \neq \{ 0 \} $. Denoting by $Z_1, \ldots, Z_{n-d}$ a rational basis of $B$, we have $Y_1 \wedge \ldots Y_d \wedge Z_1 \wedge \ldots \wedge Z_{n-d} = 0.$ In the following, we denote by $Q = (Z_1 | \cdots | Z_{n-d})$ the matrix whose columns are the $Z_i$. The equality $Y_1 \wedge \ldots Y_d \wedge Z_1 \wedge \ldots \wedge Z_{n-d} = 0$ implies the nullity of the determinant of the following matrix $
\begin{pmatrix}
\begin{matrix}G \\ 
\Sigma 
\end{matrix} & \begin{matrix}
Z_1 & \cdots & Z_{n-d}
\end{matrix}
\end{pmatrix} \in \mathcal{M}_n(\R).$ 
\bigskip 

This determinant is a polynomial with rational coefficients in the coefficients of $M$, as the $Z_i$ are rational.
We can also view this determinant as a polynomial in $\Q(\FF)[X_1, \ldots, X_{d(n-d)}]$ evaluated at the $d(n-d)$ coefficients of $\Sigma$. Since, by~\ref{condition_2}, these coefficients form an algebraically independent family over $\Q(\FF)$, then this polynomial is identically zero.
Thus, we can replace the coefficients of $\Sigma$ by any real family, and the determinant will be zero. Therefore, we have:
\begin{align}\label{det_nul}
\forall \Sigma \in \mathcal{M}_{n-d,d}(\R ), \quad \det \begin{pmatrix}
\begin{matrix}G \\ 
\Sigma 
\end{matrix} & \begin{matrix}
Z_1 & \cdots & Z_{n-d}
\end{matrix}
\end{pmatrix} = 0.
\end{align}

For $\Delta$ a minor of size $n-d$ of $Q$, we denote by $\Ind(\Delta)$ the set of indices $1 \leq i_1 < \ldots < i_{n-d} \leq n $ of the rows from which $\Delta$ is extracted. We also denote by $\Mat({\Delta}) $ the submatrix of $Q$ of size $n-d$ from which $\Delta$ is obtained.
We prove by strong induction on $r \in \llbracket 0, \min(d, n-d) \rrbracket$ the property: \\
"For any minor $\Delta$ of size $n-d$ of $Q$ such that $\#(\Ind(\Delta) \cap \llbracket 1,d \rrbracket )= r $, we have $\Delta = 0$."
\begin{itemize}[label = $\bullet$]
\item If $r = 0 $, then $\Delta$ is the minor extracted from the last $n - d $ rows of $Q$. Taking $\Sigma = (0) $ in $(\ref{det_nul})$, we obtain:
\begin{align*}
0= \det \begin{pmatrix}
\begin{matrix}G \\ 
0
\end{matrix} & \begin{matrix}
* \\
\Mat({\Delta})
\end{matrix}
\end{pmatrix} = \det(G) \Delta,
\end{align*}
and thus $\Delta = 0 $ because $G$ is invertible.
\item Let $r \in \llbracket 1,\min(d,n-d) \rrbracket$ and assume the property is true for all $0 \leq k < r$. Let $\Delta$ be a minor of size $n-d$ of $Q$ such that $ \#(\Ind(\Delta) \cap \llbracket 1,d \rrbracket )= r $.
Without loss of generality, we assume that $ \Ind(\Delta) \cap \llbracket 1,d \rrbracket=\llbracket 1, r\rrbracket$ and $ \Ind(\Delta) \cap \llbracket d+1,n \rrbracket = \llbracket d+1,n-r \rrbracket$.
For $i$ an integer and $U$ a matrix, we denote by $U_i$ the $i$-th row of $U$.
Thus, $\Mat({\Delta}) = \begin{pmatrix}
B_{1} & \cdots & B_r & B_{d+1} & \cdots & B_{{n-r}}
\end{pmatrix}^\intercal$ with $B_i \in \MM_{1,n-d}(\Q)$.
We define a matrix $\Sigma = \begin{pmatrix}
\Sigma_1 & \cdots & \Sigma_{n-d}
\end{pmatrix}^\intercal \in \mathcal{M}_{n-d,d}(\R )$ by setting:
\begin{align*}
\Sigma_i = \left\{ \begin{array}{cl}
G_{i-n+d+r} &\text{ if } i \in \llbracket n-d-r+1,n- d\rrbracket \\
0 &\text{ otherwise }
\end{array} \right. .
\end{align*}
We have defined $\Sigma$ so that the last $r$ rows of $\Sigma$ are equal to the first $r$ rows of $G$.
We now use (\ref{det_nul}):
\begin{align*}
0= \det \begin{pmatrix}
\begin{matrix}G \\ 
\Sigma
\end{matrix} & B
\end{pmatrix} = \det \begin{pmatrix}
\begin{matrix}G_1\\ \vdots \\ G_r \\ G_{r+1} \\ \vdots \\ G_d \\ 0 \\ \vdots \\ 0 \\ G_1 \\ \vdots \\ G_r 
\end{matrix} & 
\begin{matrix} B_1\\ \vdots \\ B_r \\ B_{r+1}\\ \vdots \\ B_d \\ B_{d+1} \\ \vdots \\ B_{n-r} \\ B_{n-r+1} \\ \vdots \\ B_n 
\end{matrix} 
\end{pmatrix} = \pm \det \begin{pmatrix}
\begin{matrix}G_1\\ \vdots \\ G_r \\ G_{r+1} \\ \vdots \\ G_d \\ G_1 \\ \vdots \\ G_r \\ 0 \\ \vdots \\ 0 
\end{matrix} & 
\begin{matrix} B_{n-r+1} \\ \vdots \\ B_n \\ B_{r+1} \\ \vdots \\ B_{d} \\ B_1\\ \vdots \\ B_r \\ B_{d+1} \\ \vdots \\ B_{n-r} 
\end{matrix} 
\end{pmatrix} 
\end{align*} 
by swapping rows.
We denote by $M$ the last matrix appearing in the above inequality. 
Using Laplace expansion (Theorem 1.8 of \cite{caldero_germoni}) with $J = \llbracket 1,d \rrbracket$ we have: 
\begin{align}\label{Laplace_GB}
0 = \sum\limits_{I \in P(d,n)} (-1)^{\ell(I) + \ell(J)} \Delta_{I,J}(M) \Delta_{\Bar{I},\Bar{J}}(M)
\end{align}
where $ \Bar{I} = \llbracket 1, n \rrbracket \setminus I$ and $\Delta_{I,J}(M)$ is the minor of $M$ associated with the rows indexed by $I$ and the columns indexed by $J$.
We then study the minors of $M$ according to the choice of $I \in P(d,n)$.
\begin{itemize}[label = $\bullet$]
\item If $I \cap \llbracket d + r +1, n \rrbracket \neq \emptyset $, then $\Delta_{I,J}(M) = 0 $ because it contains a null row. 
\item Otherwise $I \cap \llbracket d + r +1, n \rrbracket = \emptyset $, and we have $\#(I \cap \llbracket r +1, r +d \rrbracket) \geq d- r $. 
\begin{itemize}[label = $\diamond$]
\item If $\#(I \cap \llbracket r +1, r +d \rrbracket) = d- r $, then $I\cap \llbracket 1, r \rrbracket = \llbracket 1, r \rrbracket $. In this case, if $I \neq \llbracket 1, d\rrbracket$, then $\Delta_{I,J}(M)=0$ because there are two equal rows in this minor.
\item If $\#(I \cap \llbracket r +1, r +d \rrbracket) > d- r $, then $\# (\Bar{I} \cap \llbracket r +1, r +d \rrbracket) < r $. In this case, $\Delta_{\Bar{I},\Bar{J}}(M)$ is a minor of $B$ such that $\#(\Ind(\Delta_{\Bar{I},\Bar{J}}(M)) \cap \llbracket 1, d \rrbracket ) <r$. By the induction hypothesis, this minor is null.
\end{itemize}
\end{itemize}
Using (\ref{Laplace_GB}), the only possibly non-zero term is the one corresponding to $I = \llbracket 1,d \rrbracket$ and we have:
\begin{align*}
0 = \pm \det(G) \det(B_1,\ldots, B_r,B_{d+1}, \ldots, B_{n-r}) = \pm \det(G) \Delta.
\end{align*}
Since $\det(G) \neq 0 $ by assumption~\ref{condition_1}, we have $\Delta = 0$. We have thus shown that every minor of size $n-d$ of $Q$ is null. In particular, $\rank(B) < n-d $ which is contradictory since $Z_1, \ldots, Z_{n-d}$ form a basis of $B$.
\end{itemize}
\end{proof}

Proposition~\ref{prop_haut_appli} is a result concerning the height of a space by "decomposing" it over the image and kernel of a particular orthogonal projection. 

\begin{prop}\label{prop_haut_appli}
Let $n \in \N \setminus\{0 \}$ and $e \in \llbracket 1, n \rrbracket$. Let $ p: \R^n \longrightarrow \R^n$ be an orthogonal projection satisfying $p(\Z^n) \subset \Z^n$. Then, for any $B \in \RR_n(e)$, we have: 
 \begin{align*}
 H(B) = H(\ker(p) \cap B) \cdot H(p(B)).
 \end{align*}
\end{prop}

\begin{proof}
We work here in the space $B$; we consider the restriction of $p$ to $B$, still denoted as $p$. Let $t = \dim(p(B))$. Since $p(\Z^n) \subset \Z^n$, $p(B)$ is rational. As $\dim(B) = e$, we have $\dim(\ker(p)) = e - t$. Let $X_1, \ldots, X_{e-t}$ be a $\Zbasis$ of $\ker(p) \cap \Z^n$. Since $\ker(p) \cap \Z^n \subset B \cap \Z^n$, according to Corollary~3 of Theorem~1 of \cite{Cassels_geom}, there exist vectors $X_{e-t+1}, \ldots, X_e$ such that $X_1, \ldots, X_e $ form a $\Zbasis$ of $B \cap \Z^n$. We will show that we can replace $X_i$ by $p(X_i)$ for $i \in \llbracket e-t +1, e \rrbracket$. Indeed, we have:
\begin{align}\label{tilde_X_noyau}
\forall i \in \llbracket e-t +1, e \rrbracket, \, \widetilde{X_i} = X_i - p(X_i) \in \ker(p) \cap \Z^n
\end{align}
since $p$ is a projection and $X_i$ and $p(X_i)$ are integer vectors. For $i \in \llbracket e-t +1, e \rrbracket$, we can write:
\begin{align}\label{tilde_X_decom}
\widetilde{X_i} = \sum\limits_{k = 1}^{e-t} u_{i,k} X_k
\end{align}
with $u_{i,k} \in \Z$ for $k \in \llbracket 1, e-t \rrbracket$. All $\widetilde{X_i}$ are thus integer combinations of $X_1, \ldots, X_{e-t}$. In particular, the change of basis matrix from $(X_1, \ldots, X_d)$ to $(X_1, \ldots, X_{e-t}, p(X_{e-t+1}), \ldots, p(X_{e}))$ is of the form:
\begin{align*}
\begin{pmatrix}
I_{e-t} & * \\
0 & I_{t}
\end{pmatrix} \in M_e(\Z)
\end{align*}
according to \eqref{tilde_X_noyau} and \eqref{tilde_X_decom}. Since the determinant of this matrix is $1$, we deduce that the family $X_1, \ldots, X_{e-t}, p(X_{e-t+1}), \ldots, p(X_{e})$ is also a $\Zbasis$ of $B \cap \Z^n$.
Moreover, we notice that the family $p(X_{e-t+1}), \ldots, p(X_{e})$ is a basis of $p(B)$, so it is also a $\Zbasis$ of $p(B) \cap \Z^n$. We can now calculate the height, we have:
\begin{align*}
H(B) = \| X_1 \wedge \ldots \wedge X_{e-t} \wedge p(X_{e-t+1}) \wedge \ldots \wedge p(X_{e}) \| = \| X_1 \wedge \ldots \wedge X_{e-t} \| \cdot \| p(X_{e-t+1}) \wedge \ldots \wedge p(X_{e}) \|,
\end{align*}
the second inequality coming from the fact that $\ker(p)$ and $p(B)$ are orthogonal subspaces. We conclude by remarking that the quantities of right-hand side are the heights of the subspaces $\ker(p)$ and $p(B)$. We have $\| X_1 \wedge \ldots \wedge X_{e-t} \| = H(\ker(p)) $ and $\| p(X_{e-t+1}) \wedge \ldots \wedge p(X_{e}) \| = H(p(B))$ and that concludes the proof of the proposition.

\end{proof}

\begin{req}
 The condition $p(\Z^n) \subset \Z^n$ is actually very restrictive. Let $(e_i)_{i \in \llbracket 1, n \rrbracket}$ be the canonical basis of $\R^n$. Since $p$ is a projection, the subordinate norm of $p$ satisfies $\| p \| \leq 1$. 
Then for all $ i \in \llbracket 1,n \rrbracket $:
\begin{align*}
 \|p(e_i)\| \leq \|e_i\| =1.
\end{align*}
Now $p(e_i) \in \Z^n$, and thus $p(e_i) \in \{0, \pm e_1, \ldots,\pm e_n \}$. Hence the only possible orthogonal projections are those onto subspaces of the form $\underset{ i \in I}{\Span}(e_i) $ with $I \subset \llbracket 1, n \rrbracket$.
\end{req}

Lemma~\ref{2lem_sigma_alg_indep} elaborates on constructing irrational numbers reminiscent of Liouville number $ \sum\limits_{k = 1}^{+ \infty} \frac{1}{10^{k!}}$. The numbers constructed in Lemma~\ref{2lem_sigma_alg_indep} and their irrationality measure have been studied in \cite{Shallit}, \cite{Poorten_Shallit} and \cite{Bugeaud_approx_Cantor}. The reader can also consult \cite[Section $8$]{Levesley_Salp_Velani} and \cite{Bugeaud_trans}.

\begin{lem}\label{2lem_sigma_alg_indep}
 Let $e \in \Nx$, $\theta \in \N \setminus \{0,1\}$ and $\alpha = (\alpha_k)_{k \in \N} $ a sequence verifying there exists $\cons \label{2constante_suite_alpha} > 1$ such that:
 \begin{align*}
\forall k \in \N, \quad \alpha_{k+1} > c_{\ref{2constante_suite_alpha} } \alpha_k.
\end{align*}
 Let $J \subset \Nx $ be of cardinality at least $2$ and $\FF\subset \R$ be a finite subset. Let $\phi: \N \longrightarrow \llbracket 0,e \rrbracket $ be a function such that for all $ i \in \llbracket 0,e \rrbracket$ one has $\#\phi^{-1}(\{ i \} ) = \infty.$
Then there exist $e +1 $ sequences $(u_k^{i})_{ i \in \llbracket 0,e \rrbracket, k \in \N}$ such that: 
\begin{align}\label{2condition_suite_u}
 \text{for all } i \in \llbracket 0,e \rrbracket \text{ and } k \in \N, \quad u_k^{i} &\left\{ \begin{array}{lll}
 \in J &\text{ if } \phi(k) = i \\
 = 0 &\text{ else }
 \end{array} \right. , 
\end{align}
and with $\left( \sum\limits_{k = 0}^{+ \infty} \frac{u_k^i}{\theta^{\floor{\alpha_k}}} \right)_{ i \in \llbracket 0,e \rrbracket}$ algebraically independent over $\Q(\FF)$.
\end{lem}

\begin{proof}
 Let $\sum\limits_{k = 0}^{+ \infty} \frac{u_k^0}{\theta^{\floor{\alpha_k}}}, \ldots, \sum\limits_{k = 0}^{+ \infty} \frac{u_k^e}{\theta^{\floor{\alpha_k}}} $ be denoted by $\sigma_0, \ldots, \sigma_{e}$, and let us reason by induction on $t \in \llbracket 0,e \rrbracket$.

The set of algebraic numbers over $\Q(\FF)$ is countable since $\#\FF< + \infty$, and the set of sequences $(u_k^0)$ satisfying (\ref{2condition_suite_u}) is uncountable because $\#J \geq 2$. Therefore, we choose a sequence such that $\sigma_0$ is transcendental over $\Q(\FF)$.
Now, suppose that we have constructed $\sigma_0, \ldots, \sigma_t$ as an algebraically independent family over $\Q$ with $t \in\llbracket 0, e-1\rrbracket$. The set of algebraic numbers over $\Q(\FF,\sigma_0, \ldots, \sigma_t)$ is countable, but the set of sequences $(u_k^{t+1})_{k \in \N}$ satisfying (\ref{2condition_suite_u}) is uncountable because $\#J \geq 2$. Therefore, we can choose a sequence such that $\sigma_{t+1} $ is transcendental over $\Q(\FF,\sigma_0, \ldots, \sigma_t)$, completing the induction.

\end{proof}

Finally we prove a lemma that is used throughout this article to show that an integer vector belongs to a given rational subspace.

\begin{lem}\label{2lem_X_in_B}
 Let $Y \in \Z^n $ and $B$ be a rational subspace of dimension $e$. Let $ X_1, \ldots, X_e $ be a basis of $B$ with $X_i \in \Z^n $ for all $i \in \llbracket 1, e \rrbracket$. If $ \| Y \wedge X_1 \wedge \ldots \wedge X_e \| <1$, then $Y \in B$.
\end{lem}

\begin{proof}
 Denoting by $\| \cdot \|_{\infty}$ the supremum norm, we have $ \| Y \wedge X_1 \wedge \ldots \wedge X_e \|_{\infty} \leq \| Y \wedge X_1 \wedge \ldots \wedge X_e \| <1.$

 However, $ Y \wedge X_1 \wedge \ldots \wedge X_e $ has integer coordinates since the vectors under consideration do. Therefore, its infinity norm is zero, and thus $ Y \wedge X_1 \wedge \ldots \wedge X_e = 0.$
This implies that there is a linear dependency among the vectors $Y, X_1, \ldots, X_e$, and since the family of $X_i$ forms a basis of $B$, we have:
\begin{align*}
 Y \in \Span(X_1, \ldots, X_e) = B.
\end{align*}
\end{proof}

\section{Diophantine exponents of a direct sum of lines}\label{sect_sum}
In this section, we prove a result allowing one to calculate the Diophantine exponents of a subspace which is a sum of lines included in distinct rational subspaces. It corresponds to Chapter 4 of \cite{Guillot_these}, where the reader can find more detailed explanations.

We fix $n \in \N $ and we denote by $g(A,e)$ the quantity:
\begin{align*}
 g(A,e) = g(\dim(A), e,n) = \max(0, \dim(A) +e -n) .
\end{align*}
We define $P(k + g(A,e), d)$ as the set of all subsets with $k + g(A,e)$ elements of $\llbracket 1, d\rrbracket$. One has the following result.

\begin{theo}\label{2theo_somme_sev}
Let $d \in \left\llbracket 1, \floor{\frac{n}{2} } \right\rrbracket$. Assume that $ \bigoplus\limits_{j=1}^d R_j \subset \R^n $ with $R_j$ rational subspaces of dimension $r_j$.\\ Let $A = \bigoplus\limits_{j=1}^d A_j$ with $A_j \subset R_j$ and $\dim(A_j) = 1$. For $J \subset \llbracket 1,d \rrbracket $, we set $A_J = \bigoplus\limits_{j \in J} A_j$. 
\\ Let $e \in \llbracket 1, n-1 \rrbracket $ and $k \in \llbracket 1,\min(d,e) - g(A,e) \rrbracket$. The following statements are equivalent:
\begin{enumerate}[label = (\roman*)]
 \item $A$ is $(e,k)\tir$irrational.
 \item $\forall J \in P(k + g(A,e), d), \quad A_J $ is $(e,k+g(A,e) -g(A_J,e))\tir$irrational.
\end{enumerate}
Furthermore, in this case one has:
\begin{align*}
\mu_n(A|e)_k = \max\limits_{J \in P(k + g(A,e), d)} \mu_n({A}_{J}| e)_{k+g(A,e) -g(A_J,e)}.
\end{align*}
\end{theo}

Theorem~\ref{2theo_somme_sev} is a direct consequence of the following proposition, which is applied inductively. 

\begin{prop}\label{prop_princi}
Let $d, e, k, R_j, A_j$ be as in Theorem~\ref{2theo_somme_sev}. 
\\Let $J \subset \llbracket 1,d \rrbracket $ such that $\#J \geq k+ g(A_J,e) + 1$. Then the following statements are equivalent:
\begin{enumerate}[label = (\roman*)]
 \item $A_J$ is $(e,k)$-irrational. \label{4implic1}
 \item For all $j \in J $, $A_{J \setminus \{j\}}$ is $(e,k+g(A_J,e) -g(A_{J \setminus \{j\}},e))$-irrational. \label{4implic2}
\end{enumerate}

Moreover, in this case we have:
\begin{align*}
\mu_n(A_J|e)_k = \max_{j \in J} \mu_n({A}_{J \setminus \{j\}}| e)_{k+g(A_J,e) -g(A_{J \setminus \{j\}},e)} .
\end{align*}
\end{prop}

The rest of the section is devoted to the proof of Proposition~\ref{prop_princi}. From now on, we take $d, e, k, R_j,$ and $A_j$ as in Theorem~\ref{2theo_somme_sev}, and we consider $J \subset \llbracket 1,d \rrbracket$ such that $\#J \geq k + g(A_J,e) + 1$.

\subsection{First implication and orthogonal subspaces}

The implication $\ref{4implic1} \Longrightarrow~\ref{4implic2} $ is the easiest. Assume that the space $A_J$ is $(e,k)$-irrational, and let $j \in J$. Recall that $g(\dim(A_J),e,n) = g(A_J,e)$ and $g(\dim(A_{J \setminus \{j \} },e,n)) = g(\dim(A_{J \setminus \{j \} },e))$. 

\begin{claim}\label{2cor_crois_exponents_inclusion}
Let $A' \subset A$ be two subspaces of $\mathbb{R}^n$ of respective dimensions $d'$ and $d$. Then for $e$ and $j$ such that $A \in \II_n(d,e)_j$, we have:
\begin{align*}
A' \in \II_n(d',e)_{j+g(d,e,n)-g(d',e,n)} \text{ and } \mu_n(A|e)_j \geq \mu_n(A'|e)_{j+g(d,e,n)-g(d',e,n)}.
\end{align*}
\end{claim}

\begin{proof}
Let $e$ and $j$ be such that $A \in \II_n(d,e)_j$. Then for any rational subspace $B$ of dimension $e$, we have
\begin{align*}
\dim(A' \cap B) \leq \dim(A \cap B) < j + g(d,e,n) \leq j+g(d,e,n)-g(d',e,n) + g(d',e,n).
\end{align*}
Thus, we conclude that $A' \in \II_n(d',e)_{j+g(d,e,n)-g(d',e,n)}$. Let $\varepsilon > 0$; by definition of the Diophantine exponent, there exist infinitely many rational subspaces $B$ of dimension $e$ such that:
\begin{align*}
\omega_{j + g(d,e,n)}(A',B) = \psi_{j+g(d,e,n)-g(d',e,n)}(A',B) \leq H(B)^{-\mu_n(A'|e)_{j+g(d,e,n)-g(d',e,n)} + \varepsilon}.
\end{align*}
By applying the corollary of Lemma~12 of \cite{Schmidt} with $k = j + g(d,e,n)$, we have
\begin{align*}
\omega_{j + g(d,e,n)}(A,B) \leq \omega_{j + g(d,e,n)}(A',B) \leq H(B)^{-\mu_n(A'|e)_{j+g(d,e,n)-g(d',e,n)} + \varepsilon}.
\end{align*}
Since $\omega_{j + g(d,e,n)}(A,B) = \psi_j(A,B)$, we deduce that
$\forall \varepsilon >0, \quad \mu_n(A|e)_j \geq \mu_n(A'|e)_{j+g(d,e,n)-g(d',e,n)} - \varepsilon,$ which concludes the proof by letting $\varepsilon$ tend to $0$.

\end{proof}
Since $A_{J \setminus \{j\}} \subset A_J$, this claim implies that $A_{J \setminus \{j \}} \in \II_n(A_{J \setminus \{j \}},e)_{k+g(A_{J},e) - g((A_{J \setminus \{j \} },e)} $, which proves the first part of the lemma. Moreover, we have 
$ \mu_n(A_J|e)_k \geq \mu_n({A}_{J \setminus \{j\}}| e)_{k+g(A_J,e) - g(A_{J \setminus \{j \}}| e)}.$
Since this holds for any $j \in J$, we conclude that: 
\begin{align}\label{3premiere_ineg}
 \mu_n(A_J|e)_k \geq \max\limits_{j \in J}\mu_n({A}_{J \setminus \{j\}}| e)_{k+g(A_J,e) - g(A_{J \setminus \{j \}}| e)}.
 \end{align}

\subsection{Reduction to orthogonal rational subspaces}\label{section_reduction_ortho}
To prove the other implication $\ref{4implic2} \Longrightarrow~\ref{4implic1} $, we now reduce to the case where the $R_j$ are generated by vectors of the canonical basis. We introduce the subspaces $R'_j$ defined as follows:
\begin{align*}
&R'_j = \{0\}^{r_1 } \times \ldots \times \{0 \}^{r_{j-1}} \times {\mathbb{R}^{r_j}} \times \{0 \}^{r_{j+1}} \ldots \times \{0 \}^{r_d} \subset \mathbb{R}^n \text{ for } j \in \llbracket 1, d\rrbracket.
\end{align*}
For $j \in \llbracket 1, d\rrbracket$, let $\varphi_j $ be a rational isomorphism from $R_j$ to $R'_j$. We then choose $\varphi: \mathbb{R}^n \to \mathbb{R}^n$ a rational isomorphism such that:
\begin{align*}
\forall j \in \llbracket 1,d \rrbracket, \quad \varphi_{|R_j} = \varphi_j.
\end{align*}
According to Theorem $1.2$ of \cite{joseph_spectre} (which can be extended to the case $d +e > n$), we have $\phi(A) \in \II_n(d,e)_k$ if $A \in \II_n(d,e)_k$ and:
\begin{align*}
\mu_n(\varphi(A)|e)_k = \mu_n(A|e)_k 
\text{ and } \mu_n(\varphi(A_J)|e)_k = \mu_n(A_J|e)_k
\end{align*}
for all $J \subset \llbracket 1, d \rrbracket$, $e \in \llbracket 1, n-1 \rrbracket$, and $k$ such that the considered subspaces are $(e,k)$-irrational. Therefore, we can now assume that $R_j =R'_j $ for all $j \in \llbracket 1,d \rrbracket$. For $J \subset \llbracket1,d \rrbracket$, we denote by $R_J = \bigoplus\limits_{j \in J} R_j$, and this direct sum is orthogonal. We denote by $p_j$ the orthogonal projection onto $R_j$, and $\widehat{p_{j}}$ the orthogonal projection onto $R_{\llbracket 1, d \rrbracket \setminus \{ j\}}$.

\begin{lem}\label{petit_vecteurs}
Let $J$ be a non-empty subset of $\llbracket 1, d \rrbracket $ and $F$ be a subspace of $\R^n$ such that $\dim(F) < \#J $. Let $\cons \label{cons_lemme_angle}= \frac{ 1}{\sqrt{n^2 +1 }}$. Then there exists $j \in J$ such that, for all $X \in F $: 
\[ \| \widehat{p_j}(X) \| \geq c_{\ref{cons_lemme_angle}} \|X\|.\]
\end{lem}

This result is optimal in the sense that if $\dim(F) = \#J$, then it is false. The interesting fact here, is that the constant does not depend with $F$, otherwise it would be trivial (one of the $\widehat{p_j}$ is injective on $F$).

\begin{proof}
We assume the contrary by contradiction. Thus, there exists a family $(X_j)_{j \in J}$ of vectors in $F$ such that for every $j \in J $,
$ \|\widehat{p_j}(X_j) \| < c_{\ref{cons_lemme_angle}} \|X_j\|$. For every $j \in J$, $
 \| \widehat{p_j}(X_j) \|^2 < c_{\ref{cons_lemme_angle}}^2 \|X_j\|^2 
 = c_{\ref{cons_lemme_angle}}^2(\|p_j(X_j)\|^2 +\|\widehat{p_j}(X_j)\|^2) $
thus $c_{\ref{cons_lemme_angle}}^2\| p_j(X_j) \|^2 > (1-c_{\ref{cons_lemme_angle}}^2)\|\widehat{p_j}(X_j)\|^2$. We recall that $\frac{\| \cdot \|^2_1}{n} \leq \| \cdot \|^2 \leq n \| \cdot \|^2_\infty $. This gives:
\begin{align} \label{inegal_pj_pjchapeau}
 \| p_j(X_j) \|_\infty^2 &>\frac{1-c_{\ref{cons_lemme_angle}}^2}{n^2c_{\ref{cons_lemme_angle}}^2} \| \widehat{p_j}(X_j) \|_1^2 = \| \widehat{p_j}(X_j) \|_1^2
\end{align}
by definition of $c_{\ref{cons_lemme_angle}}$. Let $X_j = \begin{pmatrix} x_{1,j} & \cdots & x_{n,j} \end{pmatrix}^\intercal$ and $\ell_j$ an index such that $ \| p_j(X_j) \|_\infty = | x_{\ell_j,j} |$. We notice that the indices $\ell_j$ are distinct since the $R_i$ are in direct sum, generated by vectors of the canonical basis, and because $p_j(X_j) \in R_j$.

We now examine the family $(X_j)$. Let $M \in \MM_{n,\#J}(\R) $ be the matrix whose columns are these $X_j$ for $j \in J $. Finally, let $M_J $ be the square matrix of size $\#J$ extracted from $M$ whose rows correspond to the rows of $M$ indexed by the $\ell_j$. We have:
\begin{align*}
 M_J = \begin{pmatrix}
 \begin{matrix}
 x_{\ell_1,j_1} \\ \vdots \\ x_{\ell_{\#J},j_1}
 \end{matrix}
 & \cdots &
 \begin{matrix}
 x_{\ell_1,j_{\#J}} \\ \vdots \\ x_{\ell_{\#J},j_{\#J}}
 \end{matrix}
 \end{pmatrix}
\end{align*}
writing $J = \{j_1 < \ldots < j_{\#J} \}$. According to inequality (\ref{inegal_pj_pjchapeau}), for all $i \in \llbracket 1, \# J \rrbracket$ we have:
 $| x_{\ell_{i},j_i} | = \| p_{j_i}(X_{j_i}) \|_\infty > \| \widehat{p_{j_i}}(X_{j_i}) \|_1 \geq \left\| \sum\limits_{s \neq j_i }p_s(X_{j_i}) \right\|_1 \geq \sum\limits_{s \neq {j_i} } |x_{\ell_s,j_i} |.
$
Therefore $M_J$ is a strictly diagonally dominant matrix so it is invertible. In particular, $M$ has rank $\#J$ and the vectors $X_j$ form a linearly independent family in $F$, and thus $\dim(F) \geq \dim (\Span_{j \in J} (X_j)) = \#J$ which leads to a contradiction.

\end{proof}

\subsection{Study of angles Between projected subspaces}

The purpose of this section is to prove the following lemma, which allows us to "reduce" the dimension of the space $A_J$.

\begin{lem}\label{Rec_angles}
Let $J \subset \llbracket 1,d \rrbracket $ such that $\#J \geq k + g(A_J,e) + 1 $ and $C$ be a vector subspace of $\R^n$ of dimension $e$. Then there exists $j = j(C) \in J$ such that, denoting by $C_j = \widehat{p_j}(C)$, we have:
\begin{align*}
 &\dim(C_j) \geq k +g(A_J,e) \\
 \text{ and } &\omega_{k+g(A_J,e) }(A_{J\setminus \{j \}}, C_j) \leq c_{\ref{cons_lemme_angle1}}\omega_{k+g(A_J,e) }({A}_{J}, C) 
\end{align*}
where $\cons \label{cons_lemme_angle1} >0 $ is a constant depending only on $n$.
\end{lem}

In what follows, for a vector $x \in \R^n$ and a given subspace $V \subset \R^n$, we denote by $x^V$ the orthogonal projection of $x$ onto $V$. We also denote by $A_J^\perp$ as the orthogonal complement of $A_J$ in $\R^n$. With the definitions of $\widehat{p_j}$ and the decomposition of vectors along the subspaces $R_i$ (see paragraph~\ref{section_reduction_ortho}), we establish the following claim. 

\begin{claim}\label{lem_minoration_U_j}
Let $J \subset \llbracket 1,d \rrbracket$. For all $U \in \R^n$ and $ j \in J $, we have: 
 \begin{align*}
 \|U^{A_J^\perp}\| \geq \left\| \widehat{p_j}(U) ^{A_{J \setminus \{j\}} ^\perp}\right\|.
 \end{align*}
\end{claim}

Sure, here's the translation:

\begin{proof}
We introduce, for $A$ and $R$ two vector subspaces of $\mathbb{R}^n$, the following notation: $A^{\perp_R} = A^\perp \cap R.$
We have the relations:
\begin{align}
A_J^\perp = R_{\llbracket 1,d \rrbracket \setminus J} \oplus \overset{\perp}{\bigoplus\limits_{i \in J}} A_i^{\perp_{R_i}} \text{ and }
A_{J \setminus {j}}^\perp = R_{\llbracket 1,d \rrbracket \setminus {J} \cup {j}} \oplus \overset{\perp}{\bigoplus\limits_{i \in {J \setminus {j}}}} A_i^{\perp_{R_i}}. \label{4ortho_A_J}
\end{align}
Let $U\in C$. Since $p_j(U) \in R_j$, according to $(\ref{4ortho_A_J})$, we have $p_j(U)^{A_J^\perp} = p_j(U)^{A_j^{\perp_{R_j}}} \in R_j.$ Indeed, $R_j$ is orthogonal to all other components of $A_J^\perp$. Similarly, since $\widehat{p_j}(U) \in R_{\llbracket 1,d \rrbracket \setminus {j} }$, we have $\widehat{p_j}(U)^{A_i^{\perp_{R_j}}} =0 $, and thus $\widehat{p_j}(U)^{A_J^\perp} =\widehat{p_j}(U)^{R_{\llbracket 1,d \rrbracket \setminus J } } + \sum\limits_{i \in J \setminus {j}} \widehat{p_j}(U)^{A_i^{\perp_{R_i}} } \in R_{\llbracket 1,d \rrbracket \setminus {j}}.$

Moreover, as $\widehat{p_j}(U)^{R_{\llbracket 1,d \rrbracket \setminus J \cup {j}} } = \widehat{p_j}(U)^{R_{\llbracket 1,d \rrbracket \setminus J }} $, using $(\ref{4ortho_A_J})$, we have:
\begin{align}\label{4egalité_sans_j}
\widehat{p_j}(U)^{A_J^\perp} =\widehat{p_j}(U)^{R_{\llbracket 1,d \rrbracket \setminus J \cup {j } } } + \sum\limits_{i \in J \setminus {j}} \widehat{p_j}(U)^{A_i^{\perp_{R_i}} } = \widehat{p_j}(U) ^{A_{J \setminus {j}} ^\perp}.
\end{align}
The vectors $p_j(U) ^{A_J^\perp} \in R_j $ and $ \widehat{p_j}(U) ^{A_J^\perp} \in R_{\llbracket 1,d \rrbracket \setminus {j}} = R_j^\perp$ are therefore orthogonal. Hence, we can lower bound the norm of $U^{A_J^\perp} $ as follows $|U^{A_J^\perp}|^2 = | p_j(U) ^{A_J^\perp}+ \widehat{p_j}(U) ^{A_J^\perp} |^2 = | p_j(U) ^{A_J^\perp} |^2+ | \widehat{p_j}(U) ^{A_J^\perp} |^2 \geq | \widehat{p_j}(U) ^{A_J^\perp} |^2.$
Thus, the lemma is proven since we have already seen that $\widehat{p_j}(U) ^{A_J^\perp} = \widehat{p_j}(U) ^{A_{J \setminus {j}}^\perp} $ in $(\ref{4egalité_sans_j})$.

\end{proof}

\begin{proofe}[Lemma~\ref{Rec_angles}]
 Let $J \subset \llbracket 1,d \rrbracket$ such that $\#J \geq k + g(A_J,e) + 1$. We denote by $C'$ a vector subspace of $C$ of dimension $k + g(A_J,e)$ such that $\omega_{k + g(A_J,e)}(A_J,C) = \omega_{k + g(A_J,e)}(A_J,C').$
As $\dim(C') < \#J$, Lemma~\ref{petit_vecteurs} gives $j \in J$ such that for all $X \in C'$, $
 \| \widehat{p_j}(X) \| \geq c_{\ref{cons_lemme_angle}} \|X\|.$ 
We now study the space $C'_j = \widehat{p_j}(C')$. Considering any $U_j \in C'_j \setminus \{0 \}$, there exists $U \in C'$ such that $\widehat{p_j}(U) = U_j$. According to the definition of the angle between two vectors $
 \omega_1(\Span(U_j), A_{J \setminus \{j\}}) = \| U_j ^{A_{J \setminus \{j\}}^\perp} \| \cdot \| U_j \|^{-1}.$
We now use Claim~\ref{lem_minoration_U_j} which gives $\| U_j ^{A_{J \setminus \{j\}}^\perp} \| \leq \| U ^{A_{J}^\perp} \|$, and recall that $\| U_j \| \geq c_{\ref{cons_lemme_angle}} \| U \|$. Thus, we have:
\begin{align*}
 \omega_1(\Span(U_j), A_{J \setminus \{j\}}) \leq \frac{\| U ^{A_J^\perp} \| }{c_{\ref{cons_lemme_angle}}\| U \| } 
 = \frac{1}{c_{\ref{cons_lemme_angle}}}\omega_1(\Span(U),A_J) 
 \leq \frac{1}{c_{\ref{cons_lemme_angle}}}\omega_{k + g(A_J,e)}(C',A_J) \leq \frac{1}{c_{\ref{cons_lemme_angle}}} \omega_{k + g(A_J,e)}(C,A_J)
\end{align*}
since $U \in C'$. The last inequality comes from the fact that $\dim(C') = k + g(A_J,e) \leq \#J = \dim(A_J)$ and Lemma 2.3 of \cite{joseph_spectre}.
As $k + g(A_J,e) = \dim(C'_j) \leq \#J -1 = \dim(A_{J\setminus \{j \} })$, there exists $U_j \in C'_j \setminus \{0 \}$ such that $ \omega_1(\Span(U_j), A_{J \setminus \{j \} }) = \omega_{k+g(A_J,e)}(C'_j, A_{J\setminus\{j \}} ).$
Therefore:
\begin{align*}
 \omega_{k+g(A_J,e)}(C'_j, A_{J\setminus\{j \}} ) = \omega_1(\Span(U_j),A_{J \setminus \{j\}}) \leq \frac{1}{c_{\ref{cons_lemme_angle}}}\omega_{k+g(A_J,e)}(C,A_J).
\end{align*}
Since $C'_j \subset C_j$, we have $\omega_{k+g(A_J,e)}(C'_j, A_{J\setminus\{j \}} ) \geq \omega_{k+g(A_J,e)}(C_j, A_{J\setminus\{j \}} ) $ according to the corollary of Lemma 12 of \cite{Schmidt}. Hence $
 \omega_{k+g(A_J,e)}(C_j, A_{J\setminus\{j \}} ) \leq \frac{1}{c_{\ref{cons_lemme_angle}}}\omega_{k+g(A_J,e)}(C,A_J) $ and 
Lemma~\ref{Rec_angles} is thus proved with $c_{\ref{cons_lemme_angle1}} = c_{\ref{cons_lemme_angle}}^{-1}$.

\end{proofe}

\subsection{Second implication and conclusion}

We now have all the tools to prove the second implication $\ref{4implic2} \Longrightarrow~\ref{4implic1} $ and the result on exponents. Let $J \subset \llbracket 1,d \rrbracket$ such that $\# J \geq k + g(A_J,e) + 1$. We suppose here that for all $j \in J$, 
$A_{J \setminus \{j\}} \text{ is } (e,k+g(A_J,e ) -g(A_{J \setminus \{j\}},e ))\tir\text{irrational.} $

\bigskip
\begin{proofe}[$\ref{4implic2} \Longrightarrow~\ref{4implic1} $ ]
We argue by contraposition and assume that $A_J$ is not $(e,k)\tir$irrational. \\Then there exists $C$ a rational space of dimension $e$ such that $\dim(A_J \cap C) \geq k + g(A_J,e)$, which is equivalent to $\omega_{k+g(A_J,e)}(A_J,C) = 0$. \\According to Lemma~\ref{Rec_angles}, there exists $j \in J$ such that $
 \dim(C_j) \geq k +g(A_J,e)$ and $\omega_{k+g(A_J,e)}({A}_{J\setminus \{j \}}, C_j) \leq c_{\ref{cons_lemme_angle1}} \omega_{k+g(A_J,e)}({A}_{J}, C)$
with $C_j = \widehat{p_j}(C)$.
The space $C_j$ is rational and of dimension less than or equal to $e$. Thus, there exists $\widetilde{C} $ a rational space of dimension $e$ such that $C_j \subset \widetilde{C}$. Then, by the corollay of Lemma 12 of \cite{Schmidt}:
\begin{align*}
 0 \leq c_{\ref{cons_lemme_angle}} \omega_{k+g(A_J,e)}({A}_{J\setminus \{j \}}, \widetilde{C}) \leq c_{\ref{cons_lemme_angle}} \omega_{k+g(A_J,e)}({A}_{J\setminus \{j \}}, C_j) &\leq \omega_{k+g(A_J,e)}({A}_{J}, C) =0
\end{align*}
and therefore $ \omega_{k+g(A_J,e) -g(A_{J \setminus \{j\}},e) + g(A_{J \setminus \{j\}},e) }({A}_{J\setminus \{j \}}, \widetilde{C}) = \omega_{k+g(A_J,e)}({A}_{J\setminus \{j \}}, \widetilde{C})=0$
 meaning that $A_{J\setminus \{j\}}$ is not $(e,k+g(A_J,e) -g(A_{J \setminus \{j\}},e) )\tir$irrational.

\end{proofe}

To complete the proof of the Proposition~\ref{prop_princi}, it remains to show $\max\limits_{j \in J} \mu_n({A}_{J \setminus \{j\}}| e)_{k+g(A_J,e) - g(A_{J \setminus \{j \}},e )} \geq \mu_n(A_J|e)_k$ as the other inequality has been proven in $(\ref{3premiere_ineg})$.

Let $\varepsilon > 0 $. Let us set $\gamma = \mu_n(A_J|e)_k$, then there exist infinitely many rational subspaces $C$ of dimension $e$ such that:
\begin{align}\label{4min_psik_AJC}
 \psi_k(A_J,C) \leq H(C)^{-\gamma + \varepsilon}.
\end{align}
According to Lemma~\ref{Rec_angles}, for any such space $C$, there exists $j \in J$ such that $\dim(C_j) \geq k +g(A_J,e) $ and 
$ \omega_{k +g(A_J,e)}({A}_{J\setminus \{j \}}, C_j) \leq c_{\ref{cons_lemme_angle}}\omega_{k +g(A_J,e)}({A}_{J}, C) $
where $C_j = \widehat{p_j}(C)$ and $c_{\ref{cons_lemme_angle}} > 0$ depends only on $n$. Since the $j \in J$ are finite, there exists $j \in J$ such that for infinitely many rational subspaces $C$ we have these inequalities with $C_j$. As the possible dimensions for $C_j$ are finite, there exists $t \in \llbracket k + g(A_J,e), e \rrbracket$ such that for infinitely many rational subspaces $C$ of dimension $e$ verifying $(\ref{4min_psik_AJC})$, we have $\dim(C_j) = t $ and
\begin{align}\label{4inega_angle}
 \omega_{k +g(A_J,e)}({A}_{J\setminus \{j \}}, C_j) \leq c_{\ref{cons_lemme_angle}} \omega_{k +g(A_J,e)}({A}_{J}, C).
\end{align}
We denote by $\beta_{e,j} = \mu_n(A_{J\setminus \{j \}}|e)_{k + g(A_J,e)- g(A_{J\setminus \{j \}},e)} $ and $\beta_{t,j} = \mu_n(A_{J\setminus \{j\}}|t)_{k + g(A_J,e) - g(A_{J\setminus \{j \}},t)} $. 
Since $t \leq e$, we have $\beta_{e,j} \geq \beta_{t,j} .$ Using again the definition of the Diophantine exponent $\beta_{t,j} $, there exists a constant $\cons \label{4cons_gamma_t} > 0$ depending only on $A$ and $\varepsilon$ such that for any rational space $B$ of dimension $t$ we have: 
\begin{align}\label{4ineg_gamma_t}
 c_{\ref{4cons_gamma_t}}H(B) ^{-\beta_{t,j} - \varepsilon} \leq 
 \psi_{k + g(A_J,e) - g(A_{J\setminus \{j\}},t)}({A}_{J\setminus \{j\}}, B). 
\end{align}
Thus, for any space $C$ of dimension $e$ verifying $\dim(C_j) =t $ and $(\ref{4min_psik_AJC})$ we have:
\begin{align}\label{4inega_gamma_combine}
 c_{\ref{4cons_gamma_t}}H(C )^{-\beta_{e,j} - \varepsilon} \leq c_{\ref{4cons_gamma_t}}H(C_j)^{-\beta_{e,j} - \varepsilon} \leq c_{\ref{4cons_gamma_t}}H(C_j)^{-\beta_{t,j} - \varepsilon} \leq 
 \psi_{k + g(A_J,e) - g(A_{J\setminus \{j\}},t)}({A}_{J\setminus \{j\}}, C_j), 
\end{align}
using inequality (\ref{4ineg_gamma_t}) and according to Proposition~\ref{prop_haut_appli} which gives $ H(C) = H(C_j) H(\ker(\widehat{p_j}) \cap C) \geq H(C_j).$ 
\\Recalling that $ \psi_{k + g(A_J,e) - g(A_{J\setminus \{j\}},t)}({A}_{J\setminus \{j\}}, C_j) = \omega_{k +g(A_J,e)}({A}_{J\setminus \{j \}}, C_j)$ if $\dim(C_j) = t$, by combining inequalities (\ref{4min_psik_AJC}), (\ref{4inega_angle}), and (\ref{4inega_gamma_combine}) we obtain:
\begin{align*}
 c_{\ref{4cons_gamma_t}}c_{\ref{cons_lemme_angle}}^{-1} H(C )^{-\beta_{e,j} - \varepsilon} \leq \omega_{k +g(A_J,e)}(A_J, C) \leq \psi_{k }(A_J,C) \leq H(C)^{-\gamma + \varepsilon}
\end{align*}
and thus, by letting $H(C)$ tend to infinity, we obtain $\beta_{e,j} \geq \gamma - 2 \varepsilon$. Since this inequality holds for any $\varepsilon$, we have $\beta_{e,j} \geq \gamma $ and thus in particular: 
\begin{align*}
 \max\limits_{j \in J}\beta_{e,j} = \max\limits_{j \in J} \mu_n(A_{J\setminus \{j\}}|e)_{k + g(A_J,e) - g(A_{J\setminus \{j \}},e)} \geq \mu_n(A_J|e)_k.
\end{align*}
Proposition~\ref{prop_princi} therefore holds.

\section{Construction of a line with prescribed exponents}\label{sect_cas_d1}

 In this section, we prove the case $d = 1$ of Theorem~\ref{6theo_principal}. To do so, we prove the following proposition. It corresponds to Chapter 5 of \cite{Guillot_these}, where the reader can find more detailed explanations.

 \begin{prop}\label{5prop_technique}
 Let $(\gamma_k)_{k \in \Nx} \in \left[ 2 +\frac{ \sqrt{5}-1}{2}, + \infty \right) ^{\Nx}$ a periodic sequence of period $T \in \Nx$.\\ There exists a line $A$ in $\R^n$ such that $A \in \II_n(1,n-1)_1$ and:
 \begin{align*}
 \forall e \in \llbracket 1, n-1 \rrbracket, \quad \mu_n(A|e)_1 = \max\limits_{i \in \llbracket 0, T-1\rrbracket} \gamma_{i+1}\ldots\gamma_{i+e}.
 \end{align*}
\end{prop}
\bigskip

This proves Theorem~\ref{6theo_principal} in the case $d =1$ with $C_1 = 2 +\frac{ \sqrt{5}-1}{2}$. Indeed, given $(\gamma_1, \ldots, \gamma_{n-1}) \in \left[ 2 +\frac{ \sqrt{5}-1}{2}, + \infty \right) ^{n-1} $, we set $T =2n-2 $ and complete this sequence by
\begin{align*}
\forall j \in \llbracket n, 2n-2 \rrbracket, \quad \gamma_j = 2 +\frac{ \sqrt{5}-1}{2} 
\end{align*}
and periodically as follows: $\forall i \in \llbracket 1, 2n-2 \rrbracket, \forall k \in \N, \quad \gamma_{i + k(2n-2)} = \gamma_i$. The Proposition~\ref{5prop_technique} provides a line $A$ such that $A\in \II_n(1,e)_1$ for all $e \in \llbracket 1, n-1\rrbracket $ and 
$\mu_n(A|e)_1 = \max\limits_{i \in \llbracket 0, 2n-3\rrbracket} \gamma_{i+1}\ldots\gamma_{i+e} = \max\limits_{\ell \in \llbracket 0, n-1-e \rrbracket} \gamma_{1, \ell +1} \ldots \gamma_{i, \ell + e}.$
It is noteworthy that Theorem~\ref{6theo_principal} is also more precise than the case $d = 1$ of the Theorem~\ref{1theo_construction}: the hypothesis $ \min\limits_{\ell \in \llbracket 1, m \rrbracket}(\beta_{1,\ell})^{c_{\ref{7cons_petite_hyp_theoc2c1}}} > \max\limits_{\ell \in \llbracket 1, m \rrbracket}(\beta_{1,\ell})^{c_{\ref{7cons_petite_hyp_theoc2}}} $ is not required for $d = 1$ . 

\bigskip 
Moreover, using Proposition~\ref{5prop_technique} we can establish that the image of the joint spectrum of $\left( \mu_n(\cdot|1)_1, \ldots, \mu_n(\cdot|n-1)_1 \right) $ contains a set with non-empty interior. Consequently, it furnishes an alternative proof for Corollary~\ref{1cor_roy_indep}.


\subsection{Construction of the line \texorpdfstring{$A$}{} and height of \texorpdfstring{$B_{N,e}$}{}}\label{subsect_cons_d1}
In this section, we construct the line $A$ as stated in Proposition~\ref{5prop_technique}, along with vectors $X_N$ for $N \in \mathbb{N}$, which achieve the best approximations of $A$.

Let $(\gamma_k)_{k \in \Nx} \in \left[ 2 +\frac{ \sqrt{5}-1}{2}, + \infty \right) ^{\Nx}$ a periodic sequence of period $T \in \Nx$. We introduce the sequence $\alpha = (\alpha_k)_{k \in \mathbb{N}}$ defined as follows:
\begin{align*}
\alpha_0 = 1, \text{ and }
\forall k \in \mathbb{N}, \quad \alpha_{k+1} = \gamma_{k+1}\alpha_k.
\end{align*}

Let $\theta$ be a prime number greater than or equal to $5$, and $\phi: \mathbb{N} \rightarrow \llbracket 0, n-2 \rrbracket $ be defined as $\phi(k) = (k \mod (n-1)) \in \llbracket 0, n-2 \rrbracket$, 
where $k \mod (n-1)$ is the remainder of the division of $k$ by $n-1$. According to Lemma~\ref{2lem_sigma_alg_indep}, since $\gamma_k \geq 2 + \frac{\sqrt{5}-1}{2} > 1 $, there exist $n-1$ sequences $u^0,\ldots, u^{n-2}$ satisfying:
\begin{align}\label{5cons_suite_u}
\forall j \in \llbracket 0, n-2 \rrbracket, \forall k \in \mathbb{N},\quad u^j_k \left\{ \begin{array}{lll}
&\in \{1,2\} &\text{ if } \phi(k) = j \\
&= 0 &\text{ otherwise }
\end{array} \right.
\end{align}
such that $u^j_k \neq 0 $ if and only if $j = k \mod (n-1)$;
and such that the family $(\sigma_0, \ldots, \sigma_{n-2})$ is algebraically independent over $\mathbb{Q}$ with $
\forall j \in \llbracket 0, n-2 \rrbracket, \quad \sigma_j = \sum\limits_{k = 0}^{+ \infty} \frac{u^j_k}{\theta^{\lfloor\alpha_k\rfloor}}$.

We define $A = \Span(Y)$ where $
Y = \begin{pmatrix}
1 &
\sigma_0 &
\cdots &
\sigma_{n-2}
\end{pmatrix}^\intercal$. According to Lemma~\ref{3lem_1_irrat}, the space $A$ is $(e,1)$-irrational for all $e \in \llbracket 1, n-1 \rrbracket$.

To compute the Diophantine exponents associated to $A$, we define, for $N \in \mathbb{N}$, the truncated vector:
\begin{align}\label{5def_XN_ chap5}
X_N = \theta^{\lfloor\alpha_N\rfloor} \begin{pmatrix}
1 &
\sigma_{0,N} &
\cdots &
\sigma_{n-2,N}
\end{pmatrix}^\intercal = \theta^{\lfloor\alpha_N\rfloor} \begin{pmatrix}
1 &
 \sum\limits_{k = 0}^{N} \frac{u^0_k}{\theta^{\lfloor\alpha_k\rfloor}} &
\cdots &
 \sum\limits_{k = 0}^{N} \frac{u^{n-2}_k}{\theta^{\lfloor\alpha_k\rfloor}}
\end{pmatrix}^\intercal \in \mathbb{Z}^n
\end{align}
where we have defined $\sigma_{j,N} = \sum\limits_{k = 0}^{N} \frac{u^{j}_k}{\theta^{\lfloor\alpha_k\rfloor}} $ for $j \in \llbracket 0, n-2 \rrbracket$. We especially have 
$
\left| \sigma_j - \sigma_{j,N} \right| = \left| \sum\limits_{k = N+1}^{+ \infty} \frac{u^{j}_k}{\theta^{\lfloor\alpha_k\rfloor}} \right|
 \leq 2 \sum\limits_{k = \floor{\alpha_{N+1}} }^{+ \infty} \frac{1}{\theta^{k} } \leq \frac{2 \theta}{\theta - 1 } \frac{ 1} {\theta^{\lfloor\alpha_{N+1}\rfloor}}.$
Hence, we have 
\begin{align}\label{5sigma_moins_sigmaN}
 \left| \sigma_j - \sigma_{j,N} \right| \leq\frac{ c_{\ref{5cons_sigma_moins_sigmaN}}} {\theta^{{\alpha_{N+1}}}}
\end{align} with $\cons = \frac{2\theta^2}{\theta -1} \label{5cons_sigma_moins_sigmaN}$.\
We have $\theta^{- \lfloor\alpha_N\rfloor} X_N\underset{N\to+\infty}{\longrightarrow} Y $ and thus 
 there exist constants $\cons \label{5cons_minor_norme_XN} > 0 $ and $\cons \label{5cons_major_norme_XN} > 0 $ independent of $N$ such that:
 \begin{align}\label{5lem_norme_XN}
 \forall N \in \mathbb{N}, \quad c_{\ref{5cons_minor_norme_XN}} \theta^{{\alpha_N}} \leq \| X_N \| \leq c_{\ref{5cons_major_norme_XN}} \theta^{{\alpha_N}}.
 \end{align}
Moreover, the definition of the angle $\omega(X,Y)$ 
and $(\ref{5sigma_moins_sigmaN})$ give
\begin{align}\label{5omega_XN_Y}
 \omega(Y, X_N) \leq \frac{ \| Y - X_N \|}{\| Y \|} \leq c_{\ref{5cons_omega_XN_Y}} \theta^{-\alpha_{N+1}}
\end{align}
with $\cons \label{5cons_omega_XN_Y}$ independent of $N$. 

The proof of Proposition~\ref{5prop_technique} unfolds by showin that the subspaces of best approximation of $A$ are those spanned by these vectors $X_N$. For this purpose, we define, for $N \in \mathbb{N}$ and $e \in \llbracket 1, n-1 \rrbracket$ the rational subspace 
$B_{N,e} = \Span (X_N, \ldots, X_{N+e-1}).$

\bigskip
We focus on studying the space $B_{N,e} = \Span(X_N, \ldots, X_{N+e-1})$, especially on providing a $\mathbb{Z}$-basis of $B_{N,e} \cap \mathbb{Z}^n$ and computing the height $H(B_{N,e})$. For $N \in \mathbb{N}$, it is noted that:
\begin{align}\label{5rec_X_N}
 X_{N+1} = \theta^{\lfloor\alpha_{N+1}\rfloor -\lfloor\alpha_N\rfloor} X_N + w_{N+1}
\end{align}
where $w_{N+1} = \begin{pmatrix}
 0 & u^0_{N+1} & \cdots & u^{n-2}_{N+1}
\end{pmatrix}^\intercal \in \mathbb{Z}^n$.

According to the construction of the sequences $u^j$ in $(\ref{5cons_suite_u})$, the vector $w_{N+1}$ has exactly one non-zero coordinate, which is $u^j_{N+1}$ for $j = \phi(N+1)$. We define the following vector:
\begin{align}\label{eq_vNdef}
 v_{N+1} = \frac{1}{u^{\phi(N+1)}_{N+1}}w_{N+1}.
\end{align}
We have defined vectors $v_{N+1}, \ldots, v_{N+e-1}$, and it is observed that these vectors are pairwise distinct and belong to the canonical basis; moreover, the first vector $(1,0, \ldots, 0)^\intercal$ of this basis is not included. Furthermore, we have 
$ B_{N,e} = \Span(X_N, v_{N+1}, \ldots, v_{N+e-1})
$
by using $(\ref{5rec_X_N})$ and the definition of $B_{N,e}$. Consequently, $\dim(B_{N,e}) = e$ since the first coordinate of $X_N$ is non-zero.

\begin{claim}\label{5B_N_base_v}
 For $N \in \mathbb{N}$, the set $X_N, v_{N+1}, \ldots, v_{N+e-1}$ forms a $\mathbb{Z}$-basis of $B_{N,e} \cap \mathbb{Z}^n$.
\end{claim}

\begin{proof}
 To prove the claim, it remains to show that if $X = aX_N + b_1v_{N+1} + \ldots + b_{e-1}v_{N+e-1} \in \mathbb{Z}^n $ with $(a, b_1, \ldots,b_{e-1}) \in \R^e$ then $(a, b_1, \ldots,b_{e-1}) \in \mathbb{Z}^e$. The first coordinate of $X$ is $a\theta^{\lfloor\alpha_N\rfloor}$, thus $ a\theta^{\lfloor\alpha_N\rfloor} \in \mathbb{Z}.$
 Now, examining the $(\phi(N)+1)$-th coordinate of $X$, we have: 
 \begin{align}\label{5a_theta_sigma}
 a\theta^{\lfloor\alpha_N\rfloor} \sigma_{\phi(N),N} + \sum\limits_{j = 1}^{e-1} b_j \frac{u^{\phi(N)}_{N+j}}{u^{\phi(N+j)}_{N+j} } = a\theta^{\lfloor\alpha_N\rfloor} \sigma_{\phi(N),N} \in \mathbb{Z}
\end{align}
since $u^{\phi(N)}_N \in \{ 1,2\}$, and for all $j \in \llbracket 1, n-2 \rrbracket$, $u^{\phi(N)}_{N+j} = 0$. Now $ \sigma_{\phi(N),N} = U_{\phi(N)}\theta^{-\lfloor\alpha_N\rfloor} $ with $\gcd(\theta, U_{\phi(N)}) = 1 $ because $u^{\phi(N)}_N \in \{ 1,2\}$ and $\theta \geq 5$. By Bézout's theorem, there exist integers $p_1$ and $p_2$ such that $p_1 \theta^{\lfloor\alpha_N\rfloor} + p_2 U_{\phi(N)} = 1$. Then,
\begin{align*}
 a = a (p_1 \theta^{\lfloor\alpha_N \rfloor}+ p_2 U_{\phi(N)}) = p_1 a \theta^{\lfloor\alpha_N\rfloor }+ p_2 a U_{\phi(N)} \in \mathbb{Z}
\end{align*}
since $a\theta^{\lfloor\alpha_N\rfloor} \in \mathbb{Z}$ and $aU_{\phi(N)} = a \theta^{\lfloor\alpha_N\rfloor} \sigma_{\phi(N),N} \in \mathbb{Z}$ by $(\ref{5a_theta_sigma})$. 
In particular, $X - a X_N = b_1v_{N+1} + \ldots + b_{e-1}v_{N+e-1} \in \mathbb{Z}$. Since the $v_{N+j}$ are distinct vectors of the canonical basis, $b_j \in \mathbb{Z}$ for all $j \in \llbracket 1,e-1 \rrbracket$.

\end{proof}

Now, we proceed to calculate the height of the space $B_{N,e}$.
\begin{lem}\label{5lem_haut_BN}
 There exist constants $\cons \label{5cons_minor_haut_BN} > 0$ and $\cons \label{5cons_major_haut_BN} > 0$ independent of $N$ such that:
 \begin{align*}
 \forall N \in \mathbb{N}, \quad c_{\ref{5cons_minor_haut_BN}} \theta^{{\alpha_N}} \leq H(B_{N,e}) \leq c_{\ref{5cons_major_haut_BN}} \theta^{{\alpha_N}}.
 \end{align*}
\end{lem}

\begin{proof}
 According to the Claim~\ref{5B_N_base_v}, the vectors $X_N, v_{N+1}, \ldots, v_{N+e-1}$ form a $\mathbb{Z}$-basis of $B_{N,e}\cap \mathbb{Z}^n$ ; thus,
$ H(B_{N,e}) = \| X_N \wedge v_{N+1} \wedge \ldots \wedge v_{N+e-1} \|
$. Hence:
\begin{align*}
 H(B_{N,e}) \leq \| X_N \| \| v_{N+1} \| \ldots \| v_{N+e-1} \| \leq \| X_N \|
\end{align*}
and the upper bound of the lemma is obtained by taking $c_{\ref{5cons_major_haut_BN}} =c_{\ref{5cons_major_norme_XN}} $, the constant from the relation~(\ref{5lem_norme_XN}).

On the other hand, we know that the $ v_{N+j} $ for $ j \in \llbracket 1, e-1 \rrbracket $ form a linearly independent set. Moreover, $ Y \notin \Span(v_{N+1}, \ldots, v_{N+e-1}) $ because the first coordinate of $ Y $ is $1$ while those of $ v_{N+j} $ are zero. Therefore, $ \|Y \wedge v_{N+1} \wedge \ldots \wedge v_{N+e-1} \| > 0 $. Additionally, the $ (e-1) $-tuple $ (v_{N+1}, \ldots, v_{N+e-1}) $ can only take a finite number of values as $ N $ varies. Thus, there exists a constant $ \cons \label{5cons_v_non_col} > 0 $ independent of $ N $ such that $ \|Y \wedge v_{N+1} \wedge \ldots \wedge v_{N+e-1} \| \geq c_{\ref{5cons_v_non_col}} $ for all $N$.
Since $ \theta^{-\lfloor\alpha_N\rfloor}X_N \underset{N \to \infty}{\longrightarrow} Y $ , there exists $ N_0 \in N $ such that for any $N \geq N_0$:
\begin{align*}
 H(B_{N,e}) \geq \frac{c_{\ref{5cons_v_non_col}}}{2}\theta^{\lfloor\alpha_N\rfloor}. 
\end{align*}
Let $ c_{\ref{5cons_minor_haut_BN}} = \min\left(\min\limits_{0 \leq N <N_0} (\theta^{-{\alpha_N}}H(B_{N,e}) ) , \frac{c_{\ref{5cons_v_non_col}}}{2\theta}\right) > 0 $. 
Then, for all $ N \in \mathbb{N}, $
$ H(B_{N,e}) \geq c_{\ref{5cons_minor_haut_BN}} \theta^{{\alpha_N}} 
$
and the lower bound of the lemma is proved.

\end{proof}
\subsection{First angle between \texorpdfstring{$A$}{} and \texorpdfstring{$B_{N,e}$}{} }\label{subsect_angle_d1}
In this section we estimate the first angle between the line $A$ and the rational subspaces $B_{N,e}$. We prove the following lemma.

\begin{lem}\label{5lem_angle_entre_A_BN}
 
One has $$c_{\ref{5cons_minor_prox_A_BN}} \theta^{-{\alpha_{N+e}}} \leq \psi_1(A,B_{N,e}) \leq c_{\ref{5cons_major_prox_A_BN}} \theta^{-\alpha_{N+e}}$$
with $ \cons \label{5cons_minor_prox_A_BN} $ and $\cons \label{5cons_major_prox_A_BN}$ independent of $N$.

\end{lem}
The difficult part is bounding the angle from below. To prove this, we study the minors of size $2$ of a specific matrix and show that at least one of them is not too small.

\begin{proof}
 Let $N \in \mathbb{N}$. Since $X_{N+e-1} \in B_{N,e}$ and $Y \in A$, we have $ \psi_1(A,B_{N,e}) \leq \omega(Y, X_{N+e-1}) \leq c_{\ref{5cons_omega_XN_Y}} \theta^{-\alpha_{N+e}} $ according to (\ref{5omega_XN_Y}). Only the lower bound remains to be proven.

 Using (\ref{5rec_X_N}) and \eqref{eq_vNdef}, we have $B_{N,e} = \Span (X_{N+e-1}, v_{N+1}, \ldots, v_{N+e-1}).$
Let $X \in B_{N,e} \setminus \{0\}$. We will show that $ \omega(X,Y) \geq c_{\ref{5cons_minor_prox_A_BN}} \frac{1}{\theta^{\lfloor\alpha_{N+e}\rfloor}}$
and hence, since $\psi_1(A,B_{N,e}) = \min\limits_{X \in B_{N,e} \setminus \{0\}} \omega(X,Y)$, this will prove the lemma.

We express $X$ in the basis $X_{N+e-1}, v_{N+1}, \ldots, v_{N+e-1}$:
\begin{align*}
 X = a Z_{N+e-1} + b_1 v_{N+1} + \ldots + b_{e-1} v_{N+e-1},
\end{align*}
where $Z_{N+e-1} = \theta^{-\lfloor\alpha_{N+e-1}\rfloor} X_{N+e-1}$. By normalizing $X$, we can assume $a^2 + \sum\limits_{j = 1}^{e-1} b_j^2 = 1.$
This assumption gives, in particular, $\|X\| \leq c_{\ref{5cons_maj_X}}$ with $\cons \label{5cons_maj_X}$ independent of $N$ since the vectors $v_{N+j}$ have norm $1$ and $ \| Z_{N+e-1} \| \leq c_{\ref{5cons_major_norme_XN}} \theta $ according to $(\ref{5lem_norme_XN})$. We now seek to bound the quantity $\|X \wedge Y\|$ from below. We recall that we have:
\begin{align*}
 X = \begin{pmatrix}
 a &
 a\sigma_{0,N+e-1} + \sum\limits_{j = 1}^{e-1} b_ju^0_{N+j} &
 \cdots &
 a\sigma_{n-2,N+e-1} + \sum\limits_{j = 1}^{e-1} b_ju^{n-2}_{N+j} 
 \end{pmatrix}^\intercal
 \text{ and } Y = \begin{pmatrix}
 1 &
 \sigma_{0} &
 \cdots &
 \sigma_{n-2}
 \end{pmatrix}^\intercal,
\end{align*}
and that we can bound $\|X \wedge Y\|$ from below by the absolute value of any $2 \times 2$ minor of the matrix $(X|Y)$. Thus, we have:
\begin{align*}
 \forall i \in \{0, \ldots, n-2\}, \quad \|X \wedge Y\| \geq \left| \det\begin{pmatrix} a & 1 \\
 a\sigma_{i,N+e-1} + \sum\limits_{j = 1}^{e-1} b_ju^i_{N+j} & \sigma_i \end{pmatrix} \right| = \left| a (\sigma_i - \sigma_{i,N+e-1}) - \sum\limits_{j = 1}^{e-1} b_ju^i_{N+j} \right|.
\end{align*}
Since $\sigma_i - \sigma_{i,N+e-1} = \sum\limits_{k = N+e}^{+\infty} \frac{u^i_k}{\theta^{\lfloor\alpha_k\rfloor}}$, we finally obtain:
\begin{align}\label{5minor_prod_ext_X_Y}
 \forall i \in \{0, \ldots, n-2\}, \quad \|X \wedge Y\| \geq \left| a \sum\limits_{k = N+e}^{+\infty} \frac{u^i_k}{\theta^{\lfloor\alpha_k\rfloor}} - \sum\limits_{j = 1}^{e-1} b_ju^i_{N+j} \right|.
\end{align}

We now distinguish between two cases.

\textbullet \: \underline{Case 1:} For all $j \in \{1, \ldots, e-1\}$, we have $|b_j| < \frac{1}{\theta^{\lfloor\alpha_{N+e}\rfloor}}$. Then we have:
 $a^2 = 1 - \sum\limits_{j = 1}^{e-1} b_i^2 \geq 1 - \frac{e-1}{\theta^{2e}}$
since $\lfloor\alpha_{N+e}\rfloor \geq e$. We set $\cons \label{5cons_cas<} = \sqrt{1 - \frac{e-1}{\theta^{2e}}} > 0$, hence $|a| \geq c_{\ref{5cons_cas<}}$. Let $i = \phi(N+e) \in \{0, \ldots, n-2\}$. Then we have
\begin{align*}
 u^i_{N+e} \geq 1 \text{ and } u^i_{N+1} = \ldots = u^i_{N+e-1} =0.
\end{align*}
We use $(\ref{5minor_prod_ext_X_Y})$ with this $i$ and find:
\begin{align*}
 \|X \wedge Y\| \geq \left| a \sum\limits_{k = N+e}^{+\infty} \frac{u^i_k}{\theta^{\lfloor\alpha_k\rfloor}} \right| \geq |a| \left| \frac{u^i_{N+e}}{\theta^{\lfloor\alpha_{N+e}\rfloor}}\right| \geq c_{\ref{5cons_cas<}} \frac{1}{\theta^{\lfloor\alpha_{N+e}\rfloor}}.
\end{align*}

\textbullet \: \underline{Case 2:} There exists $j_0 \in \{1,e -1\}$ such that $|b_{j_0}| \geq \frac{1}{\theta^{\lfloor\alpha_{N+e}\rfloor}}$. Let $i = \phi(N+j_0)$. Then we have
\begin{align*}
 u^{i}_{N+j_0} \geq 1 \text{ and } u^i_{N+1} = \ldots =u^i_{N+j_0-1} =u^i_{N+j_0+1} = \ldots = u^i_{N+e-1} = u^i_{N+e} = 0.
\end{align*}
By applying $(\ref{5minor_prod_ext_X_Y})$ with $i$, we find:
\begin{align*}
 \|X \wedge Y\| &\geq \left| a \sum\limits_{k = N+e+1}^{+\infty} \frac{u^i_k}{\theta^{\lfloor\alpha_k\rfloor}} - b_{j_0} u^i_{N+j_0}\right| 
 \geq |b_{j_0}| - |a| \sum\limits_{k = N+e+1}^{+\infty} \frac{u^i_k}{\theta^{\lfloor\alpha_k\rfloor}} 
 \geq \frac{1}{\theta^{\lfloor\alpha_{N+e}\rfloor}} - \frac{1}{\theta^{\lfloor\alpha_{N+e+1}\rfloor}} \frac{2\theta}{\theta-1}
\end{align*}
since $|a| \leq 1$. This yields, since $\theta \geq 5$:
$ \| X \wedge Y \| \geq \frac{1}{\theta^{\floor{\alpha_{N+e}}}} - \frac{1}{2\theta^{\floor{\alpha_{N+e}} }} = \frac{1}{2\theta^{\floor{\alpha_{N+e}} }}. $

\medskip
Now let $c_{\ref{5cons_minor_prox_A_BN}} = \frac{\min(c_{\ref{5cons_cas<}}, \frac{1}{2})}{\|Y\|c_{\ref{5cons_maj_X}}}$. Since $\|X\| \leq c_{\ref{5cons_maj_X}}$ and according to the two cases we have studied, we have:
\begin{align*}
 \omega(X,Y) = \frac{\| X \wedge Y\|}{\| X\| \cdot\|Y\|} \geq \frac{\min(c_{\ref{5cons_cas<}}, \frac{1}{2})}{c_{\ref{5cons_maj_X}}\cdot \|Y\|}\frac{1}{\theta^{\lfloor\alpha_{N+e}\rfloor}} = c_{\ref{5cons_minor_prox_A_BN}} \frac{1}{\theta^{\lfloor\alpha_{N+e}\rfloor}}
\end{align*}
which completes the proof of the lemma.
\end{proof}

Since $\frac{\alpha_{N+e}}{\alpha_N} = \gamma_{N+1}\ldots\gamma_{N+e}$ , Lemma~\ref{5lem_haut_BN} and Lemma~\ref{5lem_angle_entre_A_BN} give
\begin{align}\label{5relation_angle_hauteur}
 c_{\ref{5cons_minor_prox_A_BN2}} H(B_{N,e})^{-\gamma_{N+1}\ldots\gamma_{N+e}} \leq \psi_1(A,B_{N,e}) \leq c_{\ref{5cons_major_prox_A_BN2}} H(B_{N,e})^{-\gamma_{N+1}\ldots\gamma_{N+e}},
\end{align}
with $\cons \label{5cons_minor_prox_A_BN2}$ and $\cons \label{5cons_major_prox_A_BN2}$ independent of $N$. We recall that the sequence $(\gamma_k)$ is periodic with period $T$. By selecting $N$ such that $\gamma_{N+1}\ldots\gamma_{N+e} = \max\limits_{i \in \llbracket 0, T-1\rrbracket} \gamma_{i+1}\ldots\gamma_{i+e}$, one has:
\begin{align}\label{inegal_frombelow}
 \mu_n(A|e)_1 \geq \max\limits_{i \in \llbracket 0, T-1\rrbracket} \gamma_{i+1}\ldots\gamma_{i+e}.
\end{align}
Now, we will show that the $B_{N,e}$ are the best approximations of $A$; this is the subject of the next section.

\subsection{Best approximations of \texorpdfstring{$A$}{}}\label{subsect_bestapprox_d1}
In order to prove that the subspaces $B_{N,e}$ realize the best approximations of $A$, we show the following lemma. Let us denote by $K_e$ the quantity $ \max\limits_{i \in \llbracket 0, T-1\rrbracket} \gamma_{i+1}\ldots\gamma_{i+e}$.

\begin{lem}\label{5lem_meilleurs_espaces}
Let $\varepsilon > 0$ and $B$ be a rational subspace of dimension $e$ such that:
\begin{align}\label{5hyp_lem_meilleurs_espaces}
 \psi_1(A,B) \leq H(B)^{-K_e - \varepsilon}.
\end{align}
Then, if $H(B)$ is large enough relative to $\varepsilon$ and $A$, there exists $N \in \mathbb{N}$ such that $B = B_{N,e}$.
\end{lem}

\begin{proof}
Let $N \in \mathbb{N}$ be the integer satisfying:
\begin{align}\label{5choix_du_N}
 \theta^{{\alpha_{N+e-1}}} \leq H(B)^{K_e + \frac{\varepsilon}{2} - 1} < \theta^{{\alpha_{N+e}}}.
\end{align}
We then show that this $N$ is suitable if $H(B)$ is large enough. Let $Z_1, \ldots, Z_e$ be a $\Zbasis$ of $B \cap \Z^n$. We will show that if $H(B)$ is sufficiently large, then for all $i \in \llbracket 0, e-1 \rrbracket$, the quantity $ \DD_{N+i} = \| X_{N+i} \wedge Z_1 \ldots \wedge Z_e \|$ vanishes. Given a vector $ x \in \mathbb{R}^n $ and a subspace $ V \subset \mathbb{R}^n $, we denote by $ x^V $ the orthogonal projection of $ x $ onto $ V $. One has $
 \DD_{N+i} = \| X_{N+i}^{B^\perp} \| \| Z_1 \ldots \wedge Z_e \| = \omega_1(\Span(X_{N+i}), B) \| X_{N+i} \| H(B).$ 
Now, $ \omega_1(\Span(X_{N+i}), B) \leq \omega(X_{N+i}, Y) + \omega_1(A, B)$ by the triangle inequality (see \cite[section 8]{Schmidt}). Also, $\| X_{N+i} \| \leq c_{\ref{5cons_major_norme_XN}} \theta^{{\alpha_{N+i}}} $ according to (\ref{5lem_norme_XN}), and $ \omega_1(A,B) = \psi_1(A, B)$. Thus:
\begin{align*}
 \DD_{N+i} &\leq c_{\ref{5cons_major_norme_XN}} \theta^{{\alpha_{N+i}}} \left( \omega(X_{N+i}, Y) + \psi_1(A, B) \right) H(B).
\end{align*}
Using (\ref{5omega_XN_Y}) which gives $ \omega(X_{N+i}, Y) \leq c_{\ref{5cons_omega_XN_Y}}\theta^{-{\alpha_{N+i+1}}} $ and the hypothesis (\ref{5hyp_lem_meilleurs_espaces}), we obtain:
\begin{align*}
 \DD_{N+i} \leq c_{\ref{5cons_major_norme_XN}}\theta^{{\alpha_{N+i}}} \left( c_{\ref{5cons_omega_XN_Y}} \theta^{-{\alpha_{N+i+1}}} + H(B)^{-K_e - \varepsilon} \right) H(B) \leq c_{\ref{cons_maj_DNi}} \left( \theta^{-{\alpha_{N+i+1}}+{\alpha_{N+i}}} H(B) + \theta^{{\alpha_{N+i}}}H(B)^{-K_e - \varepsilon + 1} \right)
\end{align*}
with $\cons \label{cons_maj_DNi} >0$ independent of $B$. According to the choice of $N$ in (\ref{5choix_du_N}), we have $\theta^{{\alpha_{N+i}}} \leq \theta^{{\alpha_{N+e-1}}} \leq H(B)^{K_e+ \frac{\varepsilon}{2} -1 }$ and $ \theta \geq H(B)^{\frac{K_e+ \frac{\varepsilon}{2} -1 }{{\alpha_{N+e}}} } $. Therefore, for all $i \in \llbracket 0,e-1 \rrbracket$:
\begin{align}\label{5inegal_D_N_i}
 \DD_{N+i} \leq c_{\ref{cons_maj_DNi}}\left( H(B)^{\tau_i } + H(B)^{- \frac{\varepsilon}{2} } \right)
\end{align}
where we set $\tau_i = 1 + \frac{(-{\alpha_{N+i+1}}+{\alpha_{N+i}})(K_e + \frac{\varepsilon}{2} -1) }{{\alpha_{N+e}}}$. We first note that $-{\alpha_{N+i+1}}+{\alpha_{N+i}}$ is maximal for $i = 0$. Indeed, if $i \geq 1$:
\begin{align*}
 {\alpha_{N+1}}+{\alpha_{N+i}} \leq 2\alpha_{N+i} \leq \alpha_{N+i+1} \leq {\alpha_{N+i+1}} +{\alpha_{N}}
\end{align*}
using the fact that $\gamma_{N+i+1} \geq 2 $ and $\alpha_N \geq 0$.
We thus have $\tau_i \leq \tau_0$ for all $ i \in \llbracket 0, e-1 \rrbracket$, and we will show that $\beta_0 \leq -c_{\ref{cons_epsilon_negatif}} \varepsilon$ with a certain constant $c_{\ref{cons_epsilon_negatif}} > 0 $ independent of $B$ that remains to be defined. We have:
\begin{align*}
 \tau_0 = 1 + \frac{(-{\alpha_{N+1}}+{\alpha_{N}})(K_e + \frac{\varepsilon}{2} -1) }{{\alpha_{N+e}}} = 1 + \frac{(-\gamma_{N+1} + 1)(K_e + \frac{\varepsilon}{2} -1) }{{\gamma_{N+1}\ldots\gamma_{N+e}}};
\end{align*}
and finally:
\begin{align}\label{5maj_separe_Ke_eps}
 \tau_0 &= \frac{\gamma_{N+1}\ldots\gamma_{N+e} + (-\gamma_{N+1} + 1)(K_e + \frac{\varepsilon}{2} -1) }{{\gamma_{N+1}\ldots\gamma_{N+e}}} \nonumber \\
 &\leq \frac{ (-\gamma_{N+1}+ 1) (K_e - 1 ) + \gamma_{N+1}\ldots\gamma_{N+e} }{\gamma_{N+1}\ldots\gamma_{N+e}} - c_{\ref{cons_epsilon_negatif}} \varepsilon
\end{align}
where $\cons \label{cons_epsilon_negatif} = \frac{1+\sqrt{5}}{4K_e} 
 \leq \frac{(\gamma_{N+1} - 1)}{2\gamma_{N+1}\ldots\gamma_{N+e}}$ is independent of $B$. We now bound the other term:
 \begin{align*}
 (-\gamma_{N+1}+ 1) (K_e - 1 ) + \gamma_{N+1}\ldots\gamma_{N+e} \leq (-\gamma_{N+1}+ 1) (K_e - 1 ) + K_e \\
 = -K_e(\gamma_{N+1} -2) + \gamma_{N+1} -1 .
 \end{align*}
Since $K_e = \max\limits_{i \in \llbracket 0,T-1 \rrbracket } \gamma_{i+1}\ldots\gamma_{i+e}\geq \gamma_{N+1} > 2$, we have $-K_e(\gamma_{N+1} -2) + \gamma_{N+1} -1 \leq - \gamma_{N+1}(\gamma_{N+1} -2) + \gamma_{N+1} -1 \leq -\gamma_{N+1}^2 + 3\gamma_{N+1} -1 \leq 0$
since $\gamma_{N+1} \geq 2 +\frac{\sqrt{5}-1}{2} = \frac{3+\sqrt{5}}{2}$. 
Thus, according to inequality (\ref{5maj_separe_Ke_eps}), for all $i \in \llbracket 0,e-1 \rrbracket$:
\begin{align*}
 \tau_i \leq \tau_0 \leq -c_{\ref{cons_epsilon_negatif}} \varepsilon
\end{align*}
with $c_{\ref{cons_epsilon_negatif}} > 0$ independent of $B$. 

Returning to the inequality (\ref{5inegal_D_N_i}), for all $i \in \llbracket 0,e-1 \rrbracket$ we have:
\begin{align*}
 \DD_{N+i} \leq c_{\ref{cons_maj_DNi}}\left( H(B)^{\tau_i } + H(B)^{- \frac{\varepsilon}{2} } \right) \leq c_{\ref{cons_maj_DNi}}\left( H(B)^{-c_{\ref{cons_epsilon_negatif}} \varepsilon} + H(B)^{- \frac{\varepsilon}{2} } \right).
\end{align*}
The term on the right tends to $0$ as $H(B) \to + \infty$. In particular, if $H(B)$ is large enough depending on $c_{\ref{cons_maj_DNi}},c_{\ref{cons_epsilon_negatif}} $, and $\varepsilon$, we have for all $ i \in \llbracket 0,e-1 \rrbracket, \quad \| X_{N+i} \wedge Z_1 \ldots \wedge Z_e \| = \DD_{N+i} < 1. $ Lemma~\ref{2lem_X_in_B} gives: 
\begin{align*}
 \forall i \in \llbracket 0,e-1 \rrbracket, \quad X_{N+i} \in B.
\end{align*}
Recalling that $B_{N,e} = \Span(X_N, \ldots, X_{N+e-1})$, we have shown that if $H(B)$ is large enough, then $B_{N,e} \subset B$ for $N$ satisfying $ \theta^{{\alpha_{N+e-1}}} \leq H(B)^{K_e + \frac{\varepsilon}{2} -1 } < \theta^{{\alpha_{N+e}}}$. By equality of dimensions, we then have $B_{N,e} =B$, and the lemma is proven.
\end{proof}

We have proven that the subspaces $B_{N,e}$ provide the best approximations of $A$ for the first angle. With the relation $\eqref{5relation_angle_hauteur}$, we established that:
\begin{align*}
 \psi_1(A,B_{N,e}) \geq c_{\ref{5cons_minor_prox_A_BN}} H(B_{N,e})^{-\gamma_{N+1}\ldots\gamma_{N+e}} \geq c_{\ref{5cons_minor_prox_A_BN}} H(B_{N,e})^{-\max\limits_{i \in \llbracket 0, T-1\rrbracket} \gamma_{i+1}\ldots\gamma_{i+e}}.
\end{align*}
Thus, the Diophantine exponent is bounded from above: $ \mu_n(A|e)_1 \leq \max\limits_{i \in \llbracket 0, T-1\rrbracket} \gamma_{i+1}\ldots\gamma_{i+e} $ which concludes the proof of Proposition~\ref{5prop_technique} using the first inequality \eqref{inegal_frombelow}.

\section{Proofs of Theorem~\ref{6theo_principal} and Theorem~\ref{1theo_premier_angle}} \label{sect_cas_dd}

Using Proposition~\ref{5prop_technique} we can prove Theorem~\ref{6theo_principal} by induction on $d$. It corresponds to Chapter 6 of \cite{Guillot_these}.

\bigskip

In section~\ref{sect_cas_d1}, the case $d = 1$ has been proven. For $d > 1$, we construct a space $A$ of dimension $d$ as the direct sum of a line $\Span(Y_1)$ (constructed as in section~\ref{sect_cas_d1}) and a space of dimension $d -1$ whose Diophantine exponents are known by the induction hypothesis. 
We then compute $\mu_n(A|e)_1$, for $e \in \llbracket 1, n-d \rrbracket$. To do this, we actually show that the best subspaces approximating $A$ are those that are close to $\Span(Y_1)$, and thus $ \mu_n(A|e)_1 = \mu_n(\Span(Y_1)|e)_1$.

We define $C_{1} = 2 + \dfrac{\sqrt{5}-1}{2}$ and for every integer $d \in \llbracket 2, n-1 \rrbracket$,
\begin{align}\label{6definition_Cd}
C_{d} = 5n^2 (C_{d-1})^{2n}.
\end{align}

For $d \in \llbracket 1, n-1 \rrbracket$, we define the following induction hypothesis $\HH(d)$:
\\ \textbf{For every $(\gamma_1, \ldots, \gamma_{n-d}) \in [C_{d}, +\infty[^{n-d}$, there exists a space $A$ of dimension $d$ generated by vectors of the form:}
\begin{align*}
 Y_1 = \begin{pmatrix}
 1 \\
 0 \\ 
 \vdots \\
 0 \\
 \tau_{0,1} \\
 \vdots \\ 
 \tau_{n-d-1,1}
 \end{pmatrix}, \quad Y_2 = \begin{pmatrix}
 1\\
 0\\ 
 \vdots \\
 \tau_{0,2} \\
 \tau_{1,2} \\
 \vdots \\ 
 \tau_{n-d,2}
 \end{pmatrix}, \ldots, \quad Y_d = \begin{pmatrix}
 1\\
 \tau_{0,d} \\
 \tau_{1,d} \\
 \vdots \\
 \\
 \vdots \\ 
 \tau_{n-1,d}
 \end{pmatrix}
\end{align*}
\textbf{with $\tau_{i,j} >0 $ such that the family $(\tau_{i,j})$ is algebraically independent over $\Q$, which satisfies for all $e \in \llbracket 1, n-d \rrbracket $: }
\begin{align*}
 \mu_n(A|e)_1 = \max\limits_{i \in \llbracket 0, n-d-e \rrbracket } \gamma_{i+1}\ldots\gamma_{i+e}.
\end{align*}
The case $d = 1$ is treated in Proposition~\ref{5prop_technique} so we fix $d \in \llbracket 2, n-1 \rrbracket$ and suppose that $\HH(d-1)$ is true. 

\subsection{Construction of the subspace \texorpdfstring{$A$}{}}
We apply $\HH(d-1)$ with $\gamma'_1 = \ldots = \gamma'_{n-d+1} = C_{d-1}$. Then, there exist vectors $Y_2, \ldots, Y_d \in \R^n$ of the form 
\begin{align*}
Y_2 &= \begin{pmatrix}
1 & 0 & \cdots& 0 & \tau_{0,2} & \cdots & \tau_{n-d,2}
\end{pmatrix}^\intercal, \\ Y_3 &= \begin{pmatrix}
1 & 0 & \cdots & 0 & \tau_{0,3} & \tau_{1,3} & \cdots & \tau_{n-d+1,3}
\end{pmatrix}^\intercal, \\ &\vdots & \\ \quad Y_d &= \begin{pmatrix}
1 & \tau_{0,d} & \tau_{1,d} & \cdots & \tau_{n-1,d}
\end{pmatrix}^\intercal
\end{align*}
with $\tau_{i,j} >0 $ and such that the family $(\tau_{i,j})$ is algebraically independent over $\Q$. The space $A' = \Span(Y_2, \ldots, Y_d)$ satisfies 
$ \forall e \in \llbracket 1, n-d+1 \rrbracket \quad \mu_n(A'|e)_1 = \lambda_e$ where we denote by $\lambda_e = (C_{d-1})^e $.

Let $ (\gamma_1, \ldots, \gamma_{n-d}) \in [C_{d}, +\infty[^{n-d}$. We shall now construct $Y_1 \in \R^n$ as in section~\ref{sect_cas_d1}. We extend the sequence $\gamma = (\gamma_k)_{k \in \Nx}$ by $ \gamma_i = C_d$ for all $i \in \llbracket n-d, 2n-2d \rrbracket$,
and by periodicity $\forall i \in \llbracket 1, 2n-2d \rrbracket, \forall k \in \N, \gamma_{i + k(2n-2d)} = \gamma_i.$
The sequence $(\gamma_k)_{k \in \Nx}$ then takes its values only in $\{ \gamma_1, \ldots, \gamma_{n-d}, C_d \}$.

\bigskip
We also introduce the sequence $\alpha = (\alpha_k)_{k \in \N}$ defined by
 $\alpha_0 = 1$, and $ \forall k \in \N, \quad \alpha_{k+1} = \gamma_{k+1}\alpha_k $. Let $\theta$ be a prime number greater than or equal to $5$, and let $\phi: \N \to \llbracket 0, n-d-1 \rrbracket $ be defined by:
\begin{align*}
 \phi(k) &= (k \mod (n-d)) \in \llbracket 0, n-d-1 \rrbracket.
\end{align*}

According to Lemma~\ref{2lem_sigma_alg_indep}, there exist $n-d$ sequences $u^0,\ldots, u^{n-d-1}$ satisfying: 
\begin{align}\label{6cons_suite_u}
 \forall j \in \llbracket 0, n-d-1 \rrbracket, \forall k \in \N, \quad u^j_k \left\{ \begin{array}{lll}
 \in \{1, 2 \} &\text{ if } \phi(k) = j \\
 = 0 &\text{ otherwise }
 \end{array} \right. 
\end{align}
where $u^j_k \neq 0 $ if and only if $j = k \mod (n-d-1)$; and such that the family $(\sigma_0, \ldots, \sigma_{n-d-1})$ is algebraically independent over $\Q(\FF )$, where $\FF =(\tau_{i,j})$ and for $j \in \llbracket 0, n-d-1 \rrbracket$, $ \sigma_j = \sum\limits_{k = 0}^{+ \infty} \dfrac{u^j_k}{\theta^{\floor{\alpha_k}}}.$ 
We then set $Y_1 = \begin{pmatrix}
 1 & 0 & \cdots & 0 &
 \sigma_0 &
 \cdots &
 \sigma_{n-d-1}
 \end{pmatrix}^\intercal $.
Since this construction is the same as the one in section~\ref{sect_cas_d1}, Proposition~\ref{5prop_technique} yields:
\begin{align}\label{6exposant_Y1}
 \forall e \in \llbracket 1,n-d \rrbracket, \quad \mu_n(\Span(Y_1)|e)_1 = \max\limits_{i \in \llbracket 0, n-d-e \rrbracket } \gamma_{i+1}\ldots\gamma_{i+e}
 \end{align}
by embedding $Y_1$ into $\R \times \{0\}^{d-1} \times \R^{n-d}$ and using Theorem 1.2 of \cite{joseph_spectre}.
We then set $A = \Span(Y_1, \ldots, Y_d)$ and denote by $K_e$ the quantities $ K_e = \max\limits_{i \in \llbracket 0, n-d-e \rrbracket } \gamma_{i+1}\ldots\gamma_{i+e}$ for $e \in \llbracket 1, n-d \rrbracket$. 
\bigskip

Define $M$ as the matrix which columns are $Y_1, \ldots, Y_d$. We can write $M = \begin{pmatrix}
 G \\ \Sigma
 \end{pmatrix}$ with:
 \begin{align*}
 G = \begin{pmatrix}
 1 & 1 & \cdots & \cdots & 1 \\
 0 & & & 0& \tau_{0,d} \\
 \vdots & \cdots & 0 & \tau_{0,d-1} & \tau_{1,d} \\
 \vdots & \udots & & & \vdots \\
 0 & \tau_{0, 2} & \cdots & \cdots & \tau_{d-2,d}
 \end{pmatrix}
 \text{ and } \Sigma = \begin{pmatrix}
 \sigma_0 & \tau_{1,2} & \cdots & \tau_{d-1,d} \\
 & &\vdots & \\
 \sigma_{n-d-1} & \tau_{n-d,2} & \cdots & \tau_{n-2,d}
 \end{pmatrix}.
 \end{align*}
 We have $\Sigma \in \mathcal{M}_{n-d,d}(\R)$ and $G \in \GL_{d}(\R)$ because $\det(G) = \pm \prod\limits_{i=2}^{d} \tau_{0, i} \neq 0 $ by expanding with respect to the first column. Moreover, by construction, the family $\{ \sigma_k \} \cup \{ \tau_{i,j}\}$ is algebraically independent over $\Q$. In particular, we notice that the coefficients of $\Sigma$ form a set algebraically independent over $\Q(\FF)$ where $\FF$ is the set of coefficients of $G$. So the space $A$ spanned by the vectors $Y_1, \ldots, Y_d$ is $(n-d,1)$-irrational by Lemma~\ref{3lem_1_irrat}.

 We now fix $e \in \llbracket 1, n-d \rrbracket$ and compute $\mu_n(A|e)_1$. Using $(\ref{6exposant_Y1})$, we have $\mu_n(\Span(Y_1)|e)_1 = K_e.$
We note that $\Span(Y_1) \subset A$ and thus obtain
\begin{align}\label{6minoration_exposant_par_Ke}
 \mu_n(A|e)_1 \geq \mu_n(\Span(Y_1)|e)_1 = K_e.
\end{align}
It remains to bound from above the exponent $\mu_n(A|e)_1$ by $K_e$, and this is the subject of the remainder of the section.

\subsection{First angle}

We want to show that the "best" rational subspaces approaching $A$ are those that best approximate the line $\Span(Y_1)$. As in section~\ref{sect_cas_d1} we call them $B_{N,e}$ for $N \in \N$, namely $B_{N,e }= \Span(X_N, X_{N+1}, \ldots, X_{N+e-1})$ with 
\begin{align*}
X_N &= \theta^{\lfloor \alpha_N \rfloor} \begin{pmatrix}
1 &
0 &
\cdots &
0 &
\sigma_{0,N} &
\cdots &
\sigma_{n-d-1,N}
\end{pmatrix}^\intercal \in \mathbb{Z}^n
\end{align*}
where we have defined $\sigma_{j,N} = \sum\limits_{k=0}^{N} \frac{u^{j}_k}{\theta^{\lfloor \alpha_k \rfloor}} \in \frac{1}{\theta^{\lfloor \alpha_N \rfloor}} \mathbb{Z}$ for $j \in \llbracket 0, n-d-1 \rrbracket$. The subspace $B_{N,e}$ satisfies Claim~\ref{5B_N_base_v} and Lemma~\ref{5lem_haut_BN}. In this section, we first bound from below the angle $\psi_1(A, B_{N,e})$.

\begin{lem}\label{6min_psi_A_BN}
There exists a constant $\cons \label{6cons_min_angle_1} >0 $ independent of $N$ such that for all $N \in \mathbb{N}$: 
\begin{align*}
\psi_1(A, B_{N,e}) \geq c_{\ref{6cons_min_angle_1}} \theta^{ - \alpha_{N+e}}.
\end{align*}
\end{lem}

\begin{proof}
Let $X \in B_{N,e}$ and $Y \in A$ be nonzero vectors such that $\omega(X,Y) = \psi_1(A,B_{N,e}).$
According to Claim~\ref{5B_N_base_v}, we have $B_{N,e }= \Span(X_N, v_{N+1}, \ldots, v_{N+e-1})$ where $v_j = \begin{pmatrix}
 0 &
 \cdots & 0 & \rho_{0, j} & \cdots & \rho_{n-d-1, j}
\end{pmatrix}^\intercal $ with $\rho_{i, j} = \dfrac{1}{u^{\phi(j)}_{j}}u^i_{j} \in \{0,1 \}$ for $i \in \llbracket 0, n-d-1 \rrbracket$.
We then define $a, a_1, \ldots, a_{e-1} \in \mathbb{R}$ such that $ X = a\theta^{-\floor{\alpha_N}}X_N + \sum\limits_{i =1}^{e-1} a_iv_{N+i}$, so $X$ is of the form 
\begin{align*}
 X = \begin{pmatrix}
a &
0 &
\cdots &
0 &
a\sigma_{0,N} + \sum\limits_{i = 1 }^{e-1} a_i \rho_{0,N+i} &
\cdots &
a\sigma_{n-d-1,N} +\sum\limits_{i = 1 }^{e-1} a_i \rho_{n-d-1,N+i}
\end{pmatrix}^\intercal.
\end{align*}
Similarly, we define $b_1, \ldots, b_d \in \mathbb{R}$ such that $Y = \sum\limits_{j = 1}^d b_jY_j $, so $Y$ is of the form 
\begin{align*}
 Y = \begin{pmatrix}
 \sum\limits_{j=1}^d b_j &
b_d \tau_{0,d} &
\cdots &
\sum\limits_{j = 2}^d b_j \tau_{j-2,j} &
b_1\sigma_0 + \sum\limits_{j =2 }^d b_j \tau_{j-1,j} &
\cdots & 
b_1\sigma_{n-d-1} + \sum\limits_{j =2 }^d b_j\tau_{j+n-d-2,j}
\end{pmatrix}^\intercal.
\end{align*}
Assume, without loss of generality, that we have $a^2 + \sum\limits_{i=1}^{e-1} a_i^2 = \sum\limits_{j=1}^d b_j^2 = 1.$ In particular, this assumption yields 
$ \| X \| \cdot \|Y\| \leq c_{\ref{6cons_maj_prodXY}} $
with $\cons \label{6cons_maj_prodXY} >0$ a constant independent of $N$. 
\\ We then seek to bound $\|X \wedge Y \|$ from below. Here, we use the fact that the coordinates of the vector $X \wedge Y$ are the minors of size $2$ of the matrix $\begin{pmatrix}
 X & Y
\end{pmatrix} \in \MM_{n,2}(\mathbb{R})$. In particular, $\| X \wedge Y \| $ is bounded below by the absolute value of each of these minors. 
We define the following quantities: 
\begin{align*}
 \sigma = \max(1, \max\limits_{i \in \llbracket 0, n-d-1\rrbracket} (\sigma_i )), &\quad
 \tau = \min\limits_{ j \in \llbracket 2, d\rrbracket} \min\limits_{i \in \llbracket 1, n-1-d+j\rrbracket}(\tau_{i, j} ), \\
 T = \max\limits_{ j \in \llbracket 2, d\rrbracket} \max\limits_{i \in \llbracket 1, n-1-d+j\rrbracket}(\tau_{i, j} ), &\quad
 s = \dfrac{1}{\theta^{\floor{\alpha_{n-d-1}}}}.
\end{align*}
We then assume $N \geq n-d-1$, thus for all $i \in \llbracket 0, n-d-1 \rrbracket $ there exists $ k \in \llbracket 0, N \rrbracket$ such that $\phi(k) = i$, hence $u^i_k \in \{1,2\}$ and 
$ s \leq \sigma_{i, N}. $
Finally, we define the following quantity:
\begin{align}\label{6defM}
 M = \dfrac{\tau}{4(d-1)(T+\sigma)(1 + \frac{T}{\tau})^{d-2} } >0.
\end{align}
We will show that
\begin{align}\label{6minor_voulue_XYwedge}
 \| X \wedge Y \| \geq c_{\ref{6cons_minor_XwedgeY}} \theta^{ - \alpha_{N+e}}
\end{align}
with $\cons \label{6cons_minor_XwedgeY} > 0 $ independent of $N$, by considering different cases based on the values taken by the $b_j$.
\\ \textbullet \, \underline{First case:} If there exists $k \in \llbracket 2, d \rrbracket$ such that $\left| 
 \sum\limits_{j = k}^d b_j \tau_{j-k,j} \right| \geq M \theta^{- \alpha_{N+e}}$, we study the minor of the matrix $\begin{pmatrix}
 X & Y
\end{pmatrix} $ corresponding to rows $d- k +2 $ and $d+1 +\ell $ with $\ell \in \llbracket 0, n-d-1 \rrbracket$. We then have:
\begin{align}\label{6minor_cas_1_bminor}
 \forall \ell \in \llbracket 0, n-d-1 \rrbracket, \quad \| X \wedge Y \| &\geq \left| \det \begin{pmatrix}
 0 & \sum\limits_{j = k}^d b_j \tau_{j-k,j} \\
 a\sigma_{\ell,N} + \sum\limits_{i = 1 }^{e-1} a_i \rho_{\ell,N+ i} & b_1\sigma_\ell + \sum\limits_{j =2 }^d b_j \tau_{j+ \ell -1,j}
 \end{pmatrix} \right| \nonumber \\
 &= \left| \sum\limits_{j = k}^d b_j \tau_{j-k,j} \right| \cdot \left| a\sigma_{\ell,N} + \sum\limits_{i = 1 }^{e-1} a_i \rho_{\ell,N+i} \right| \nonumber \\
 &\geq M \theta^{- \alpha_{N+e}}\left| a\sigma_{\ell,N} + \sum\limits_{i = 1 }^{e-1} a_i \rho_{\ell,N+i} \right|.
\end{align}
We then distinguish between two cases: $|a| \geq \dfrac{1}{2\sigma\sqrt{e}}$ and $|a| < \dfrac{1}{2\sigma\sqrt{e}}$. In the first case, we choose $\ell = \phi(N+e)$. We have in particular $ \rho_{\ell,N+1} = \ldots = \rho_{\ell,N+e-1} = 0
$
and thus according to $(\ref{6minor_cas_1_bminor})$ and using the fact that $ s \leq \sigma_{i, N}:$
\begin{align*}
 \| X \wedge Y \| \geq M \theta^{- \alpha_{N+e}}\left| a\sigma_{\ell,N} \right| \geq \dfrac{M\sigma_{\ell,N}}{2\sigma\sqrt{e}} \theta^{- \alpha_{N+e}} \geq \dfrac{Ms}{2\sigma\sqrt{e}} \theta^{- \alpha_{N+e}}
\end{align*}
which yields $(\ref{6minor_voulue_XYwedge})$.
In the second case, we have $|a| < \dfrac{1}{2\sigma\sqrt{e}} $ and thus $ \sum\limits_{i = 1}^{e-1}a_i^2 = 1 - a^2 \geq 1- \dfrac{1}{4\sigma^2 e} \geq \dfrac{e-1}{e}$ since $\sigma \geq 1$. In particular, there exists $i \in \llbracket 1, e-1 \rrbracket$ such that $|a_i| \geq \dfrac{1}{\sqrt{e}}$. We then choose $\ell = \phi(N+i)$. In particular $ \rho_{\ell, N+i } = 1 \text{ and } \rho_{\ell,N+1} = \ldots =\rho_{\ell,N+i-1} = \rho_{\ell,N+i+1}= \ldots = \rho_{\ell,N+e-1} = 0.$
According to $(\ref{6minor_cas_1_bminor})$ we have:
\begin{align*}
 \| X \wedge Y \| &\geq M \theta^{- \alpha_{N+e}}\left| a\sigma_{\ell,N} + a_i \right| \geq M \theta^{- \alpha_{N+e}} \left(|a_i| - |a|\sigma \right) \geq M \theta^{- \alpha_{N+e}} \left(\dfrac{1}{\sqrt{e}} - \dfrac{1}{2\sqrt{e}}\right) \geq \dfrac{M}{2\sqrt{e}} \theta^{- \alpha_{N+e}}.
\end{align*}
\textbullet \, \underline{Second case:} 
 $\forall k \in \llbracket 2, d \rrbracket, \quad \left| 
 \sum\limits_{j = k}^d b_j \tau_{j-k,j} \right| < M \theta^{- \alpha_{N+e}}.$
We will show later (in Lemma~\ref{6lem_maj_b}) that we then have: 
\begin{align}\label{6resultat_sur_familleb_admis}
 \forall j \in \llbracket 2, d \rrbracket, \quad |b_j| \leq \dfrac{M(1 + \frac{T}{\tau})^{d-2} \theta^{-\alpha_{N+e}}}{\tau}, \quad \left| \sum\limits_{j=1}^d b_j \right| \geq \dfrac{1}{2} \text{ and } |b_1| \geq \dfrac{3}{4}
\end{align}
because $M = \dfrac{\tau}{4(d-1)(T+\sigma)(1 + \frac{T}{\tau})^{d-2} }\leq \dfrac{\tau}{4(d-1)(1 + \frac{T}{\tau})^{d-2} \theta^{-\alpha_{N+e}}} $. We currently accept this result. We consider the minor of the matrix $\begin{pmatrix}
 X & Y
\end{pmatrix} $ corresponding to rows $1$ and $d+1 +\ell $ with $\ell \in \llbracket 0, n-d-1 \rrbracket$. We then have for any $\ell \in \llbracket 0, n-d-1 \rrbracket,$
\begin{align}\label{6minor_cas_2_bminor}
 \| X \wedge Y \| &\geq \left| \det \begin{pmatrix}
 a & \sum\limits_{j = 1}^d b_j \\
 a\sigma_{\ell,N} + \sum\limits_{i = 1 }^{e-1} a_i \rho_{\ell,N+i} & b_1\sigma_\ell + \sum\limits_{j =2 }^d b_j \tau_{j+\ell-1,j}
 \end{pmatrix} \right|. 
\end{align}
We then distinguish between two cases: $|a| \geq K$ and $|a| < K $ with $K = \dfrac{1}{4(\sigma + (T + \sigma)(d-1))\sqrt{e}}$. In the first case, we choose $ \ell = \phi(N+e)$. We have in particular $\rho_{\ell,N+e} = 1 \text{ et } \rho_{\ell,N+1} = \ldots = \rho_{\ell,N+e-1} = 0$ and $\sigma_{\ell,N} = \sum\limits_{k = 0}^{N} \dfrac{u^{\ell}_k}{\theta^{\floor{\alpha_k}}} = \sum\limits_{k = 0}^{N+e-1 } \dfrac{u^{\ell}_k}{\theta^{\floor{\alpha_k}}} $ since then $u^{\ell}_{N+1} = \ldots = u^{\ell}_{N+e-1} = 0 $. We deduce that
\begin{align*}
 | \sigma_\ell - \sigma_{\ell, N} | = \sum\limits_{k = N+e}^{+ \infty } \dfrac{u^{\ell}_k}{\theta^{\floor{\alpha_k}}} \geq\dfrac{u^{\ell}_{N+e}}{\theta^{\floor{\alpha_{N+e}}}} \geq \dfrac{1}{\theta^{{\alpha_{N+e}}}}.
\end{align*}
Inequality $(\ref{6minor_cas_2_bminor})$ then gives:
\begin{align*}
 \| X \wedge Y \| &\geq \left| a\left(b_1\sigma_\ell + \sum\limits_{j =2 }^d b_j \tau_{j+ \ell-1,j} \right) - a\sigma_{\ell,N} \sum\limits_{j = 1}^d b_j\right| \geq |a| \left| b_1 (\sigma_\ell - \sigma_{\ell, N}) - \sum\limits_{j =2 }^d b_j (\sigma_{\ell, N} - \tau_{j+ \ell-1,j}) \right| 
\end{align*}
and thus by recalling that $\sigma$ and $T$ bound from above respectively the $\sigma_{\ell} $ and the $\tau_{\ell,j}$ we have 
\begin{align*}
 \| X \wedge Y\| &\geq |a| \left(|b_1|\dfrac{1}{\theta^{{\alpha_{N+e}}}} - (\sigma +T) \sum\limits_{j =2 }^d |b_j| \right) \geq |a|\left(\dfrac{3}{4\theta^{{\alpha_{N+e}}}} - (\sigma +T)(d-1) \dfrac{M(1 + \frac{T}{\tau})^{d-2} }{\tau \theta^{\alpha_{N+e}}} \right) \\
 &\geq \dfrac{|a|}{\theta^{\alpha_{N+e}}} \left(\dfrac{3}{4} - \dfrac{1}{4} 
 \right) \geq \dfrac{K}{2\theta^{\alpha_{N+e}}}
\end{align*}
using $(\ref{6resultat_sur_familleb_admis})$ and by definition of $M$ in $(\ref{6defM})$.

 In the second case, we have $|a| < K = \dfrac{1}{4(\sigma + (T + \sigma)(d-1))\sqrt{e}}$ and thus:
 \begin{align*}
 \sum\limits_{i = 1}^{e-1}a_i^2 = 1 - a^2 \geq 1- \dfrac{1}{16(\sigma + (T + \sigma)(d-1))^2 e} \geq \dfrac{e-1}{e}
 \end{align*}
 since $\sigma \geq 1$. In particular, there exists $i \in \llbracket 1, e-1 \rrbracket$ such that $|a_i| \geq \dfrac{1}{\sqrt{e}}$. We then choose $\ell = \phi(N+i)$. Specifically we have 
$ \rho_{\ell, N+i } = 1 \text{ and } \rho_{\ell,N+1} = \ldots \rho_{\ell,N+i-1} = \rho_{\ell,N+i+1}= \ldots = \rho_{\ell,N+e-1} = 0.$
According to $(\ref{6minor_cas_2_bminor})$ and $(\ref{6resultat_sur_familleb_admis})$, we have:
\begin{align*}
 \| X \wedge Y \| &\geq \left| a\bigg(b_1\sigma_\ell + \sum\limits_{j =2 }^d b_j \tau_{j + \ell -1,j} \bigg) - (a\sigma_{\ell,N} + a_i ) \sum\limits_{j = 1}^d b_j\right| \\
 &\geq \left| a_i \sum\limits_{j = 1}^d b_j\right| - \left|ab_1(\sigma_\ell - \sigma_{\ell,N}) \right| - \left| a\sum\limits_{j = 2}^d b_j(\tau_{j + \ell -1,j} - \sigma_{\ell, N}) \right| \\ 
 &\geq \dfrac{|a_i|}{2} - |a| (\sigma + (T + \sigma)\sum\limits_{j = 2}^d |b_j|) \geq \dfrac{1}{2\sqrt{e}} - |a| \bigg(\sigma + (T + \sigma)(d-1) \bigg) \geq \dfrac{1}{4\sqrt{e}}
\end{align*}
by bounding $\sum\limits_{j = 2}^d |b_j|$ by $ d-1$ since $\sum\limits_{j=1}^d b_j^2 = 1$.
Consequently, $\|X \wedge Y \| \geq \theta^{ - \alpha_{N+e}}$ if $N $ is big enough. In all cases, we find that, for all $N \geq n-d -1$, $ \| X \wedge Y \| \geq c_{\ref{6cons_minor_XwedgeY}} \theta^{ - \alpha_{N+e}}$
with $c_{\ref{6cons_minor_XwedgeY}} > 0 $ independent of $N$. By decreasing $c_{\ref{6cons_minor_XwedgeY}}$ if necessary, since $\psi_1(A,B_{N,e}) > 0 $ due to the $(e,1)$-irrationality of $A$, we can assume this inequality holds for all $N \in \N$.
\\ Using the fact that $ \| X \| \cdot \|Y\| \leq c_{\ref{6cons_maj_prodXY}} $, we therefore obtain:
\begin{align*}
 \psi_1(A, B_{N,e}) = \omega(X,Y) = \dfrac{\| X \wedge Y \| }{\| X \| \cdot \|Y\|} \geq c_{\ref{6cons_minor_XwedgeY}} c_{\ref{6cons_maj_prodXY}} ^{-1} \theta^{ - \alpha_{N+e}}.
\end{align*}
Thus, the lemma is proven with $c_{\ref{6cons_min_angle_1}} = c_{\ref{6cons_minor_XwedgeY}} c_{\ref{6cons_maj_prodXY}} ^{-1}$.
 
\end{proof}

We will now prove the result used in the previous proof.

\begin{lem}\label{6lem_maj_b}
Let $b_1, \ldots, b_d $ satisfy $ \sum\limits_{j=1}^d b_j^2 = 1$. Suppose there exists $0 <M \leq \dfrac{\tau}{4(d-1)(1 + \frac{T}{\tau})^{d-2} \theta^{-\alpha_{N+e}}} $ such that:
 \begin{align}\label{6lem_maj_b_hyp}
 \forall k \in \llbracket 2, d \rrbracket, \quad \left| 
 \sum\limits_{j = k}^d b_j \tau_{j-k,j} \right| < M \theta^{- \alpha_{N+e}}.
\end{align}
Then 
\begin{align*}
 |b_1| \geq \dfrac{3}{4}, \quad \forall j \in \llbracket 2, d \rrbracket, \quad |b_j| \leq \dfrac{M(1 + \frac{T}{\tau})^{d-2} \theta^{-\alpha_{N+e}}}{\tau} \text{ and } \left| \sum\limits_{j=1}^d b_j \right| \geq \dfrac{1}{2} 
\end{align*}
where $\tau$ and $T$ are respectively the minimum and maximum of the family $\{ \tau_{i,j}\}$.
\end{lem}

\begin{proof}
We proceed by descending induction on $i \in \llbracket 2, d\rrbracket$. We prove a more refined result:
\begin{align}\label{6lem_maj_b_hyp_rec}
\forall i \in \llbracket 2, d\rrbracket, \quad |b_i| \leq \dfrac{M(1 + \frac{T}{\tau})^{d-i} \theta^{-\alpha_{N+e}}}{\tau}.
\end{align}
If $i = d $ then hypothesis $(\ref{6lem_maj_b_hyp})$ with $ k =d$ gives $ |b_d| \leq \dfrac{ M \theta^{- \alpha_{N+e}}}{ \tau_{n-1,d}} \leq \dfrac{M(1 + \frac{T}{\tau})^{d-d} \theta^{-\alpha_{N+e}}}{\tau}.$

Let $i\in \llbracket 2, d-1 \rrbracket $. Suppose that for all $i' > i $ inequality $(\ref{6lem_maj_b_hyp_rec})$ holds. Then we apply hypothesis $(\ref{6lem_maj_b_hyp})$ with $ k = i$ and have $ \left| \sum\limits_{j= i}^d b_j \tau_{j-i, j} \right| < M \theta^{- \alpha_{N+e}}.$
In particular, we have $ | b_i \tau_{j-i, i}| \leq M \theta^{- \alpha_{N+e}} + \sum\limits_{j= i+1}^d \left| b_j \tau_{j-i, j} \right| $, and using the induction hypothesis:
\begin{align*}
 | b_i | \leq \dfrac{1}{\tau_{j-i, i}}\left(M \theta^{- \alpha_{N+e}} + \sum\limits_{j= i+1}^d \dfrac{M(1 + \frac{T}{\tau})^{d-j} \theta^{-\alpha_{N+e}}}{\tau}T \right) &\leq \dfrac{M}{\tau} \theta^{- \alpha_{N+e}} \left(1 + \dfrac{T}{\tau} \left(\dfrac{1 -(1 + \frac{T}{\tau})^{d-i} }{-\frac{T}{\tau}}\right)\right) \\&= \dfrac{M}{\tau}\theta^{- \alpha_{N+e}}\left(1 + \frac{T}{\tau}\right)^{d-i}
\end{align*}
which proves $(\ref{6lem_maj_b_hyp_rec})$ for all $i \in \llbracket 2,d \rrbracket$ and thus the second inequality of the lemma. In particular, we have:
\begin{align*}
 \forall i \in \llbracket 2, d\rrbracket, \quad |b_i| \leq \dfrac{M(1 + \frac{T}{\tau})^{d-2} \theta^{-\alpha_{N+e}}}{\tau} \leq \dfrac{1}{4(d-1)}.
\end{align*}
We use the fact that $\sum\limits_{i = 1}^d b_i^2 = 1$ to obtain $ |b_1|^2 = 1 - \sum\limits_{i = 2}^d b_i^2 \geq 1 - (d-1)\dfrac{1}{16(d-1)^2} \geq \dfrac{3}{4}$,
and moreover 
$ \left| \sum\limits_{j=1}^d b_j \right| \geq |b_1| - \left| \sum\limits_{j=2}^d b_j \right| \geq \dfrac{3}{4} - (d-1)\dfrac{1}{4(d-1)} = \dfrac{1}{2}.$

\end{proof}
\subsection{Rational subspaces of best approximation and conclusion}

We show here that the subspaces $B_{N,e}$ are the ones that achieve the best approximations of $A$ in the first angle. Recall that the integer $e \in \llbracket 1, n-d \rrbracket$ is fixed.

\begin{lem}\label{6lem_meilleurs_espaces}
Let $\varepsilon > 0$ and $B$ be a rational subspace of dimension $e$ such that:
\begin{align}\label{6hyp_lem_meilleurs_espaces}
\psi_1(A,B) \leq H(B)^{-K_e- \varepsilon}.
\end{align}
Then if $H(B)$ is large enough depending on $\varepsilon$ and $A$, there exists $N \in \mathbb{N}$ such that $B = B_{N,e}$.
\end{lem}

Let $Z_1, \ldots, Z_e$ be a $\mathbb{Z}$-basis of $B \cap \mathbb{Z}^n$. We follow the scheme of proof of Lemma~\ref{5lem_meilleurs_espaces}, showing that if $H(B)$ is large enough, then for every $i \in \llbracket 0, e-1 \rrbracket$, the inequality
$\DD_{N+i} = \| X_{N+i} \wedge Z_1 \wedge \ldots \wedge Z_e \| < 1$ holds for some $N \in \mathbb{N}$ to be determined. For this purpose, we prove two preliminary claims, recalling the notation $\lambda_e = (C_{d-1})^e$.

\begin{claim}\label{6lem_maj_Y1_Y_d_Z1_Ze}
Let $\varepsilon > 0$. Suppose that $B$ is a rational space of dimension $e$ such that $\psi_1(A,B) \leq H(B)^{-\lambda_e -\varepsilon}$.
Then, if $H(B)$ is large enough depending on $A$ and $\varepsilon$, we have:
\begin{align*}
\forall \delta \in \left(0, \dfrac{\varepsilon}{2}\right), \quad \| Y_1 \wedge \ldots \wedge Y_d \wedge Z_1 \wedge \ldots \wedge Z_e \| 
&\leq c_{\ref{6cons_major_Y_Z}} \psi_1(A,B) H(B)^{1+\lambda_e + \delta} 
\end{align*}
with $\cons \label{6cons_major_Y_Z} > 0$ independent of $B$ but possibly dependent on $\delta$.
\end{claim}

\begin{proof}
Let $Y = \sum\limits_{j=1}^d a_j Y_j \in A$ of norm $1$ such that $\psi_1(A,B) = \omega(Y,p_B(Y))$ with $p_B$ the orthogonal projection onto $B$. We use the following relation:
\begin{align}\label{6relation_a_1_X}
\omega ( \sum\limits_{j = 2}^d a_j Y_j, p_B(Y) ) - \omega(Y, p_B(Y) ) \leq \omega(Y, \sum\limits_{j = 2}^d a_j Y_j )
\end{align}
coming from the triangle inequality (see equation (3) of \cite{Schmidt}). Now,
\begin{align}\label{6maj_angle_YajYj}
\omega(Y, \sum\limits_{j = 2}^d a_j Y_j ) = \dfrac{\left\| \sum\limits_{j = 1}^d a_j Y_j \wedge \sum\limits_{j = 2}^d a_j Y_j \right\| }{\left\| \sum\limits_{j = 2}^d a_j Y_j \right\| } = \dfrac{\left\| a_1 Y_1 \wedge \sum\limits_{j = 2}^d a_j Y_j \right\| }{\left\| \sum\limits_{j = 2}^d a_j Y_j \right\| } \leq \left\| a_1 Y_1 \right\|.
\end{align}
Moreover, let $\delta \in \left( 0, \dfrac{\varepsilon}{2} \right)$. Since $ \sum\limits_{j = 2}^d a_j Y_j \in A' = \Span(Y_2, \ldots, Y_d)$ and assuming $ \sum\limits_{j = 2}^d a_j Y_j \neq 0 $, the induction hypothesis $\HH(d-1)$ gives:
$\omega( \sum\limits_{j = 2}^d a_j Y_j, p_B(Y) ) \geq \psi_1(A',B ) \geq c_{\ref{6cons_Aprime}} H(B)^{-\mu_n(A'|e)_1-\delta} = c_{\ref{6cons_Aprime}} H(B)^{-\lambda_e-\delta}$ where $\cons > 0 \label{6cons_Aprime}$ depends only on $A'$ and $\delta$. Revisiting $(\ref{6relation_a_1_X})$ and $(\ref{6maj_angle_YajYj})$, we get:
\begin{align*}
c_{\ref{6cons_Aprime}} H(B)^{-\lambda_e-\delta} - \psi_1(A,B) \leq |a_1| \|Y_1\|.
\end{align*}
Recalling that $\delta < \dfrac{\varepsilon}{2}$ and by assumption $\psi_1(A,B) \leq H(B)^{-\lambda_e-\varepsilon}$, we can conclude since $H(B)$ is large enough, leading to $H(B)^{-\frac{\varepsilon}{2}} \leq \dfrac{c_{\ref{6cons_Aprime}}}{2}$, hence:
\begin{align*}
\dfrac{c_{\ref{6cons_Aprime}}}{2} H(B)^{-\lambda_e-\delta} \leq H(B)^{-\lambda_e-\delta}\left(c_{\ref{6cons_Aprime}} - H(B)^{-\frac{\varepsilon}{2}}\right) \leq c_{\ref{6cons_Aprime}} H(B)^{-\lambda_e-\delta} - H(B)^{ -\lambda_e-\varepsilon} \leq |a_1| \|Y_1\|
\end{align*}
thus, $ c_{\ref{6cons_minor_a_1}}H(B)^{-\lambda_e-\delta} \leq |a_1|$ with $\cons \label{6cons_minor_a_1} >0 $ independent of $B$ but dependent on $\delta$. Moreover, this inequality is still true if $ \sum\limits_{j = 2}^d a_j Y_j = 0 $, because in this case $a_1 = \|Y_1 \|^{-1}$.
In particular, $a_1 \neq 0$. Let $D_{Y,Z} = \| Y_1 \wedge \ldots \wedge Y_d \wedge Z_1 \wedge \ldots \wedge Z_e \|$ and let us compute:
\begin{align*}
D_{Y,Z} &= \dfrac{1}{|a_1|}\| a_1Y_1 \wedge Y_2 \wedge \ldots \wedge Y_d \wedge Z_1 \wedge \ldots \wedge Z_e \| \\
&= \dfrac{1}{|a_1|}\left\| \left(a_1Y_1 + \sum\limits_{j = 2}^d a_j Y_j - p_B(X)\right) \wedge Y_2 \wedge \ldots \wedge Y_d \wedge Z_1 \wedge \ldots \wedge Z_e \right\| \\
&= \dfrac{1}{|a_1|}\| \left(Y - p_B(Y)\right) \wedge Y_2 \wedge \ldots \wedge Y_d \wedge Z_1 \wedge \ldots \wedge Z_e \|
\end{align*}
since $ \sum\limits_{j = 2}^d a_j Y_j - p_B(X) \in \Span(Y_2, \ldots, Y_d, Z_1, \ldots, Z_e)$. We can then bound:
\begin{align*}
D_{Y,Z} \leq \dfrac{\| Y - p_B(Y)\| \cdot \| Y_2 \wedge \ldots \wedge Y_d \| \cdot \| Z_1 \wedge \ldots \wedge Z_e \| }{|a_1|} \leq \dfrac{\psi_1(A,B) \| Y_2 \wedge \ldots \wedge Y_d \| H(B) }{ c_{\ref{6cons_minor_a_1}}H(B)^{-\lambda_e-\delta}} 
\end{align*}
since $Z_1, \ldots, Z_e$ is a $\Zbasis$ of $B \cap \Z^n$.
Thus, we have for all $\delta \in \left( 0, \dfrac{\varepsilon}{2} \right), \quad \| Y_1 \wedge \ldots \wedge Y_d \wedge Z_1 \wedge \ldots \wedge Z_e \| 
\leq c_{\ref{6cons_major_Y_Z}} \psi_1(A,B) H(B)^{1+\lambda_e+\delta} $
where $c_{\ref{6cons_major_Y_Z}} = \| Y_2 \wedge \ldots \wedge Y_d \| c_{\ref{6cons_minor_a_1}}^{-1}$.

\end{proof}

\begin{claim}\label{6lem_maj_DN}
Let $\varepsilon > 0$. Suppose that $B$ is a rational space of dimension $e$ such that $\psi_1(A,B) \leq H(B)^{-\lambda_e - \varepsilon}$. For an integer $N$, let $\DD_N = \| X_{N} \wedge Z_1 \ldots \wedge Z_e \| $.
Then, if $H(B)$ is sufficiently large depending on $A$ and $\varepsilon$, for all $N \in \N$ and $\delta \in \left(0, \frac{\varepsilon}{2}\right)$, we have:
 \begin{align*}
 \DD_N \leq c_{\ref{6cons_maj_D_N_lem}}\theta^{{\alpha_{N}} (1 + t\lambda_{e+1}+t\delta) } H(B)^{1+ t\lambda_{e+1}+t\delta} \left(\psi_1(A,B) H(B)^{ \lambda_e+\delta } + \dfrac{1 }{\theta^{{\alpha_{N+1}}}} \right)
 \end{align*}
where $t = \min(d-1,e+1)$, and $\cons \label{6cons_maj_D_N_lem} >0$ is independent of $N$ and $B$ but depends on $\delta$.
\end{claim}

\begin{proof}
Let $\EE_N = \| X_N \wedge Y_2 \wedge \ldots \wedge Y_d \wedge Z_1 \ldots \wedge Z_e \|$. Using the results at the end of page 446 of \cite{Schmidt}, we have:
\begin{align}\label{6lien_phi_EEN}
 \EE_N &= \| Y_2 \wedge \ldots \wedge Y_d \wedge X_N \wedge Z_1 \ldots \wedge Z_e \| =\varphi(A', C_N ) \| Y_2 \wedge \ldots \wedge Y_d \| \cdot \| X_N \wedge Z_1 \ldots \wedge Z_e \|
\end{align}
where $C_N = \Span(X_N,Z_1 \ldots, Z_e)$, $A' = \Span(Y_2, \ldots, Y_d)$, $ \varphi(A', C_N ) =\psi_1(A',C_N) \ldots \psi_u(A',C_N)$ and $u = \min(d-1, \dim(C_N))$.
We have $\varphi(A', C_N ) \geq \psi_1(A',C_N) ^u $. Since $\dim(C_N) \in \{e,e+1\}$, we have $u \leq \min(d-1,e+1) = t $, hence 
$ \varphi(A', C_N ) \geq \psi_1(A',C_N)^t$ as $\psi_1(A',C_N)\leq 1$. Equation (\ref{6lien_phi_EEN}) becomes:
\begin{align}\label{6maj_DN_1}
 \DD_N = \| X_N \wedge Z_1 \ldots \wedge Z_e \| \leq \dfrac{c_{\ref{6cons_ext_Y2_Yd}}\EE_N}{\psi_1(A',C_N)^t}
\end{align}
where $\cons \label{6cons_ext_Y2_Yd}= \| Y_2 \wedge \ldots \wedge Y_d \|^{-1}$. Furthermore, since $C_N$ is a rational space of dimension $f \in \{e, e+1\}$, by the induction hypothesis $\HH(d-1)$, we have 
for all $ \delta >0, \quad \psi_1(A',C_N) \geq c_{\ref{6cons_minor_angle_A'_CN}}H(C_N)^{-\lambda_{e+1} - \delta}$
where $\cons \label{6cons_minor_angle_A'_CN} >0$ is independent of $N$ and depends on $\delta$. We now fix $\delta \in \left( 0, \frac{\varepsilon}{2} \right)$.

As $X_N, Z_1, \ldots, Z_e$ are integral vectors, we have $H(C_N) \leq\|X_N\| \cdot \| Z_1 \wedge \ldots \wedge Z_e \| \leq c_{\ref{6con_maj_norme_XN}} \theta^{\alpha_N} H(B)$ where $\cons \label{6con_maj_norme_XN} $ is independent of $N$. Equation (\ref{6maj_DN_1}) yields:
\begin{align}\label{6maj_DN_2}
 \DD_N \leq c_{\ref{6cons_ext_Y2_Yd}} c_{\ref{6cons_minor_angle_A'_CN}}^{-t} c_{\ref{6con_maj_norme_XN}}^{t\lambda_{e+1} + t\delta }\theta^{\alpha_N(t\lambda_{e+1} +t\delta) } H(B)^{t\lambda_{e+1} + t\delta}\EE_N.
\end{align}

Furthermore, we bound $\EE_N$. For $N \in \N$, let $Z_N = \theta^{-\floor{\alpha_N}} X_N$ and $W_N = Y_1 - Z_N$. Then we have:
\begin{align*}
 \EE_N &= \theta^{\floor{\alpha_N}}\| Z_N\wedge Y_2 \wedge \ldots \wedge Y_d \wedge Z_1 \ldots \wedge Z_e \| \\
 &= \theta^{\floor{\alpha_N}} \| (Y_1- W_N) \wedge Y_2 \wedge \ldots \wedge Y_d \wedge Z_1 \ldots \wedge Z_e \| \\
 &\leq \theta^{{\alpha_N}}\left(\| Y_1 \wedge Y_2 \wedge \ldots \wedge Y_d \wedge Z_1 \ldots \wedge Z_e \| + \| W_N \wedge Y_2 \wedge \ldots \wedge Y_d \wedge Z_1 \ldots \wedge Z_e \|\right).
\end{align*}
We analyze the two terms separately. Firstly:
\begin{align*}
 \| Y_1 \wedge Y_2 \wedge \ldots \wedge Y_d \wedge Z_1 \ldots \wedge Z_e \| \leq c_{\ref{6cons_major_Y_Z}} \psi_1(A,B) H(B)^{1+\lambda_e + \delta} 
\end{align*}
by Claim~\ref{6lem_maj_Y1_Y_d_Z1_Ze}. Secondly, by construction, we have $\|W_N\| = \| Y_1 -Z_N \| \leq \cons \label{6cons_diff_tronque} \theta^{-\alpha_{N+1}}$, thus:
\begin{align*}
 \| W_N \wedge Y_2 \wedge \ldots \wedge Y_d \wedge Z_1 \ldots \wedge Z_e \|&\leq \| W_N \| \cdot \| Y_2 \wedge \ldots \wedge Y_d \| \cdot \| Z_1 \ldots \wedge Z_e \| \leq c_{\ref{6cons_diff_tronque}} \theta^{-\alpha_{N+1}} c_{\ref{6cons_ext_Y2_Yd2}}H(B)
\end{align*}
where $\cons \label{6cons_ext_Y2_Yd2} =\| Y_2 \wedge \ldots \wedge Y_d \| $. These two inequalities allow us to bound $\EE_N$:
\begin{align*}
 \EE_N \leq c_{\ref{6cons_pour_pas_que_ca_deborde}}\left(\psi_1(A,B) H(B)^{ \lambda_e + \delta } + \dfrac{1 }{\theta^{{\alpha_{N+1}}}} \right)
\end{align*}
where $\cons = \max(c_{\ref{6cons_major_Y_Z}},c_{\ref{6cons_diff_tronque}} c_{\ref{6cons_ext_Y2_Yd2}}) \label{6cons_pour_pas_que_ca_deborde}$.
Substituting this into (\ref{6maj_DN_2}), we obtain:
\begin{align*}
 \DD_N &\leq c_{\ref{6cons_ext_Y2_Yd}} c_{\ref{6cons_minor_angle_A'_CN}}^{-t} c_{\ref{6con_maj_norme_XN}}^{t\lambda_{e+1} + t\delta }\theta^{\alpha_N(t\lambda_{e+1} +t\delta) } H(B)^{t\lambda_{e+1} + t\delta}\theta^{{\alpha_N}} c_{\ref{6cons_pour_pas_que_ca_deborde}}\left(\psi_1(A,B) H(B)^{ \lambda_e + \delta } + \dfrac{1 }{\theta^{{\alpha_{N+1}}}} \right) \\
 &\leq c_{\ref{6cons_maj_D_N_lem}} \theta^{{\alpha_{N}} (1 + t\lambda_{e+1}+t\delta) } H(B)^{1+ t\lambda_{e+1}+t\delta} \left(\psi_1(A,B) H(B)^{ \lambda_e + \delta } + \dfrac{1 }{\theta^{{\alpha_{N+1}}}} \right)
\end{align*}
where $c_{\ref{6cons_maj_D_N_lem}} = \ c_{\ref{6cons_ext_Y2_Yd}} c_{\ref{6cons_minor_angle_A'_CN}}^{-t} c_{\ref{6con_maj_norme_XN}}^{t\lambda_{e+1} } c_{\ref{6con_maj_norme_XN}}^{t \delta}c_{\ref{6cons_pour_pas_que_ca_deborde}}$. 

\end{proof}

We now have all the tools necessary to prove Lemma~\ref{6lem_meilleurs_espaces}.

\begin{proofe}[Lemma~\ref{6lem_meilleurs_espaces}]
We set
\begin{align}\label{6choix_delta}
\delta = \min\left(\dfrac{\varepsilon}{4(t+1)}, \dfrac{n(C_{d-1})^n -1 }{t}\right).
\end{align}
Let $N \in \mathbb{N}$ be the integer satisfying:
\begin{align}\label{6choix_N}
\theta^{\alpha_{N+e-1}(1+t \lambda_{e+1} + t\delta)} \leq H(B)^{K_e - 1-t\lambda_{e+1} - \lambda_e + \frac{\varepsilon}{2}} < \theta^{\alpha_{N+e}(1+t \lambda_{e+1} + t\delta)}
\end{align}
where $t = \min(d-1,e+1)$. This choice makes sense; indeed, $K_e - 1-t\lambda_{e+1} - \lambda_e + \frac{\varepsilon}{2} > 0$ because $K_e \geq C_d = 5n^2(C_{d-1})^{2n} \geq 1 + n(C_{d-1})^{e+1} + (C_{d-1})^{e}$.

We have $\psi_1(A,B) \leq H(B)^{-\lambda_e - \varepsilon}$. Indeed, by hypothesis
$\psi_1(A,B) \leq H(B)^{-K_e- \varepsilon}$
and $K_e \geq (C_d)^e \geq (C_{d-1})^e \geq \lambda_e$. If $H(B)$ is large enough, Claim~\ref{6lem_maj_DN} then gives, for all $i \in \llbracket 0, e-1 \rrbracket$:
\begin{align*}
D_{N+i} &\leq c_{\ref{6cons_maj_D_N_lem}}\theta^{{\alpha_{N+i}} (1 + t\lambda_{e+1} + t\delta) } H(B)^{1+ t\lambda_{e+1} + t\delta} \left(\psi_1(A,B) H(B)^{ \lambda_e + \delta } + \dfrac{1 }{\theta^{{\alpha_{N+i+1}}}} \right) \\
&\leq c_{\ref{6cons_maj_D_N_lem}}\theta^{{\alpha_{N+i}} (1 + t\lambda_{e+1} + t\delta) } H(B)^{1+ t\lambda_{e+1}+ t\delta} \left(H(B)^{-K_e- \varepsilon} H(B)^{ \lambda_e + \delta} + \dfrac{1 }{\theta^{{\alpha_{N+i+1}}}} \right).
\end{align*}
By the choice of $N$ in $(\ref{6choix_N})$, we have for all $i \in \llbracket 0, e-1 \rrbracket$:
$\theta^{\alpha_{N+i}(1+t \lambda_{e+1} + t\delta)} \leq H(B)^{K_e - 1-t\lambda_{e+1} - \lambda_e + \frac{\varepsilon}{2} }$ by the growth of the sequence $(\alpha_N)$. 
Moreover, by the choice of $\delta$ in $(\ref{6choix_delta})$, we have $\frac{-\varepsilon}{2} + t\delta + \delta \leq \frac{-\varepsilon}{4}$. Then for all $i \in \llbracket 0, e-1 \rrbracket$:
\begin{align}\label{6maj_DNi_23}
D_{N+i} &\leq c_{\ref{6cons_maj_D_N_lem}} \left(H(B)^{\frac{-\varepsilon}{4} } + \theta^{{\alpha_{N+i}} (1 + t\lambda_{e+1} + t\delta) - \alpha_{N+i+1} } H(B)^{1+ t\lambda_{e+1}+ t\delta}\right).
\end{align}
We now focus on the second term $\GG_{N+i} = \theta^{{\alpha_{N+i}} (1 + t\lambda_{e+1} + t\delta) - \alpha_{N+i+1} } H(B)^{1+ t\lambda_{e+1}+ t\delta}$. We note $ \eta_e = 1+t \lambda_{e+1} + t\delta$. It satisfies $ \eta_e \leq 2n(C_{d-1})^{n}$ by choice of $\delta $ in $(\ref{6choix_delta})$. We have 
\begin{align*}
\alpha_{N+i} \eta_e - \alpha_{N+i+1} &= \alpha_N (\gamma_{N+1} \ldots \gamma_{N+i})(\eta_e - \gamma_{N+i+1}) \leq 0
\end{align*}
because we can show, by the choice of $C_d$ in \eqref{6definition_Cd}, that:
\begin{align}\label{6diff_Cd_etae}
C_d = 5n^2(C_{d-1})^{2n} \geq (2n(C_{d-1})^{n} +1 )2n(C_{d-1})^{n} \geq \eta_e^2 + \eta_e . 
\end{align}
The choice of $N$ gives a lower bound
$\theta \geq H(B) ^{ \frac{K_e -\eta_e + \frac{\varepsilon}{2}}{\alpha_{N+e}\eta_e}}$ and we can thus bound from above:
$\theta^{{\alpha_{N+i}} \eta_e - \alpha_{N+i+1} } \leq H(B)^{ \frac{(K_e - \eta_e+ \frac{\varepsilon}{2})(\alpha_{N+i} \eta_e - \alpha_{N+i+1})}{\alpha_{N+e}\eta_e}}$
and thus $\GG_{N+i} \leq H(B)^{ \frac{(K_e -\eta_e - \lambda_e + \frac{\varepsilon}{2} )(\alpha_{N+i} \eta_e - \alpha_{N+i+1}) } {\alpha_{N+e}\eta_e} + \eta_e } .$
We then study the exponent $w_i = \frac{(K_e -\eta_e - \lambda_e + \frac{\varepsilon}{2} )(\alpha_{N+i} \eta_e - \alpha_{N+i+1}) } {\alpha_{N+e}\eta_e} + \eta_e $ for $i \in \llbracket0,e-1 \rrbracket$.
\bigskip
We set $\cons \label{6cons_epsilon_negatif} = \dfrac{1}{2K_e}$, a constant independent of $N$. According to $(\ref{6diff_Cd_etae})$ we have $$c_{\ref{6cons_epsilon_negatif}} \leq \dfrac{\eta_e^2}{2K_e\eta_e} \leq \dfrac{ C_d - \eta_e}{2K_e \eta_e} \leq \dfrac{(\gamma_{N+i+1} - \eta_e)(2\gamma_{N+1} \ldots \gamma_{N+i})}{\gamma_{N+1} \ldots \gamma_{N+e}\eta_e}.$$ 
We then bound $w_i$ from above, noting that $\gamma_{N+1} \ldots \gamma_{N+e} \leq K_e$:
\begin{align}\label{6maj_gamma0}
 w_i &= \frac{(K_e -\eta_e - \lambda_e + \frac{\varepsilon}{2} )(\alpha_{N+i} \eta_e - \alpha_{N+i+1}) } {\alpha_{N+e}\eta_e} + \eta_e \nonumber \\
 &= \frac{(K_e -\eta_e - \lambda_e + \frac{\varepsilon}{2} )(\eta_e - \gamma_{N+1})(\gamma_{N+1} \ldots \gamma_{N+i+1}) } {\gamma_{N+1} \ldots \gamma_{N+e}\eta_e} + \eta_e \nonumber \\
 &\leq \frac{(K_e -\eta_e - \lambda_e )(\eta_e - \gamma_{N+i+1})(\gamma_{N+1} \ldots \gamma_{N+i}) + \gamma_{N+1} \ldots \gamma_{N+e}\eta_e^2 } {\gamma_{N+1} \ldots \gamma_{N+e}\eta_e} - c_{\ref{6cons_epsilon_negatif}} \varepsilon \nonumber \\
 &\leq \frac{(K_e -\eta_e - \lambda_e )(\eta_e - \gamma_{N+i+1})(\gamma_{N+1} \ldots \gamma_{N+i}) + K_e\eta_e^2 } {\gamma_{N+1} \ldots \gamma_{N+e}\eta_e} - c_{\ref{6cons_epsilon_negatif}} \varepsilon.
\end{align}
We then bound from above the first term, noting that it is maximal for $i= 0$ since $\gamma_{N+1} \ldots \gamma_{N+i} \geq 0$, 
\begin{align*}
 (K_e -\eta_e - \lambda_e )(\eta_e - \gamma_{N+1}) +K_e\eta_e^2 &\leq (K_e -\eta_e - \lambda_e )(\eta_e - C_d) + K_e\eta_e^2 \\
 &\leq (K_e -\eta_e - \lambda_e )\eta_e^2 + K_e\eta_e^2 \\
 &\leq -(\eta_e + \lambda_e )\eta_e \\
 &\leq 0
\end{align*}
using $(\ref{6diff_Cd_etae})$. Using $(\ref{6maj_gamma0})$ again, we have for all $i \in \llbracket 0,e-1 \rrbracket $, $ w_i \leq - c_{\ref{6cons_epsilon_negatif}} \varepsilon.$
Finally, inequality $(\ref{6maj_DNi_23})$ gives for all $i \in \llbracket 0, e-1 \rrbracket$, 
 $D_{N+i} \leq c_{\ref{6cons_maj_D_N_lem}} \left(H(B)^{\frac{-\varepsilon}{4}} + H(B)^{ -c_{\ref{6cons_epsilon_negatif}} \varepsilon}\right).$
In particular, if $H(B)$ is sufficiently large depending on $c_{\ref{6cons_maj_D_N_lem}},c_{\ref{6cons_epsilon_negatif}} $, and $\varepsilon$, then for all $i \in \llbracket 0, e-1 \rrbracket$,
$\| X_{N+i} \wedge Z_1 \ldots \wedge Z_e \| = \DD_{N+i} < 1.$

According to Lemma~\ref{2lem_X_in_B}, we then have $ \forall i \in \llbracket 0,e-1 \rrbracket, \quad X_{N+i} \in B.$
Recalling that $B_{N,e} = \text{Vect}(X_N, \ldots, X_{N+e-1})$, we have thus shown that, if $H(B)$ is large enough, $B_{N,e} \subset B$ for $N$ satisfying $(\ref{6choix_N})$. By equality of dimensions, we then have $B_{N,e} =B$ and Lemma~\ref{6lem_meilleurs_espaces} is proved.

\end{proofe}

The following corollary concludes the proof of Theorem~\ref{6theo_principal}.

\begin{cor}\label{6cor_maj_expos}
We have:
\begin{align*}
\mu_n(A|e)_1 \leq K_e = \max\limits_{i \in \llbracket 0, n-d-e \rrbracket } \beta_{i+1}\ldots\beta_{i+e}.
\end{align*}
\end{cor}

\begin{proof}
Suppose by contradiction that $\mu_n(A|e)_1 > K_e.$ Then there exists $\varepsilon > 0 $ such that $\mu_n(A|e)_1 \geq K_e + 2\varepsilon.$
By definition of the Diophantine exponent, there exist infinitely many rational subspaces $B$ of dimension $e$ satisfying:
\begin{align}\label{6inegl_psi_1_A_B}
0 < \psi_1(A,B) \leq H(B)^{-\mu_n(A|e)_1 + \varepsilon} \leq H(B)^{- K_e - \varepsilon}.
\end{align}
According to Lemma~\ref{6lem_meilleurs_espaces}, if $H(B)$ is large enough, we have $B = B_{N,e}$ with $N \in \mathbb{N}$ for subspaces satisfying $(\ref{6inegl_psi_1_A_B})$. There exist infinitely many integers $N \in \mathbb{N}$ such that:
\begin{align}\label{6infinite_N_psi}
0 < \psi_1(A,B_{N,e}) \leq H(B_{N,e})^{- K_e - \varepsilon}.
\end{align}
Moreover, according to Lemma~\ref{6min_psi_A_BN}, we have:
\begin{align}\label{6min_psi_1_A_BN_new}
\forall N \in \mathbb{N}, \quad \psi_1(A,B_{N,e}) \geq c_{\ref{6cons_min_angle_1}} \theta^{- {\alpha_{N+e}}} = c_{\ref{6cons_min_angle_1}}\theta^{- {\alpha_{N}}\beta_{N+1} \ldots \beta_{N+e}} \geq c_{\ref{6cons_min_angle_1}}\theta^{- {\alpha_{N}}K_e}.
\end{align}
According to Lemma~\ref{5lem_haut_BN} we have $H(B_{N,e}) \geq \cons \label{6cons_minor_haut_BN} \theta^{{\alpha_N}}.$ By combining inequalities $(\ref{6infinite_N_psi}) $ and $(\ref{6min_psi_1_A_BN_new})$ we find
$c_{\ref{6cons_min_angle_1}}c_{\ref{6cons_minor_haut_BN}}^{K_e}H(B_{N,e}) ^{- K_e} \leq H(B_{N,e})^{- K_e - \varepsilon}.$
Recall that this inequality holds for infinitely many integers $N$. Letting $N$ tend to $+ \infty$, $H(B_{N,e}) \to + \infty$ and then $c_{\ref{6cons_min_angle_1}}c_{\ref{6cons_minor_haut_BN}}^{K_e} = 0$. But $c_{\ref{6cons_min_angle_1}} >0$ and $c_{\ref{6cons_minor_haut_BN}} > 0 $, leading to a contradiction and proving the corollary. 

\end{proof}

\subsection{Proof of Theorem~\ref{1theo_premier_angle}}

We prove here, using Theorem~\ref{6theo_principal}, that the image of the spectrum $(\mu_n(\cdot|1)_1, \ldots, \mu_n(\cdot|n-d)_1)$ contains a non-empty open set. It gives then Theorem~\ref{1theo_premier_angle}.

\bigskip

\begin{proofe}[Theorem~\ref{1theo_premier_angle}]
Let $O$ be the non-empty open set consisting of $(\beta_1, \ldots, \beta_{n-d}) \in \R^{n-d}$ satisfying $\beta_1 > C_d $ and
$\forall i \in \llbracket 1, n-d-1 \rrbracket, \quad C_d \beta_{i} < \beta_{i+1} < \min\limits_{j \in \llbracket 1, i \rrbracket} \beta_{i+1-j}\beta_{j}$. 

For $(\beta_1, \ldots, \beta_{n-d}) \in O$ fixed, let us set $\beta_0 =1 $ and for all
$j \in \llbracket 1,n-1 \rrbracket, \quad \gamma_j = \beta_j \beta_{j-1}^{-1}.$ 
We have $\gamma_1 = \beta_1 \geq C_d$ and $\beta_{j} \geq C_d \beta_{j-1}$ for all $j \in \llbracket 2, n-1 \rrbracket$ by assumption on $O$. We then deduce
$ \forall j \in \llbracket 1, n-1 \rrbracket, \quad \gamma_j \geq C_d.$

 Proposition~\ref{6theo_principal} gives the existence of a subspace
$A$ in $\R^n$ satisfying $A \in \II_n(d,n-d)_1$ and:
 \begin{align*}
 \forall e \in \llbracket 1, n-1 \rrbracket, \quad \mu_n(A|e)_1 = \max\limits_{i \in \llbracket 0, n-d-e\rrbracket} \gamma_{i+1}\ldots\gamma_{i+e}.
 \end{align*}
It remains to show that $\max\limits_{i \in \llbracket 0, n-1-e\rrbracket} \gamma_{i+1}\ldots\gamma_{i+e}= \beta_e$ for all $ e \in \llbracket 1, n-1 \rrbracket$. 
We notice that for $i \in \llbracket 0, n-1-e\rrbracket$,
 $\gamma_{i+1}\ldots\gamma_{i+e} = \dfrac{\beta_{i+e}}{\beta_{i}}$. Taking $i = 0$ we thus have $\max\limits_{i \in \llbracket 0, n-1-e\rrbracket} \gamma_{i+1}\ldots\gamma_{i+e} \geq \beta_e$. On the other hand for $i \in \llbracket 1, n-1-e \rrbracket$, the hypothesis on $O$ gives $\beta_{i+e} \leq \beta_{i}\beta_{e}$ and thus in particular $\gamma_{i+1}\ldots\gamma_{i+e} \leq \beta_e.$

So $\max\limits_{i \in \llbracket 0, n-1-e\rrbracket} \gamma_{i+1}\ldots\gamma_{i+e} \leq \beta_e$ and:
\begin{align*}
\forall e \in \llbracket 1, n-d \rrbracket, \quad \mu_n(A|e)_1 = \beta_e.
\end{align*}
The image of $(\mu_n(\cdot|1)_1, \ldots, \mu_n(\cdot|n-d)_1)$ thus contains a non-empty open set, and in particular, the family of functions $(\mu_n(\cdot|1)_1, \ldots, \mu_n(\cdot|n-d)_1)$ are smoothly independent on $ \II_n(d,n-d)_{1} $.

\end{proofe}

\section{Proof of Theorem~\ref{1theo_construction}}\label{5sect_cas_general}
In this section we give the proof of Theorem~\ref{1theo_construction}. It corresponds to Chapter 7 of \cite{Guillot_these}, where the reader can find more detailed explanations. We take $c_{\ref{7cons_petite_hyp_theoc2c1}}$, $c_{\ref{7cons_petite_hyp_theoc2}}$ and $\beta_{i,\ell}$ as in Theorem~\ref{1theo_construction}.

First, for every $i \in \llbracket 1, d \rrbracket$, we introduce a term $\beta_{i, m+1}$ to ensure a property of linear independence over the sequence $(\beta_{i, \ell})_{i \in \llbracket 1, d \rrbracket, \ell \in \llbracket 1, m+1 \rrbracket}$ (namely \eqref{7_3hypot_m=beta_m+1} below). We choose $(\beta_{i, m+1})_{i \in \llbracket 1, d \rrbracket}$ such that:
\begin{align}
 &\min\limits_{\ell \in \llbracket 1, m +1 \rrbracket} \beta_{1, \ell} \geq \max\left(\left(\max_{\ell \in \llbracket 1, m +1 \rrbracket} \beta_{1, \ell}\right)^{\frac{c_{\ref{7cons_petite_hyp_theoc2}}}{c_{\ref{7cons_petite_hyp_theoc2c1}}}},(3d)^{\frac{c_{\ref{7cons_petite_hyp_theoc2}}}{c_{\ref{7cons_petite_hyp_theoc2}}-1} }\right) \label{7_1hypot_m=beta_m+1},
\end{align}
as well as for every $ i \in \llbracket 1,d-1 \rrbracket,$
\begin{align}
\min\limits_{\ell \in \llbracket 1, m +1 \rrbracket}(\beta_{i,\ell})^{c_{\ref{7cons_petite_hyp_theoc2c1}}} \geq \max_{\ell \in \llbracket 1, m +1\rrbracket}(\beta_{i+1,\ell})\label{7_2ahypot_m=beta_m+1}, \\
 \min\limits_{\ell \in \llbracket 1, m +1 \rrbracket}(\beta_{i+1,\ell})\geq \max_{\ell \in \llbracket 1, m +1 \rrbracket}(\beta_{i,\ell})^{c_{\ref{7cons_petite_hyp_theoc2}}}\label{7_2bhypot_m=beta_m+1}, 
\end{align}
and 
\begin{align}
&\text{the family } \{1\} \cup \left(\frac{\log(E_i)}{\log(E_j)}\right)_{i,j \in \llbracket 1,d \rrbracket^2, i \neq j} \text{ is linearly independent over $\mathbb{Q}$} \label{7_3hypot_m=beta_m+1}
\end{align}
where $E_i = \beta_{i,1}\ldots \beta_{i,m}(\beta_{i,m+1})^m$ for $i \in \llbracket 1,d \rrbracket$. Such $\beta_{i, m+1}$ exist because the inequalities of hypotheses $(\ref{7hypothèse_prop_princi1})$, $(\ref{7hypothèse_prop_princi2})$ and $(\ref{7hypothèse_prop_princi3})$ are strict. The set of $(\beta_{i, m+1})_{i \in \llbracket 1, d \rrbracket}$ satisfying $(\ref{7_1hypot_m=beta_m+1})$, $(\ref{7_2ahypot_m=beta_m+1})$, and $(\ref{7_2bhypot_m=beta_m+1})$ contains a non-empty open set, so we can choose $(\beta_{1,m+1}, \ldots, \beta_{d,m+1})$ satisfying $(\ref{7_3hypot_m=beta_m+1})$ within it because the set of $d$-tuples for which $(\ref{7_3hypot_m=beta_m+1})$ is not satisfied is a countable union of hypersurfaces of $\mathbb{R}^d$.

\bigskip

\subsection{Extension of \texorpdfstring{$\beta_{i,\ell}$}{} and study of \texorpdfstring{$K_{j,v}$}{} }\label{7section_def_beta}

For $i \in \llbracket 1, d \rrbracket$ and $ \ell \in \llbracket m+2, 2m \rrbracket$, we set $\beta_{i, \ell} = \beta_{i, m+1}$ and we extend the sequence of $\beta_{i, \ell}$ by periodicity by setting $\forall i \in \llbracket 1, d \rrbracket, \quad \forall \ell \in \llbracket 1, 2m \rrbracket, \quad \forall p \in \N, \quad \beta_{i, \ell + 2mp } = \beta_{i, \ell}.$

\begin{req}\label{7req_max_2m=m}
By periodicity, we have
$$ \forall i \in \llbracket 1, d \rrbracket, \quad \max\limits_{\ell \in \Nx} \beta_{i,\ell+1}\ldots \beta_{i,\ell+v}= \max\limits_{\ell \in \llbracket 0, 2m -1 \rrbracket} \beta_{i,\ell+1}\ldots \beta_{i,\ell+v}.$$
As $\beta_{i,m+1} = \beta_{i,m+1 } = \ldots = \beta_{i,2m}$ and for all $ \ell \in \llbracket 1, m \rrbracket$, $\beta_{i,m+1} \leq \beta_{i, \ell} $, we have
$$ \forall i \in \llbracket 1, d \rrbracket, \quad \forall v \in \llbracket 1, m +1 \rrbracket, \quad \max\limits_{\ell \in \llbracket 0, 2m -1 \rrbracket} \beta_{i,\ell+1}\ldots \beta_{i,\ell+v}= \max\limits_{\ell \in \llbracket 0, m + 1 -v \rrbracket} \beta_{i,\ell+1}\ldots \beta_{i,\ell+v}.$$
In particular, if $v \leq m $, we notice that $$\max\limits_{\ell \in \llbracket 0, m + 1 -v \rrbracket} \beta_{i,\ell+1}\ldots \beta_{i,\ell+v} = \max\limits_{\ell \in \llbracket 0, m -v \rrbracket} \beta_{i,\ell+1}\ldots \beta_{i,\ell+v} = K_{i,v}.$$
\end{req}
We then extend the definition of $K_{i,v}$ to $v=0$ and $v= m+1$ by:
\begin{align*}
 \forall i \in \llbracket 1, d \rrbracket, \quad \forall v \in \llbracket 0, m +1 \rrbracket, \quad K_{i,v} = \max\limits_{\ell \in \llbracket 0, m+1-v \rrbracket} \beta_{i, \ell +1} \ldots \beta_{i, \ell + v} = \max\limits_{\ell \in \Nx} \beta_{i,\ell+1}\ldots \beta_{i,\ell+v} 
\end{align*}
with the convention that an empty product is equal to $1$, so that $K_{i,0} = 1 $.

Let us define the quantity $f(e,mk) = \max(0,e-mk)$, we will denote it by $f$ when $e$ and $k$ are clearly identified.

\begin{prop}\label{7lem_min_KKi}
 Let $k \in \llbracket 1, d \rrbracket $ and $e \in \llbracket 1, k(m+1)-1 \rrbracket$. Let $1 \leq j_1 \leq \ldots \leq j_k \leq d $. Then for any $q \in \llbracket 1, k \rrbracket $, we have:
 \begin{align}\label{eq7lem_min_KKi}
 \left(1 - \frac{1}{\min\limits_{\ell \in \llbracket 1, m+1\rrbracket}\beta_{j_q, \ell}} \right)\left( \frac{1}{\sum\limits_{\ell =1+ f }^k \frac{1}{K_{j_{\ell},v_{\ell} }} } -1 \right) - K_{j_q,v_q-1} \geq 0
 \end{align}
where $v_{1}, \ldots, v_{k}$ are defined in Theorem~\ref{1theo_construction}.
\end{prop}

 The hypotheses $(\ref{7hypothèse_prop_princi1})$ and $(\ref{7hypothèse_prop_princi2})$ are stronger than the conclusion \eqref{eq7lem_min_KKi} of Proposition $\ref{7lem_min_KKi}$. In fact, assuming that \eqref{eq7lem_min_KKi} holds is sufficient to prove Theorem~\ref{1theo_construction}. Before proving Proposition~$\ref{7lem_min_KKi}$, we will prove two claims.

\begin{claim}\label{7lem_beta_minor}
 For any $\beta \geq \min\limits_{\ell \in \llbracket 1, m+1 \rrbracket} \beta_{1, \ell}$, we have $ \beta - d(1+ 2 \beta^{\frac{1}{c_{\ref{7cons_petite_hyp_theoc2}}}}) \geq 0.$

\end{claim}

\begin{proof}
 Since $\min\limits_{\ell \in \llbracket 1, m+1 \rrbracket} \beta_{1, \ell} \geq (3d)^{\frac{c_{\ref{7cons_petite_hyp_theoc2}}}{c_{\ref{7cons_petite_hyp_theoc2}} - 1 }} \geq 1$ by hypothesis $(\ref{7_1hypot_m=beta_m+1})$, we have:
 \begin{align*}
 \beta^{\frac{1}{c_{\ref{7cons_petite_hyp_theoc2}}}}(\beta^{\frac{c_{\ref{7cons_petite_hyp_theoc2}} -1 }{c_{\ref{7cons_petite_hyp_theoc2}}}} -1) \leq \beta^{\frac{1}{c_{\ref{7cons_petite_hyp_theoc2}}}}(\beta^{\frac{c_{\ref{7cons_petite_hyp_theoc2}} }{c_{\ref{7cons_petite_hyp_theoc2}}}} -1) = \beta -1,
 \end{align*}
 hence we deduce the inequality of Claim~\ref{7lem_beta_minor}.
 
\end{proof}

\begin{claim}\label{7lem_croissance_Kj}
Let $k \in \llbracket 1, d \rrbracket $ and $e \in \llbracket 1, k(m+1)-1 \rrbracket$. Let $1 \leq j_1 < \ldots < j_k \leq d $. Let $u$ be the remainder of the division of $e$ by $k$. Then for any $q \in \llbracket 1, k \rrbracket $ we have:
\begin{align*}
(K_{j_q,v_q -1})^{c_{\ref{7cons_petite_hyp_theoc2}}} \leq K_{j_{u+1}, v_{u+1}} \leq K_{j_q, v_q}.
\end{align*}
\end{claim}

\begin{proof}
We begin by observing that for all $v \in \llbracket 1, m+1 \rrbracket$, one has $(v-1)c_{\ref{7cons_petite_hyp_theoc2c1}}^{d} \leq v$ according to the definition of $c_{\ref{7cons_petite_hyp_theoc2c1}} = \left(1+ \frac{1}{m} \right)^{\frac{1}{d}}$. 
We only need to show the following two inequalities to conclude: 
\begin{align}
&(\max\limits_{\ell \in \llbracket 1, m+1\rrbracket}\beta_{j_{q}, \ell})^{(v_{q}-1)c_{\ref{7cons_petite_hyp_theoc2}} } \leq (\min\limits_{\ell \in \llbracket 1, m+1\rrbracket}\beta_{j_{u+1}, \ell})^{v_{u+1} } \label{7ineg_aprouver_1}\\
\text{and } &(\max\limits_{\ell \in \llbracket 1, m+1\rrbracket}\beta_{j_{u+1}, \ell})^{v_{u+1} } \leq (\min\limits_{\ell \in \llbracket 1, m+1\rrbracket}\beta_{j_{q}, \ell})^{v_{q} }. \label{7ineg_aprouver_2}
\end{align}
Indeed, we have $(\min \limits_{\ell \in \llbracket 1, m+1\rrbracket}\beta_{j, \ell})^{v} \leq K_{j,v} \leq (\max\limits_{\ell \in \llbracket 1, m+1\rrbracket}\beta_{j, \ell})^{v} $ for all $j \in \llbracket 1, d \rrbracket$ and $ v \in \llbracket 1, m+1 \rrbracket$. We distinguish between three cases according to the value of $q$.
\\ \textbullet \, \underline{If $q < u+1$} then by definition of $v_j$ we have $v_q = v_{u+1} + 1 \in \llbracket 2, m+1 \rrbracket. $
We then use hypothesis $(\ref{7_2bhypot_m=beta_m+1})$, applied $j_{u+1}-j_q $ times, which gives:
\begin{align*}
(\max\limits_{\ell \in \llbracket 1, m+1\rrbracket}\beta_{j_{q}, \ell})^{(v_{q}-1)c_{\ref{7cons_petite_hyp_theoc2}}^{j_{u+1}-j_q} } \leq (\min\limits_{\ell \in \llbracket 1, m+1\rrbracket}\beta_{j_{u+1}, \ell})^{v_{q} -1} = (\min\limits_{\ell \in \llbracket 1, m+1\rrbracket}\beta_{j_{u+1}, \ell})^{v_{u+1} };
\end{align*}
this yields $(\ref{7ineg_aprouver_1})$ since $j_{u+1}-j_q \geq 1$. For the other inequality, we use hypothesis $(\ref{7_2ahypot_m=beta_m+1})$, applied $j_{u+1}-j_q $ times, which gives:
\begin{align*}
(\max\limits_{\ell \in \llbracket 1, m+1\rrbracket}\beta_{j_{u+1}, \ell})^{v_{u+1} } \leq (\min\limits_{\ell \in \llbracket 1, m+1\rrbracket}\beta_{j_{q}, \ell})^{(v_{q}-1)c_{\ref{7cons_petite_hyp_theoc2c1}}^{j_{u+1}-j_q } }.
\end{align*}
Now, since $j_{u+1}-j_q \leq d$, we have $(v_{q}-1)c_{\ref{7cons_petite_hyp_theoc2c1}}^{j_{u+1}-j_q } \leq (v_{q}-1)c_{\ref{7cons_petite_hyp_theoc2c1}}^{d} \leq v_q$ and thus:
\begin{align*}
(\max\limits_{\ell \in \llbracket 1, m+1\rrbracket}\beta_{j_{u+1}, \ell})^{v_{u+1} } \leq (\min\limits_{\ell \in \llbracket 1, m+1\rrbracket}\beta_{j_{q}, \ell})^{v_{q}}
\end{align*}
which proves $(\ref{7ineg_aprouver_2})$ in the case $q < u+1$.
\\ \textbullet \, \underline{If $q = u+1$} the inequality $(\ref{7ineg_aprouver_2})$ is trivial. 
Moreover, hypotheses $(\ref{7_2ahypot_m=beta_m+1})$ and $(\ref{7_2bhypot_m=beta_m+1})$ combined give
$ (\min\limits_{\ell \in \llbracket 1, m+1\rrbracket}\beta_{j_{u+1}, \ell})\geq (\max\limits_{\ell \in \llbracket 1, m+1\rrbracket}\beta_{j_{u+1}, \ell})^{\frac{c_{\ref{7cons_petite_hyp_theoc2}}}{c_{\ref{7cons_petite_hyp_theoc2c1}}}}$
and thus:
\begin{align*}
 (\min\limits_{\ell \in \llbracket 1, m+1\rrbracket}\beta_{j_{u+1}, \ell})^{v_{u+1}} \geq (\max\limits_{\ell \in \llbracket 1, m+1\rrbracket}\beta_{j_{u+1}, \ell})^{v_{u+1} {\frac{c_{\ref{7cons_petite_hyp_theoc2}}}{c_{\ref{7cons_petite_hyp_theoc2c1}}}} } \geq (\max\limits_{\ell \in \llbracket 1, m+1\rrbracket}\beta_{j_{u+1}, \ell})^{(v_{u+1}-1) c_{\ref{7cons_petite_hyp_theoc2}} }
\end{align*}
since $ (v-1) c_{\ref{7cons_petite_hyp_theoc2c1}} \leq (v-1) c_{\ref{7cons_petite_hyp_theoc2c1}}^d \leq v$ for all $v \in \llbracket 1, m +1 \rrbracket$ by the remark at the beginning of the proof. This yields $(\ref{7ineg_aprouver_2})$ in the case $q = u+1$.
\\ \textbullet \, \underline{If $q > u+1$} then by definition of $v_j$ we have $v_q = v_{u+1} \in \llbracket 1, m \rrbracket. $ The inequality $(\ref{7ineg_aprouver_2})$ is clear because $\max\limits_{\ell \in \llbracket 1, m+1\rrbracket}\beta_{j_{u+1}, \ell} \leq (\min\limits_{\ell \in \llbracket 1, m+1\rrbracket}\beta_{j_{q}, \ell})^{c_{\ref{7cons_petite_hyp_theoc2}}^{-(j_q-j_{u+1})}} \leq \min\limits_{\ell \in \llbracket 1, m+1\rrbracket}\beta_{j_{q}, \ell}$ by applying the hypothesis $(\ref{7_2bhypot_m=beta_m+1})$ $j_q-j_{u+1}$ times.
\\ The inequality $(\ref{7ineg_aprouver_1})$ comes from $(\ref{7_2ahypot_m=beta_m+1})$ applied $j_q-j_{u+1}$ times: 
\begin{align*}
(\max\limits_{\ell \in \llbracket 1, m+1\rrbracket}\beta_{j_{q}, \ell})^{(v_{q}-1)c_{\ref{7cons_petite_hyp_theoc2}} } \leq (\min\limits_{\ell \in \llbracket 1, m+1\rrbracket}\beta_{j_{u+1}, \ell})^{(v_{q}-1)c_{\ref{7cons_petite_hyp_theoc2}}c_{\ref{7cons_petite_hyp_theoc2c1}}^{j_q-j_{u+1}}}\leq (\min\limits_{\ell \in \llbracket 1, m+1\rrbracket}\beta_{j_{u+1}, \ell})^{v_{q}}
\end{align*}
since $(v_{q}-1)c_{\ref{7cons_petite_hyp_theoc2}}c_{\ref{7cons_petite_hyp_theoc2c1}}^{j_q-j_{u+1}} \leq (v_{q}-1)c_{\ref{7cons_petite_hyp_theoc2c1}}^{d}$ because $c_{\ref{7cons_petite_hyp_theoc2}}\leq c_{\ref{7cons_petite_hyp_theoc2c1}}$ and $ (v_{q}-1)c_{\ref{7cons_petite_hyp_theoc2c1}}^{d}\leq v_q$ by the remark at the beginning of the proof. This yields $(\ref{7ineg_aprouver_1})$ in the case $q > u+1$.

\end{proof}

We can then prove Proposition~\ref{7lem_min_KKi}.

\begin{proofe}[Proposition~\ref{7lem_min_KKi}]
Let $q \in \llbracket 1, k \rrbracket$. We denote by $u$ as the remainder of the division of $e$ by $k$. Recalling that $f = f(e,mk) = max(0, e-mk)$, we have:
\begin{align*} 
 \frac{1}{\sum\limits_{\ell =1+ f }^k \frac{1}{K_{j_{\ell},v_{\ell} }}} \geq \frac{1}{(k-f) \frac{1}{ \min\limits_{\ell \in \llbracket 1 + f, k \rrbracket} K_{j_{\ell},v_{\ell} }}} \geq \frac{ K_{j_{u+1},v_{u+1}} }{ k-f} \geq \frac{ K_{j_{u+1},v_{u+1}} }{ d}
\end{align*}
because according to Lemma~\ref{7lem_croissance_Kj}, $K_{j_{u+1},v_{u+1}} = \min\limits_{\ell \in \llbracket 1 + f, k \rrbracket} K_{j_{\ell},v_{\ell} }$.
Finally, since $\min\limits_{\ell \in \llbracket 1, m+1\rrbracket} \beta_{j_q, \ell} \geq \min\limits_{\ell \in \llbracket 1, m+1\rrbracket} \beta_{1, \ell} \geq 3d$ by hypotheses $(\ref{7_1hypot_m=beta_m+1})$ and $(\ref{7_2bhypot_m=beta_m+1})$, we have
$
1 - \frac{1}{\min\limits_{\ell \in \llbracket 1, m+1\rrbracket}\beta_{j_q, \ell}} \geq \frac{1}{2}.
$
Thus:
\begin{align*}
 \left(1 - \frac{1}{\min\limits_{\ell \in \llbracket 1, m+1\rrbracket}\beta_{j_q, \ell}} \right)\left( \frac{1}{\sum\limits_{\ell =1+ f }^k \frac{1}{K_{j_{\ell},v_{\ell} }} } -1 \right) - K_{j_q,v_q-1} &\geq \frac{1}{2}\left(\frac{ K_{j_{u+1},v_{u+1}} }{ d} -1 \right) - K_{j_q,v_q-1} \\
 &= \frac{K_{j_{u+1},v_{u+1}} -d(1 +2 K_{j_q,v_q-1}) }{2d}.
\end{align*}
Then, to conclude it is sufficient to show that $K_{j_{u+1},v_{u+1}} -d(1 +2 K_{j_q,v_q-1}) \geq 0$.
If $v_q = 1$, then $K_{j_q,v_q-1} = 1$ and $K_{j_{u+1},v_{u+1}} \geq \min\limits_{\ell \in \llbracket 1, m+1\rrbracket} \beta_{1, \ell} \geq 3d$, thus Proposition~\ref{7lem_min_KKi} is proven in this case. Otherwise, if $v_q \geq 2$, applying the lower bound of Lemma~\ref{7lem_croissance_Kj} we have
\begin{align*}
 {K_{j_{u+1},v_{u+1}} -d(1 +2 K_{j_q,v_q-1}) } \geq {K_{j_q,v_q-1}^{c_{\ref{7cons_petite_hyp_theoc2}}} -d(1 +2 K_{j_q,v_q-1}) }.
\end{align*}

We conclude the proof by applying Lemma~\ref{7lem_beta_minor} with $\beta = K_{j_q,v_q-1}^{c_{\ref{7cons_petite_hyp_theoc2}}} \geq \min\limits_{\ell \in \llbracket 1, m+1 \rrbracket} \beta_{1, \ell} $ by $(\ref{7_2bhypot_m=beta_m+1})$, and we find 
$ K_{j_q,v_q-1}^{c_{\ref{7cons_petite_hyp_theoc2}}} -d(1 +2 K_{j_q,v_q-1}) \geq 0.$

\end{proofe}

We introduce the sequences $(\alpha_{i,N})_{N \in \mathbb{N}}$ for $i \in \llbracket 1,d \rrbracket$ defined by:
\begin{align*}
 \alpha_{i,0} &= 1 \\
 \forall N \in \mathbb{N}, \quad \alpha_{i, N+1 } &= \beta_{i,N+1} \alpha_{i,N}
\end{align*}

Recall that we denote by $E_i= \beta_{i,1}\ldots \beta_{i,m}(\beta_{i,m+1})^m$ for $i \in \llbracket 1,d \rrbracket$. In fact, we have $E_i= \beta_{i,1}\ldots \beta_{i,2m}$ according to the definitions of $\beta_{i,\ell}$ at the beginning of Section~\ref{7section_def_beta}, and thus by peridocity of the sequence:
\begin{align}\label{7def_alphaN_formule_EI}
 \forall i \in \llbracket 1, d\rrbracket, \quad \forall N \in \mathbb{N}, \quad \alpha_{i, N} = E_i^{ \lfloor \frac{N}{2m}\rfloor}\beta_{i,1} \ldots \beta_{i,N \mod 2m }= E_i^{ \lfloor \frac{N}{2m}\rfloor} \alpha_{i, N \mod 2m}
\end{align}
where $N \mod 2m$ is the integer $k \in \llbracket 0, 2m-1 \rrbracket$ such that $N \equiv k \mod (2m)$.
We then observe that:
\begin{align*}
 \forall i \in \llbracket 1, d \rrbracket, \quad \forall v \in \llbracket 1, m+1 \rrbracket, \quad K_{i,v} = \max_{\ell \in \llbracket 0, m + 1-v \rrbracket} \beta_{i,\ell+1}\ldots \beta_{i,\ell+v} = \max_{\ell \in \llbracket 0, m+1-v \rrbracket} \frac{\alpha_{i, \ell +v}}{\alpha_{i, \ell}} .
\end{align*}

\subsection{Construction of the subspace \texorpdfstring{$A$}{} }\label{7section_def_A}

We use the constructions made in section~\ref{sect_cas_d1} to define the space $A$ of Theorem~\ref{1theo_construction}. We then recall properties about subspaces, which we will show later to be the best approximations of $A$.

\bigskip

For all $i \in \llbracket 1, d \rrbracket$ and $\ell \in \mathbb{N}$, we have $\beta_{i, \ell} \geq 3d \geq 2 + \frac{\sqrt{5}-1}{2}$, and the sequence $(\beta_{i,N})$ is $(2m)$-periodic. Therefore, we can apply Proposition~\ref{5prop_technique}.
Let $i \in \llbracket 1, d \rrbracket$. According to Proposition~\ref{5prop_technique}, there exists a line $\Delta'_i = \Span(Y'_i) \subset \mathbb{R}^{m+1}$ such that:
\begin{align*}
 \forall e \in \llbracket 1, m \rrbracket, \quad \mu_{m+1}(\Delta'_i|e)_1 = \max_{\ell\in \llbracket 0, 2m-1 \rrbracket} \beta_{i,\ell}\ldots \beta_{i,\ell+e} = \max_{\ell \in \llbracket 0, m -e-1\rrbracket} \beta_{i,\ell}\ldots \beta_{i,\ell+e-1}
\end{align*}
using Remark~\ref{7req_max_2m=m}. Additionally, according to the proof of this property and by fixing $\theta$ to be a prime number greater than 5, $Y'_i$ is of the form 
$ Y'_i = \begin{pmatrix}
 1 & \sigma_{0,i} & \cdots & \sigma_{m-1,i}
 \end{pmatrix}^\intercal$
with $\sigma_{0,i}, \ldots, \sigma_{m-1, i}$ algebraically independent over $\mathbb{Q}$.
Since $n = (m+1)d$, we set for $i \in \llbracket 1, d \rrbracket$:
\begin{align*}
 Y_i = \begin{pmatrix}
 0 &
 \cdots &
 0 &
 (Y_i')^{\intercal} &
 0 &
 \cdots &
 0
 \end{pmatrix}^\intercal \in \{ 0 \}^{(i-1)(m+1)} \times \mathbb{R}^{m+1} \times \{0\}^{n - i(m+1)}
\end{align*}
placing the coordinates of $Y'_i$ between the positions $(i-1)(m+1) + 1$ and $i(m+1) - 1$ so that the vectors $Y_i$ are pairwise orthogonal.
Finally, the ``best'' vectors approximating the $Y_i$ are the $X_{N,i}$ with:
\begin{align*}
 X_{N,i} = \begin{pmatrix}
 0 &
 \cdots &
 0 &
 (X'_{N,i})^{\intercal} &
 0 &
 \cdots &
 0
 \end{pmatrix}^\intercal
\end{align*}
where the $X'_{N,i}$ are defined in (\ref{5def_XN_ chap5}) in section~\ref{sect_cas_d1}. 
For $v \in \llbracket 1, m +1 \rrbracket$ and $N \in \mathbb{N}$, we define:
\begin{align}
 B^i_{N, v} = \Span(X_{N, i}, \ldots, X_{N+v-1,i}).
\end{align}
The vectors $X_{N,i}$ are constructed such that $\text{dim}(B^i_{N,v}) = v$ (see Claim~\ref{5B_N_base_v}).
Lemmas $\ref{5lem_haut_BN}$, $\ref{5lem_angle_entre_A_BN}$ and inequalities $\eqref{5lem_norme_XN}$, \eqref{5omega_XN_Y}then give the following estimates.
\begin{prop}\label{7prop_BNV_chap4}
 Let $i \in \llbracket 1, d\rrbracket$. For $v \in \llbracket 1, m \rrbracket$, we have:
 \begin{align*}
 c_{\ref{7cons_norme_XNi_minor}}\theta^{\alpha_{i,N}} \leq 
 &\| X_{N,i} \| \leq c_{\ref{7cons_norme_XNi_major}}\theta^{\alpha_{i,N}}, \\
 c_{\ref{7cons_hauteur_BNvi_minor}}\theta^{\alpha_{i,N}} \leq 
 &H (B^i_{N, v}) \leq c_{\ref{7cons_hauteur_BNvi_major}}\theta^{\alpha_{i,N}},\\
 c_{\ref{7cons_angle_Yi_XNi_minor}} \theta^{-{\alpha_{i,N+1}}} \leq 
 &\omega(Y_i, X_{N,i}) \leq c_{\ref{7cons_angle_Yi_XNi_major}} \theta^{-{\alpha_{i,N+1}}}, \\
 c_{\ref{7cons_psi1_Yi_BNvi_minor}} \theta^{-{\alpha_{i,N+v}}} \leq 
 &\psi_1(\Span(Y_i), \Span(B^i_{N, v})) \leq c_{\ref{7cons_psi1_Yi_BNvi_major}} \theta^{-{\alpha_{i,N+v}}}.
 \end{align*}
 with $\cons \label{7cons_norme_XNi_minor}$, $\cons \label{7cons_norme_XNi_major}$, $\cons \label{7cons_hauteur_BNvi_minor}$, $\cons \label{7cons_hauteur_BNvi_major}$, $\cons \label{7cons_angle_Yi_XNi_minor}$, $\cons \label{7cons_angle_Yi_XNi_major}$, $\cons \label{7cons_psi1_Yi_BNvi_minor}$, and $\cons \label{7cons_psi1_Yi_BNvi_major}$ as constants independent of $N$.
\end{prop}

Let us finally set $A = \Span(Y_1, \ldots, Y_d) $ and for $J \subset \llbracket 1,d \rrbracket$, $ A_J = \Span_{j \in J} (Y_j).$ Since the vectors $Y_j$ are pairwise orthogonal, we have $\dim(A) = d$ and $\dim(A_J) = \#J$.

\begin{lem}\label{7lem_AJ_irr}
Let $J \subset \llbracket 1,d \rrbracket$ be non-empty and $e \in \llbracket \#J, \#J(m+1) -1 \rrbracket$. Then $A_J$ is $(e,\#J - g(A_J,e)) -$irrational.
\end{lem}

\begin{proof}
Suppose the contrary. Then there exists $B$ a rational subspace of dimension $e$ such that $\dim(A_J \cap B) \geq \#J - g(A_J,e) + g(A_J,e) = \#J$. Since $\dim(A_J) = \#J $, we have $A_J \cap B = A_J$. In particular, $Y_j\in B$ for all $j \in J$. Let $X_1, \ldots, X_e$ be a rational basis of $B$. We have $ Y_j \wedge X_1 \wedge \ldots \wedge X_e = 0$ for all $j \in J $, meaning that all minors of size $e+1$ of the matrix $(Y_j \mid X_1 \mid \ldots \mid X_e)$ vanish.
\\ Let $N$ be a minor of size $e$ extracted from $M= (X_1 \mid \ldots \mid X_e)$. Since $e \leq\#J(m+1) -1$, we can find $j \in J$ such that the rows of non-zero coefficients of $Y_j$ are not all rows of the submatrix extracted from $M$ corresponding to $N$. 
Consider the minor of size $e+1$ of $(Y_j \mid X_1 \mid \ldots \mid X_e)$ obtained by considering the rows of $N$ and the row of a non-zero coefficient of $Y_j$ whose existence has just been established. Expanding this minor with respect to the first column, we find:
\begin{align*}
 0 = \tau N + \tau_1 N_1 + \ldots + \tau_{m}N_{m}
\end{align*}
where $\{\tau, \tau_1, \ldots, \tau_{m}\} = \{1, \sigma_{0,i}, \ldots, \sigma_{m-1,i}\}$ and the $N_j$ are, up to sign, minors of size~$e$ of $(X_1 \mid \ldots \mid X_e)$. Since $\sigma_{0,i}, \ldots, \sigma_{m-1,i}$ are algebraically independent, $\tau, \tau_1, \ldots, \tau_{m}$ are linearly independent over $\mathbb{Q}$ and we find $N = N_1 = \ldots = N_{m} = 0$, hence $N= 0$. Every minor of size $e$ of the matrix $(X_1 \mid \ldots \mid X_e)$ is null, implying $\dim(B) < e$, which is a contradiction.

\end{proof}

\subsection{Intermediate property}\label{7section_prop_inter}
In this section, we prove the following proposition:
\begin{prop}\label{7prop_int}
Let $ J \subset \llbracket 1, d \rrbracket$ with cardinality $k \in \llbracket 1, d \rrbracket$. Write $ J = \{j_1, \ldots, j_k\}$ with $j_1 < \ldots< j_k $.
\\ Then we have:
\begin{align*}
 \forall e \in \llbracket k, k(m+1) -1 \rrbracket, \quad \mu_n(A_J| e)_{k - g(A_J,e)} = \left( \sum\limits_{q = 1+ f}^k \frac{1}{K_{j_q,v_q }} \right)^{-1}.
\end{align*}
\end{prop} 
For $e \in \llbracket k, k(m+1) -1 \rrbracket$, we thus seek to compute the Diophantine exponent $\mu_n(A_J|e)_{k - g(A_J,e)} $, which is defined since $A_J \in \II_n(k,e)_{k - g(A_J,e)}$ by Lemma $\ref{7lem_AJ_irr}$.

The entire Section~\ref{7section_prop_inter} is devoted to proving this proposition. Let $J \subset \llbracket 1, d \rrbracket$ be of cardinality $k$. Without loss of generality, we assume that $J = \llbracket 1, k \rrbracket$. According to the notation of Proposition~\ref{7prop_int}, we have $j_{q} = q$ for all $q \in \llbracket 1, k \rrbracket$. We also fix $e \in \llbracket k, k(m+1) -1 \rrbracket$.

\subsubsection{Best subspaces \texorpdfstring{$C^J_{N}$}{} }

We construct the "best" subspaces approximating $A_J$. Let $i \in \llbracket 1, d \rrbracket$. For $N_i$ an integer, we define
\begin{align*}
 C^i_{N_i} &= \Span(X_{N_i,i}, \ldots, X_{N_i +v_i -1, i}) = B^i_{N_i, v_i},
\end{align*}
where $v_i$ is defined by the relations in $(\ref{7eq_def_vj})$. Since $v_i \leq m +1$ and we have $\dim(C^i_{N_i}) = v_i$. For $N= (N_j)_{ j \in \llbracket 1, k \rrbracket} \in \mathbb{N}^{ k}$, we define
$C^J_{N}= \bigoplus\limits_{j \in \llbracket 1, k \rrbracket} C_{N_j}^j .$
As the subspaces $C_{N_j}^j$ are pairwise orthogonal by construction of the $X_{N,i}$, we have $\dim(C^J_{N}) = \sum\limits_{j = 1}^k v_j = e $. For simplification, we denote by $g = g(A_J, e) = \max(0, k + e - n).$

\begin{claim}\label{7lem_vj=m+1}
 Let $u$ be the remainder of the Euclidean division of $e$ by $k$. If $f= 0 $, then for all $j \in \llbracket 1, k \rrbracket$, we have $v_j \leq m$. 
 If $f > 0$, then $u = f$, 
 \begin{align*}
 \forall j \in \llbracket 1, f \rrbracket, \quad v_j = m+1 
 \text{ and } \forall j \in \llbracket f +1, k \rrbracket, \quad v_j =m,
 \end{align*}
 hence, if $f > 0$, for all $j \in \llbracket 1, f \rrbracket$, we have $C^j_{N_j} = \{ 0 \}^{(j-1)(m+1)} \times \mathbb{R}^{m+1} \times \{0\}^{n - j(m+1)}. $
\end{claim}

\begin{req}
 The last part of this lemma remains true if $f = 0$, with the interval $\llbracket 1,f \rrbracket$ being empty.
\end{req}

\begin{proof}
 If $f = 0$, then $e -km \leq 0$. By writing the Euclidean division $e = kv + u $, we have $v \leq m-1 $ or $(v=m$ and $ u = 0$). For all $j \in \llbracket 1, u\rrbracket $, $v_j = v+1$, and for all $j \in \llbracket u+1, k\rrbracket $, $v_j = v$, thus $v_j \leq m$ for all $j \in \llbracket 1, k \rrbracket$.

 \bigskip
 Otherwise, $f = e - km > 0$. The Euclidean division $e = kv + u $ of $e $ by $k$ then satisfies $v \geq m.$ Additionally, as $e \leq k(m+1)-1 $, we have $v = m$. Finally, $ u = e - kv = e- km = f. $
 By definition of $v_j$, 
 \begin{align*}
 \forall j \in \llbracket 1, f \rrbracket, \quad v_j = v+1 = m+1 
 \text{ and } \forall j \in \llbracket f +1, k \rrbracket, \quad v_j = v = m.
 \end{align*} 
 For $ j \in \llbracket 1, f \rrbracket $, by definition of $X_{N,j} $, one has $ C^j_{N_j} \subset \{ 0 \}^{(j-1)(m+1)} \times \mathbb{R}^{m+1} \times \{0\}^{n - j(m+1)}.$ 
 However, $\dim(C^j_{N_j} ) =v_j = m+1$, thus by equality of dimensions, this proves the last part of the lemma.
 
\end{proof}

\begin{lem}\label{7lem_haut_CNM}
 We have 
 $$c_{\ref{7cons_haut_CN_minor}} \theta^{\sum\limits_{j = f+1 }^k \alpha_{j,N_j}} \leq H(C^J_{N}) \leq c_{\ref{7cons_haut_CN_major}} \theta^{\sum\limits_{j = f+1 }^k \alpha_{j,N_j} }$$
 with $\cons \label{7cons_haut_CN_minor}$ and $\cons \label{7cons_haut_CN_major}$ independent of $N =( N_1, \ldots, N_k)$. 
\end{lem}

\begin{proof}
 We recall that $C^J_{N}= \bigoplus_{j \in \llbracket 1, k \rrbracket} C_{N_j}^j$ and that this sum is orthogonal ; thus 
 $ H(C^J_{N}) = \prod_{j = 1}^k H(C^j_{N_j}).$
 According to Claim~\ref{7lem_vj=m+1}, for all $j \in \llbracket 1, f \rrbracket$, we have $H(C^j_{N_j}) = 1$. Furthermore, for $j \in \llbracket f+ 1, k \rrbracket, C^j_{N_j} = B^j_{N_j, v_j}$ with $ v_j \leq m$. Proposition~\ref{7prop_BNV_chap4} thus gives $ c_{\ref{7cons_hauteur_BNvi_minor}}\theta^{\alpha_{j,N_j}} \leq H(C^j_{N_j}) \leq c_{\ref{7cons_hauteur_BNvi_major}}\theta^{\alpha_{j,N_j}} $, proving the lemma.
 
\end{proof}

We now establish relationships between the angle $\psi_{k-g}(A_J, C^J_N)$ and the height $H(C^J_N)$ in terms of the sequences $(\alpha_{j,N_j})$ for $j \in J$.

\begin{lem}\label{7lem_minor_prox_Y_j_CN}
 There exists $\cons \label{7cons_min_Yj_CNM} > 0 $ independent of $N = (N_1, \ldots,N_k)$ such that 
 \begin{align*}
 \forall i \in \llbracket f+1, k \rrbracket, \quad \omega_1(\Span(Y_i), C^J_{N}) \geq c_{\ref{7cons_min_Yj_CNM}}\theta^{-{\alpha_{i,M_i+1}}} 
 \end{align*}
 by setting for $j \in \llbracket 1,k \rrbracket$, $M_j = N_j + v_j -1 $.
\end{lem}

\begin{proof}
 Let $i \in \llbracket f+ 1, k \rrbracket $. According to Proposition~\ref{7prop_BNV_chap4}, there exists a constant $\cons \label{7cons_minoration_prox_YJ_BN}$ such that 
 \begin{align}\label{7inega_de_la_prop_quonavaitavant}
 \omega_1(Y_i, B^j_{N_i, M_i}) = \psi_1(Y_i, B^i_{N_i, v_i}) \geq c_{\ref{7cons_minoration_prox_YJ_BN}} \theta^{-{\alpha_{i,N_i+ v_i}}} = c_{\ref{7cons_minoration_prox_YJ_BN}} \theta^{-{\alpha_{i,M_i +1}}}.
 \end{align}
 We can choose $c_{\ref{7cons_minoration_prox_YJ_BN}}$ independent of $i$, as $i$ takes only a finite number of values. Let $X \in C^J_{N, M} \setminus \{0\} $. We write 
 $ X = \sum\limits_{j=1}^k V_j$
 with $V_j \in C^j_{N_j}$ for all $j \in \llbracket 1, k \rrbracket$. We will bound $ \omega(Y_i, X) = \frac{\| Y_i \wedge X\|}{\| Y_i \| \cdot \|X\|}$ from below.
 
 \bigskip
 First, for all $j \in \llbracket 1, k \rrbracket $, $V_j \in \{ 0 \}^{(j-1)(m+1)} \times \mathbb{R}^{m+1} \times \{0\}^{n - j(m+1)}$. By decomposing each $V_j$ in the canonical basis, we see that this decomposition involves disjoint sets of vectors for each $j \in \llbracket 1, k \rrbracket$. Consequently, the set of vectors of the canonical basis of $\bigwedge^2 \mathbb{R}^n$ involved in the decomposition of $Y_i \wedge V_j $ are also disjoint for each $j \in \llbracket 1, k \rrbracket$. Thus, the vectors $Y_i \wedge V_j $ are pairwise orthogonal. According to the Pythagorean theorem, we have 
 $$\| Y_i \wedge X\|^2 = \| \sum\limits_{j=1}^k (Y_i \wedge V_j ) \|^2 = \sum\limits_{j=1}^k \|Y_i \wedge V_j \|^2.$$ 
 In particular, for all $j \in \llbracket 1, k \rrbracket$, we have $\| Y_i \wedge X\| \geq \|Y_i \wedge V_j \|$. 
 The $V_j$ are pairwise orthogonal, so $\|X\|^2 = \sum\limits_{j=1}^k \|V_j\|^2$. Thus, there exists $j_0 \in \llbracket 1, k \rrbracket$ such that $\|V_{j_0} \| \geq k^{\frac{-1}{2}}\|X\|$. 
 
 \begin{itemize}
 \item If $j_0 = i$ then:
 \begin{align*}
 \omega(Y_i, X) &\geq \frac{\| Y_i \wedge V_i\|}{\| Y_i \| \cdot \|X\|} \geq \frac{\| Y_i \wedge V_i\|}{\| Y_i \| \cdot k^{\frac{1}{2}}\|V_i\|} \geq \frac{\omega(Y_i,V_i)}{ k^{\frac{1}{2}}} \geq \frac{ \omega_1(Y_i, C^j_{N_i, M_i})}{ k^{\frac{1}{2}}} \geq \frac{c_{\ref{7cons_minoration_prox_YJ_BN}} \theta^{-{\alpha_{i,M_i +1}}}}{ k^{\frac{1}{2}}}
 \end{align*}
 using $(\ref{7inega_de_la_prop_quonavaitavant})$.
 
 \item If $j_0 \neq i$ then by studying the minors of size $2$ of $(Y_i \mid X)$ where we extract the line corresponding to the $1$ of $Y_i$ and another line corresponding to a non-zero coordinate of $V_{j_0}$, we have $\| Y_i \wedge X\|^2 \geq \sum\limits_{v} (1 \times v)^2 = \| V_{j_0} \|^2$ where the $v$ are the non-zero coordinates of $V_{j_0}$. This gives:
 \begin{align*}
 \omega(Y_i,X)
 \geq \frac{\| V_{j_0} \|}{\|Y_i\| \|X\|} \geq \frac{1}{ k^{\frac{1}{2}} \|Y_i\|} \geq \frac{ \theta^{- {\alpha_{i,M_{i} +1}} }}{ k^{\frac{1}{2}}\|Y_i\|}.
 \end{align*}
 We then set $\cons \label{7cons_eni} = \min(\min\limits_{i = 1} ^k \frac{k^{\frac{-1}{2}}}{ \|Y_i\|}, k^{\frac{1}{2}} c_{\ref{7cons_minoration_prox_YJ_BN}} )$ and we have:
 \begin{align*}
 \forall i \in \llbracket f+1, k \rrbracket, \quad \omega_1(\Span(Y_i), C^J_{N})= \min\limits_{X \in C^J_{N} \setminus \{0 \} } \omega(Y_i,X) \geq c_{\ref{7cons_eni}}\theta^{- {\alpha_{i,M_{i} +1}} }.
 \end{align*}
 \end{itemize}
\end{proof}

\begin{lem}\label{7lem_prox_AJ_CN_H(CN)}
 We have:
 \begin{align*}
 c_{\ref{7cons_psik_AJ_CNJ_minor}} H(C^J_{N}) ^{ \frac{- \min\limits_{j = f+1}^k \alpha_{j,M_j+1}} {\sum\limits_{j = f+1 }^k \alpha_{j,N_j} } } \leq \psi_{k-g}(A_J, C^J_{N}) \leq c_{\ref{7cons_psik_AJ_CNJ_major}} H(C^J_{N}) ^{ \frac{- \min\limits_{j = f+1}^k \alpha_{j,M_j+1}} {\sum\limits_{j = f+1 }^k \alpha_{j,N_j} } }
 \end{align*}
 with $\cons \label{7cons_psik_AJ_CNJ_minor}$ and $\cons \label{7cons_psik_AJ_CNJ_major}$ independent of $N$.
\end{lem}

\begin{proof}
 We recall that
 $\psi_{k-g}(A_J, C^J_{N}) = \omega_k(A_J, C^J_{N})$
 since $g = g(A_J,e) $ and $\dim(C^J_{N}) = e$.
 We will first prove the following inequality \begin{align}\label{7trad_approx_Aj_C}
 c_{\ref{7cons_min_angle_AJC}} \theta^{- \min\limits_{j = f+1}^k \alpha_{j,M_j+1}} \leq \psi_{k-g}(A_J, C^J_{N}) \leq c_{\ref{7cons_maj_angle_AJC}}\theta^{- \min\limits_{j = f+1}^k \alpha_{j,M_j+1}}
\end{align}
with $\cons \label{7cons_min_angle_AJC}, \cons \label{7cons_maj_angle_AJC} $ independent of the $N_j$.
 According to Lemma~\ref{7lem_vj=m+1}, for every $j \in \llbracket 1,f \rrbracket$ one has $Y_j \in C^j_{N_j} \subset C^J_{N}$ because $Y_j \in \{ 0 \}^{(j-1)(m+1)} \times \R^{m+1} \times \{0\}^{n - j(m+1)}$.
 We apply Lemma 6.1 of \cite{joseph_exposants}, which can be generalized to the case $d +e > n$, to the subspaces
 \[A_J = \bigoplus_{j =1}^k \Span(Y_j) \quad \text{and} \quad \bigoplus_{j =1}^f \Span(Y_j) \oplus \bigoplus_{j =f+1}^k \Span(X_{M_j,j}) \subset C^J_{N},\]
 recalling the notation $M_j = N_j + v_j -1$.
 Then,
 \begin{align*}
 \omega_k(A_J, C^J_{N}) &\leq \omega_k(A_J, \bigoplus_{j =1}^f \Span(Y_j) \oplus \bigoplus_{j =f+1}^k \Span(X_{M_j,j}) ) \\
 &\leq c_{\ref{7cons_elio_maj}} \left(\sum\limits_{j = 1 }^f \omega(Y_j, Y_j) + \sum\limits_{j = f+ 1 }^k \omega(Y_j, X_{M_j,j}) \right) \\
 &= c_{\ref{7cons_elio_maj}} \sum\limits_{j = f+ 1 }^k \omega(Y_j, X_{M_j,j})
 \end{align*}
 with $\cons \label{7cons_elio_maj} >0 $ independent of $Y_1, \ldots, Y_k$ and $n$. Now, according to Proposition~\ref{7prop_BNV_chap4}, for every $j \in \llbracket f+1, k \rrbracket$,
 $\omega(Y_j, X_{j,M_j}) \leq c_{\ref{7cons_angle_Yi_XNi_major}} \theta^{-{\alpha_{j,M_j+1}}}$
 where $c_{\ref{7cons_angle_Yi_XNi_major}} $ is a constant independent of $N_j$.
 Therefore,
 \begin{align*}
 \omega_k(A_J, C^J_{N}) &\leq c_{\ref{7cons_elio_maj}}c_{\ref{7cons_angle_Yi_XNi_major}} \sum\limits_{j = f+ 1 }^k \theta^{-{\alpha_{j,M_j+1}}} \leq k c_{\ref{7cons_elio_maj}}c_{\ref{7cons_angle_Yi_XNi_major}} \theta^{-\min\limits_{j = f+ 1 }^k {\alpha_{j,M_j+1}}}.
 \end{align*}
 Thus, the upper bound of $(\ref{7trad_approx_Aj_C})$ is proved with $c_{\ref{7cons_maj_angle_AJC}} = k c_{\ref{7cons_elio_maj}}c_{\ref{7cons_angle_Yi_XNi_major}} $. Now, we show the lower bound. 
 For any $j \in \llbracket f+1, k \rrbracket$, since $Y_j \in A_J$, according to Lemma 2.3 of \cite{joseph_spectre}, 
 $\omega_1(\Span(Y_j), C^J_{N}) \leq \omega_k(A_J, C^J_{N})$
 because $\dim(A_J) = k $ and $\dim(C^J_{N}) =e \geq k$. 
 Using Lemma~\ref{7lem_minor_prox_Y_j_CN}, we have $\omega_1(\Span(Y_j), C^J_{N}) \geq c_{\ref{7cons_min_Yj_CNM}}\theta^{-{\alpha_{j,M_j+1}}}$ with $ c_{\ref{7cons_min_Yj_CNM}}$ independent of $N$ and $j$. Thus,
 \begin{align*}
 \omega_k(A_J, C^J_{N}) \geq \max_{j \in \llbracket f+1, k \rrbracket} c_{\ref{7cons_min_Yj_CNM}}\theta^{-{\alpha_{j,M_j+1}}} = c_{\ref{7cons_min_Yj_CNM}}\theta^{- \min\limits_{j = f+1}^k \alpha_{j,M_j+1}}.
 \end{align*}
 Since $\omega_k(A_J, C^J_{N}) =\psi_{k-g}(A_J, C^J_{N})$, we have proved the lower bound of (\ref{7trad_approx_Aj_C}). Finally we use Lemma~\ref{7lem_haut_CNM} to conclude the proof of Lemma~\ref{7lem_prox_AJ_CN_H(CN)}.
 
\end{proof}
\subsubsection{Lower bound on the exponent}
Lemma~\ref{7lem_prox_AJ_CN_H(CN)} = allows us to bound the exponent $\mu_n(A_J| e)_{k - g(A_J,e)} $ from below by considering certain $N = (N_1, \ldots, N_k) \in \N^k$.

\begin{cor}\label{7cor_min_expos}
 We have $\mu_n(A_J| e)_{k - g(A_J,e)} \geq \left(\sum\limits_{i = f+1}^k \frac{1}{K_{i,v_{i} }} \right)^{-1}$
 with $K_{q,v_q} = \max\limits_{\ell \in \llbracket 0, m-1 \rrbracket} \beta_{q, \ell +1} \ldots \beta_{q, \ell + v_q}$.
\end{cor}

\begin{proof}
We recall that if $i \in \llbracket f+1, k \rrbracket$, according to Lemma~\ref{7lem_vj=m+1} we have $v_i \in \llbracket 1, m \rrbracket $ and then $\max\limits_{\ell \in \llbracket 0, m -1\rrbracket} \beta_{i,\ell+1}\ldots \beta_{i,\ell+v_i} = \max\limits_{\ell \in \llbracket 0, m-1 \rrbracket} \frac{\alpha_{i, \ell +v_i}}{\alpha_{i, \ell}}$.
For $i \in \llbracket f+1, k \rrbracket$, we denote by $L_i \in \llbracket 0, m-1 \rrbracket$ an integer such that $ K_{i,v_i}= \max\limits_{\ell \in \llbracket 0, m -1\rrbracket} \beta_{i,\ell+1}\ldots \beta_{i,\ell+v_i} = \frac{\alpha_{i, L_i +v_i}}{\alpha_{i, L_i}}.$ We recall that for $j \in \llbracket 1,k \rrbracket$ and $N_j \in \N$, we have $M_j = N_j + v_j -1 $. For $(N_1, \ldots, N_f )\in \N^{f}$ and $ N_{f+1} \in \Nx$ a fixed multiple of $2m$, we set for $i \in \llbracket f+ 2, k \rrbracket $:
\begin{align}\label{7eq_def_Ni}
 &N_i = 2m \floor{ \frac{N_{f+1}\log(E_{f+1})}{2m\log(E_i)} + \frac{\log(\alpha_{f+1,v_{f+1}-1})}{\log(E_i)} } + L_i.
\end{align}
We know, by Lemma~\ref{7lem_prox_AJ_CN_H(CN)}, that:
\begin{align}\label{7minor_mu_dans_proof}
 \mu_n(A_J| e)_{k - g(A_J,e)} &\geq \limsup\limits_{\underset{2m | N_{f+1}}{N_{f+1} \to + \infty}} \frac{\min\limits_{i = f+1}^k \alpha_{i,M_i +1 } }{ \sum\limits_{i = f+1}^k \alpha_{i,N_i} }
\end{align}
where the $N_{f+2}, \ldots, N_k, M_{f+1}, \ldots, M_k$ are those defined above depending on $N_{f+1}$ which will be chosen later. 
The rest of the proof is devoted to showing that this upper limit is greater than or equal to $ \left(\sum\limits_{i = f+1}^k \frac{1}{K_{i,v_{i} }} \right)^{-1}$. We fix $(N_1, \ldots, N_f ) \in \N^{f}$ and $ N_{f+1} \in \Nx$ a multiple of $2m$ and we study $\frac{\min\limits_{i = f+1}^k \alpha_{i,M_i +1 } }{ \sum\limits_{i = f+1}^k \alpha_{i,N_i} }.$
We recall that we denote by $E_i= \beta_{i,1}\ldots \beta_{i,m}(\beta_{i,m+1})^m$ for $i \in \llbracket 1,d \rrbracket$ and that
\begin{align}
 \text{the family } \{1\} \cup \left(\frac{\log(E_{f+1})}{\log(E_j)}\right)_{j \in \llbracket f+2,k \rrbracket} \text{ is linearly independent over $\Q$} \label{7indp_line_Ej}
\end{align}
according to $(\ref{7_3hypot_m=beta_m+1})$, by choice of the $\beta_{i,m+1}$.
For $i \in \llbracket f+ 2, k \rrbracket $, we define:
\begin{align}\label{7eq_def_deltai}
 &\delta_i = \partfrac{ \frac{N_{f+1}\log(E_{f+1})}{2m\log(E_i)} + \frac{\log(\alpha_{f+1,v_{f+1}+ L_{f+1}-1})}{\log(E_i)} } \in [0,1[
\end{align}
where $\partfrac{u} = u - \floor{u}$ represents the fractional part of $u \in \R$. 
The expression $(\ref{7def_alphaN_formule_EI})$ gives for $i \in \llbracket f+ 1,k \rrbracket $:
\begin{align}\label{7rel_alphaNi_Ei}
 \alpha_{i,N_i} = E_i^{\floor{\frac{N_i}{2m}}} \alpha_{i, N_i \mod 2m} = E_i^{\floor{\frac{N_i}{2m}}} \alpha_{i, L_i}
\end{align}
because $2m$ divides $N_i-L_i$ according to $(\ref{7eq_def_Ni})$. Similarly:
\begin{align}\label{7rel_alphaMi_Ei}
 \alpha_{i,M_i+1} = E_i^{\floor{\frac{N_i + v_i }{2m}}} \alpha_{i, N_i + v_i \mod 2m} = E_i^{\floor{\frac{N_i}{2m}}} \alpha_{i, L_i + v_i } 
\end{align}
since $0 \leq L_i + v_i < 2m$ for $i \in \llbracket f+1, k \rrbracket$; indeed, $L_i \in \llbracket 0, m-1 \rrbracket$ and $v_i \in \llbracket 1, m \rrbracket$ for all $i \in \llbracket f+1, k \rrbracket$ according to Lemma~\ref{7lem_vj=m+1}. According to the dependencies between $N_i$ and $N_{f+1}$, we have: 
\begin{align*} 
 \alpha_{i,N_i} &= E_i^{\floor{\frac{N_i}{2m}}} \alpha_{i, L_i} = E_i^{\floor{ \frac{N_{f+1}\log(E_{f+1})}{2m\log(E_i)} + \frac{\log(\alpha_{f+1,v_{f+1}-1})}{\log(E_i)} } }\alpha_{i, L_i} = E_i^{{ \frac{N_{f+1}\log(E_{f+1})}{2m\log(E_i) } + \frac{\log(\alpha_{f+1,v_{f+1}-1})}{\log(E_i)} } - \delta_i }\alpha_{i, L_i}. \nonumber \\
\end{align*}
Now $E_{f+1}^{ \frac{N_{f+1}}{2m} } = E_{f+1}^{\floor{\frac{N_{f+1}}{2m}}} $ since $2m | N_{f+1}$ and $\alpha_{f+1, N_{f+1}} = E_{f+1}^{\floor{\frac{N_{f+1}}{2m}}} \alpha_{f+1, L_{f+1} } $ according to $(\ref{7rel_alphaNi_Ei})$, so one has 
\begin{align} \label{7rel_alphaNf_et_alpha_Ni}
 \alpha_{i,N_i} = E_{f+1}^{ \frac{N_{f+1}}{2m} }\alpha_{f+1,v_{f+1}-1} E_i^{-\delta_i}\alpha_{i, L_i} = \frac{\alpha_{f+1, N_{f+1}} \alpha_{f+1,v_{f+1}-1} E_i^{-\delta_i}\alpha_{i, L_i}}{\alpha_{f+1, L_{f+1}}}.
\end{align}
\\ Similarly, using $(\ref{7rel_alphaMi_Ei})$ for all $i \in \llbracket f+1, k \rrbracket$, we have:
\begin{align}\label{7rel_alphaNf_et_alpha_Mi}
 \alpha_{i,M_i+1} =\frac{\alpha_{f+1, N_{f+1}} \alpha_{f+1,v_{f+1}-1} E_i^{-\delta_i}\alpha_{i, L_i+v_i}}{\alpha_{f+1, L_{f+1}}}.
\end{align}
We then study:
\begin{align*}
\frac{\min\limits_{i = f+1}^k \alpha_{i,M_i +1 } }{ \sum\limits_{i = f+1}^k \alpha_{i,N_i} } &= \frac{\min\limits_{i = f+1}^k \frac{\alpha_{i,M_i +1 }}{\alpha_{f+1, N_{f+1}}} }{ \sum\limits_{i = f+1}^k \frac{\alpha_{i,N_i}}{\alpha_{f+1, N_{f+1}}} } = \frac{\min(\frac{\alpha_{f+1, L_{f+1} + v_{f+1} }}{\alpha_{f+1, L_{f+1}}},\min\limits_{i = f+2}^k (\frac{\alpha_{f+1, L_{f+1} + v_{f+1} -1}}{\alpha_{f+1, L_{f+1}}} E_i^{-\delta_i}\alpha_{i, L_i +v_i} ) ) }{ 1+ \sum\limits_{i = f+2}^k \frac{\alpha_{f+1, L_{f+1} + v_{f+1} -1}}{\alpha_{f+1, L_{f+1}}} E_i^{-\delta_i}\alpha_{i, L_i } } 
\end{align*}
using the relations $(\ref{7rel_alphaNf_et_alpha_Ni})$ and $(\ref{7rel_alphaNf_et_alpha_Mi})$.
Recalling that $K_{f+1, v_{f+1}} = \frac{\alpha_{f+1, L_{f+1} + v_{f+1} }}{\alpha_{f+1, L_{f+1}}}$ and $\frac{\alpha_{f+1, L_{f+1} + v_{f+1} } } {\beta_{f+1, L_{f+1} + v_{f+1} }}= \alpha_{f+1, L_{f+1} + v_{f+1} -1 }$, we have:
\begin{align*}
 \frac{\min\limits_{i = f+1}^k \alpha_{i,M_i +1 } }{ \sum\limits_{i = f+1}^k \alpha_{i,N_i} } &= \frac{\min(K_{f+1, v_{f+1} },\min\limits_{i = f+2}^k (\frac{K_{f+1, v_{f+1} } }{\beta_{f+1, L_{f+1} + v_{f+1} }}E_i^{-\delta_i}\alpha_{i, L_i +v_i} ) ) }{ 1+ \sum\limits_{i = f+2}^k \frac{K_{f+1, v_{f+1} } }{\beta_{f+1, L_{f+1} + v_{f+1}}} E_i^{-\delta_i}\alpha_{i, L_i } } = \frac{\min(1,\min\limits_{i = f+2}^k (\frac{E_i^{-\delta_i}\alpha_{i, L_i +v_i} }{\beta_{f+1, L_{f+1} + v_{f+1} }} ) ) }{ \frac{1}{K_{f+1, v_{f+1}}} + \sum\limits_{i = f+2}^k \frac{ E_i^{-\delta_i}\alpha_{i, L_i } } {\beta_{f+1, L_{f+1} + v_{f+1} }}} .
\end{align*}

Finally, since $\frac{\alpha_{i, L_i + v_i} }{K_{i,v_i}} = \alpha_{i, L_i }$ for all $i \in \llbracket f+2, k \rrbracket$, we finally have:
\begin{align*}
 \frac{\min\limits_{i = f+1}^k \alpha_{i,M_i +1 } }{ \sum\limits_{i = f+1}^k \alpha_{i,N_i} } &= \frac{\min(1,\min\limits_{i = f+2}^k (\frac{E_i^{-\delta_i}\alpha_{i, L_i +v_i} }{\beta_{f+1, L_{f+1} + v_{f+1} }} ) ) }{ \frac{1}{K_{f+1, v_{f+1}}} + \sum\limits_{i = f+2}^k \frac{ E_i^{-\delta_i}\alpha_{i, L_i + v_i} } {\beta_{f+1, L_{f+1} + v_{f+1} }} \frac{1}{K_{i,v_i}}}.
\end{align*}
We have thus shown that:
\begin{align*}
 \limsup\limits_{\underset{2m | N_{f+1}}{N_{f+1} \to + \infty}} \frac{\min\limits_{i = f+1}^k \alpha_{i,M_i +1 } }{ \sum\limits_{i = f+1}^k \alpha_{i,N_i} } = \limsup\limits_{\underset{2m | N_{f+1}}{N_{f+1} \to + \infty}} \frac{\min(1,\min\limits_{i = f+2}^k (\frac{E_i^{-\delta_i}\alpha_{i, L_i +v_i} }{\beta_{f+1, L_{f+1} + v_{f+1} }} ) ) }{ \frac{1}{K_{f+1, v_{f+1}}} + \sum\limits_{i = f+2}^k \frac{ E_i^{-\delta_i}\alpha_{i, L_i + v_i} } {\beta_{f+1, L_{f+1} + v_{f+1} }} \frac{1}{K_{i,v_i}}}.
\end{align*}
\bigskip
\\Theorem 6.3 and Example 6.1 of \cite[chapter 1.6]{Kuipers_Niederreiter} , $(\ref{7indp_line_Ej})$, and the definition of $\delta_i$ in $(\ref{7eq_def_deltai})$ imply that:
\begin{align*}
 \left\{ (\delta_{f+2},\ldots,\delta_{k}), N_{f+1} \in \Nx, 2m|N_{f+1} \right\} \text{ dense in } [0,1[^{k - f- 1}.
\end{align*}
Thus, using this density, we have:
\begin{align*}
 \limsup\limits_{\underset{2m | N_{f+1}}{N_{f+1} \to + \infty}} \frac{\min\limits_{i = f+1}^k \alpha_{i,M_i +1 } }{ \sum\limits_{i = f+1}^k \alpha_{i,N_i} } &= \sup\limits_{(\delta_i) \in [0,1[^{k - f-1} } \frac{\min(1,\min\limits_{i = f+2}^k (\frac{E_i^{-\delta_i}\alpha_{i, L_i +v_i} }{\beta_{f+1, L_{f+1} + v_{f+1} }} ) ) }{ \frac{1}{K_{f+1, v_{f+1}}} + \sum\limits_{i = f+2}^k \frac{ E_i^{-\delta_i}\alpha_{i, L_i + v_i} } {\beta_{f+1, L_{f+1} + v_{f+1} }} \frac{1}{K_{i,v_i}}}.
\end{align*}
For all $i \in \llbracket f+2, k \rrbracket$ and $\delta_i \in [0,1[$, we define:
$$u_i = \frac{ E_i^{-\delta_i}\alpha_{i, L_i + v_i} } {\beta_{f+1, L_{f+1} + v_{f+1} }} \in \left( \frac{ \alpha_{i, L_i + v_i} } {E_i\beta_{f+1, L_{f+1} + v_{f+1} }}, \frac{ \alpha_{i, L_i + v_i} } {\beta_{f+1, L_{f+1} + v_{f+1} }}\right] $$
and $u_i$ takes all the values in the interval $\left( \frac{ \alpha_{i, L_i + v_i} } {E_i\beta_{f+1, L_{f+1} + v_{f+1} }}, \frac{ \alpha_{i, L_i + v_i} } {\beta_{f+1, L_{f+1} + v_{f+1} }}\right]$ as $\delta_i$ varies in $[0,1)$. 
Furthermore, $1 \in \left( \frac{ \alpha_{i, L_i + v_i} } {E_i\beta_{f+1, L_{f+1} + v_{f+1} }}, \frac{ \alpha_{i, L_i + v_i} } {\beta_{f+1, L_{f+1} + v_{f+1} }}\right]$ for all $i$ because $\alpha_{i, L_i + v_i} \leq E_i = \alpha_{i,2m}$ and $1 \leq \beta_{f+1, L_{f+1} + v_{f+1} } \leq \min\limits_{\ell \in \llbracket 1, 2m\rrbracket} \beta_{i, \ell} \leq \alpha_{i, L_i + v_i}$ according to hypothesis $(\ref{7_2bhypot_m=beta_m+1})$. Thus:
\begin{align*}
 \sup\limits_{(\delta_i) \in [0,1[^{k - f-1} } \frac{\min(1,\min\limits_{i = f+2}^k u_i ) }{ \frac{1}{K_{f+1, v_{f+1}}} + \sum\limits_{i = f+2}^k u_i \frac{1}{K_{i,v_i}}} &\geq \frac{\min(1,\min\limits_{i = f+2}^k 1 ) }{ \frac{1}{K_{f+1, v_{f+1}}} + \sum\limits_{i = f+2}^k \frac{1}{K_{i,v_i}}}= \left(\sum\limits_{i = f+1}^k \frac{1}{K_{i,v_{i} }} \right)^{-1}. 
\end{align*}

We have shown that $\mu_n(A_J|e)_{k - g(A_J,e)} \geq \left(\sum\limits_{i = f+1}^k \frac{1}{K_{i,v_{i} }} \right)^{-1}$ by using again equation $(\ref{7minor_mu_dans_proof})$, which completes the proof of the lemma.

\end{proof}

\subsubsection{Upper bound on the exponent}
In this section, we show that the subspaces $C^J_{N}$ actually achieve the "best" approximations of $A$. This allows us to bound $\mu_n(A_J| e)_{k - g(A_J,e)}$ from above and thus to conclude the proof of Proposition~\ref{7prop_int}.

\begin{lem}\label{7lem_meilleur_espaces}
Let $\varepsilon >0$ and $C$ be a rational subspace of dimension $e$ such that:
\begin{align}\label{7hyp_lem_meilleure_approx}
 \psi_{k-g}(A_J,C) \leq H(C)^{- \left(\sum\limits_{i = f+1}^k \frac{1}{K_{i,v_{i} }} \right)^{-1}- \varepsilon}. 
\end{align}
Then, if $H(C)$ is sufficiently large depending on $\varepsilon$, there exists $N \in (\Nx)^{k}$ such that 
$ C = C^J_{N}.$
\end{lem}

\begin{proof}
In this proof, we set $K = \left( \sum\limits_{j = f+1}^k \frac{1}{K_{j,v_{j} }} \right)^{-1} $.
For $i \in \llbracket 1, k \rrbracket$, we define $N_i$ as the integer satisfying:
\begin{align}\label{7choix_N_i}
\theta^{{\alpha_{i,{N_i+v_i -1 }}} } \leq H(C)^{ K+ \frac{\varepsilon}{2} - 1} < \theta^{{\alpha_{i,{N_i+v_i}}} }.
\end{align}
We will show that $N = (N_1, \ldots, N_k)$ satisfies the conditions. Let $Z_1, \ldots, Z_e$ be a $\Zbasis$ of $C\cap \Z^n$. For $q$ an integer and $i \in \llbracket 1, k \rrbracket$:
\begin{align*}
D_{q,i} = \| X_{q, i} \wedge Z_1 \wedge \ldots \wedge Z_e \|.
\end{align*}
We will prove that for all $i \in \llbracket 1,k \rrbracket$:
\begin{align*}
\forall \ell \in \llbracket 0, v_i -1 \rrbracket, \quad D_{N_{i}+\ell, i } < 1 
\end{align*}
and Lemma~\ref{2lem_X_in_B} will allow us to conclude. We fix $i \in \llbracket 1, k \rrbracket $ and $\ell \in \llbracket 0, v_i -1 \rrbracket $. One has:
\begin{align*}
D_{N_{i}+\ell, i } &= \| p_{C^\perp}(X_{N_i+\ell, i}) \| \cdot \| Z_1 \wedge \ldots \wedge Z_e \| = \omega(X_{N_i+\ell, i}, C) \| X_{N_i+\ell,i} \| H(C).
\end{align*}
Using the triangle inequality on angles, we have
$\omega(X_{N_i+\ell, i}, C) \leq \omega(X_{N_i+\ell, i}, Y_i ) + \omega(Y_i, C).$ Proposition~\ref{7prop_BNV_chap4} yields $$c_{\ref{7cons_norme_XNi_minor}}\theta^{\alpha_{i,N_i+\ell}} \leq \| X_{N_i+\ell,i} \| \leq c_{\ref{7cons_norme_XNi_major}}\theta^{\alpha_{i,N_i+\ell}} \text{ and }
c_{\ref{7cons_angle_Yi_XNi_minor}} \theta^{-{\alpha_{i,N_i+\ell+1}}} \leq \omega(X_{N_i+\ell,i},Y_i) \leq c_{\ref{7cons_angle_Yi_XNi_major}} \theta^{-{\alpha_{i,N_i+\ell+1}}}$$ with constants independent of $N_i$. 
Furthermore, since $Y_i \in A_J$ and $\dim(A_J) = k$, by Lemma 2.3 of \cite{joseph_spectre}, we have $\psi_{k-g}(A_J,C) = \omega_{k}(A_J, C) \geq \omega_1(\Span(Y_i), C). $ Therefore:
\begin{align}
D_{N_{i}+\ell, i } &\leq c_{\ref{7cons_maj_DNell}} H(C) \theta^{\alpha_{i,N_i+\ell}} \left(\theta^{-\alpha_{i,N_i+\ell+1}} + H(C)^{-K - \varepsilon}\right) \nonumber \\
&= c_{\ref{7cons_maj_DNell}} \left(\theta^{-\alpha_{i,N_i+\ell+1} + \alpha_{i,N_i+\ell}} H(C) + \theta^{\alpha_{i,N_i+\ell}}H(C)^{-K - \varepsilon +1 }\right) \label{7maj_DNiell_final}
\end{align}
with $\cons \label{7cons_maj_DNell} $ independent of $N$, using the assumption (\ref{7hyp_lem_meilleure_approx}) on $C$.
The choice of $N_i$ in $(\ref{7choix_N_i})$ gives:
\begin{align} \label{7inega_premier_terme}
\theta^{\alpha_{i,N_i+\ell}}H(C)^{-K - \varepsilon +1 } \leq \theta^{\alpha_{i,N_i+v_i -1 }} \leq H(C)^{K + \frac{\varepsilon}{2} +1 }H(C)^{-K - \varepsilon -1 } = H(C) ^{-\frac{\varepsilon}{2}}.
\end{align}
Now we consider the term $\theta^{-\alpha_{i,N_i+\ell+1} + \alpha_{i,N_i+\ell}} H(C)$ in $(\ref{7maj_DNiell_final})$. Again by $(\ref{7choix_N_i})$, we have:
\begin{align}
\theta^{-\alpha_{i,N_i+\ell+1} + \alpha_{i,N_i+\ell}}H(C) &\leq \left(H(C)^{\frac{K + \frac{\varepsilon}{2} -1}{\alpha_{i,N_i +v_i}}}\right)^{-\alpha_{i,N_i+\ell-1} + \alpha_{i,N_i+\ell}} H(C) \nonumber \\
&= H(C)^{ \frac{(K + \frac{\varepsilon}{2} -1)(-\alpha_{i,N_i+\ell+1} + \alpha_{i,N_i+\ell}) + \alpha_{i,N_i +v_i}}{\alpha_{i,N_i +v_i}}} \label{7inega_deuxieme_terme}.
\end{align}

We then focus on the numerator of this exponent; setting aside the term involving $\varepsilon$, we find:
\begin{align*}
(K -1)(-\alpha_{i,N_i+\ell+1} + \alpha_{i,N_i+\ell}) + \alpha_{i,N_i +v_i} = \alpha_{i,N_i + \ell} \left((K -1)(-\beta_{i,N_i+\ell+1} + 1) + \frac{\alpha_{i,N_i +v_i}}{\alpha_{i,N_i + \ell} }\right). 
\end{align*}
We will show that this term is negative; indeed:
\begin{align*}
(K -1)(-\beta_{i,N_i+\ell+1} + 1) + \frac{\alpha_{i,N_i +v_i}}{\alpha_{i,N_i + \ell} } &= (K - 1)(-\beta_{i,N_i+\ell+1} + 1) + \beta_{i,N_i+\ell+1} \ldots \beta_{i,N_i+v_i} \\
&= -\beta_{i,N_i+\ell+1} (K - 1 - \beta_{i,N_i+\ell+2} \ldots \beta_{i,N_i+v_i}) + K -1.
\end{align*}
Since $\ell \in \llbracket 0, v_i-1 \rrbracket$, we have $\beta_{i,N_i+\ell+2} \ldots \beta_{i,N_i+v_i} \leq \beta_{i,N_i+2} \ldots \beta_{i,N_i+v_i} \leq K_{i,v_i-1} $ because the $v_i-1$ factors of the product $\beta_{i,N_i+2} \ldots \beta_{i,N_i+v_i}$ are consecutive.We then have 
\begin{align}\label{7inegalite_presque_finale}
 (K -1)(-\beta_{i,N_i+\ell+1} + 1) + \frac{\alpha_{i,N_i +v_i}}{\alpha_{i,N_i + \ell} } &\leq -\beta_{i,N_i+\ell+1} (K - 1 - K_{i,v_i -1}) + K -1.
\end{align}
Now, Proposition~\ref{7lem_min_KKi} gives 
$ K - 1 - K_{i,v_i-1} \geq \frac{K-1}{\min\limits_{ \ell \in \llbracket 1, m+1 \rrbracket} (\beta_{i,\ell})}.$
The inequality (\ref{7inegalite_presque_finale}) becomes
\begin{align*}
(K -1)(-\beta_{i,N_i+\ell+1} + 1) + \frac{\alpha_{i,N_i +v_i}}{\alpha_{i,N_i + \ell} } 
&\leq -\beta_{i,N_i+\ell+1} \frac{K-1}{\min\limits_{ \ell \in \llbracket 1, m+1 \rrbracket} (\beta_{i,\ell})} + K-1 \leq (K-1) \left(1 - \frac{\beta_{i,N_i+\ell+1} }{\min\limits_{ \ell \in \llbracket 1, m+1 \rrbracket} (\beta_{i,\ell})} \right)
\end{align*}
and the right hand side of this inequality is a negative number. We then reconsider $(\ref{7inega_deuxieme_terme})$ and we have:
\begin{align*}
\theta^{-\alpha_{i,N_i+\ell+1} + \alpha_{i,N_i+\ell}}H(C) &\leq H(C)^{\frac{\varepsilon}{2}(\frac{-\alpha_{i,N_i + \ell +1} + \alpha_{i,N_i + \ell }}{\alpha_{i,N_i + v_i}})}. 
\end{align*}
Since $\beta_{i,j} \geq 2$ for all $j$, and $v_i \leq m+1 $ by Lemma~\ref{7lem_vj=m+1}, using $(\ref{7_2bhypot_m=beta_m+1})$ we have $$\frac{\alpha_{i,N_i + \ell +1} - \alpha_{i,N_i + \ell }}{\alpha_{i,N_i + v_i}} = \frac{\beta_{i,N_i+ \ell +1} -1 }{\beta_{i,N_i + \ell +1} \ldots \beta_{i,N_i +v_i}} \geq \frac{1}{K_{i,v_i}} \geq \frac{1}{K_{i,m+1}} \geq \frac{1}{K_{1, m+1}}.$$ Setting $\cons \label{7cons_espsK} = \frac{1}{2K_{1, m+1}}$, we then have:
\begin{align}\label{7inega_second_terme}
\theta^{-\alpha_{i,N_i+\ell+1} + \alpha_{i,N_i+\ell}}H(C) &\leq H(C)^{-c_{\ref{7cons_espsK}}\varepsilon}.
\end{align}
Revisiting the inequality $(\ref{7maj_DNiell_final})$ and the estimations $(\ref{7inega_premier_terme})$ and $(\ref{7inega_second_terme})$, we then have $D_{N_{i}+\ell, i } \leq c_{\ref{7cons_maj_DNell}}(H(C)^{-\frac{\varepsilon}{2}} + H(C)^{-c_{\ref{7cons_espsK}}\varepsilon}) $
and therefore, if $H(C)$ is large enough (depending on $\varepsilon, c_{\ref{7cons_maj_DNell}} $, and $c_{\ref{7cons_espsK}}$), we have for all $i \in \llbracket 1, k \rrbracket $:
\begin{align*}
\forall \ell \in \llbracket 0, v_i-1 \rrbracket, \quad \| X_{N_i +\ell, i} \wedge Z_1 \wedge \ldots \wedge Z_e \| = D_{N_{i}+\ell, i } < 1.
\end{align*}
Lemma~\ref{2lem_X_in_B} then implies that $X_{i,N_i+\ell} \in C$ for all $i \in \llbracket 1, k \rrbracket $ and $\ell \in \llbracket 0, v_i-1 \rrbracket$. 
Therefore, $C^J_{N} \subset C$ and by equality of dimensions $C = C^J_{N} $.

\end{proof}

\begin{cor}\label{7cor_maj_exposant}
 We have $ \mu_n(A_J| e)_{k - g(A_J,e)} \leq \left(\sum\limits_{i = f+1}^k \frac{1}{K_{j,v_{j} }} \right)^{-1}.$ 
\end{cor}

\begin{proof}
 Suppose the contrary: there exists $\varepsilon > 0$ such that
$ \mu_n(A_J| e)_{k - g(A_J,e)} \geq\left(\sum\limits_{j= f+1}^k \frac{1}{K_{j,v_{j} }} \right)^{-1} + 2\varepsilon. $
 Then there exist infinitely many rational subspaces $C$ of dimension $e$ such that 
 \begin{align}\label{7maj_dunepart}
 \psi_{k-g}(A_J,C) \leq H(C)^{-\left(\sum\limits_{j= f+1}^k \frac{1}{K_{j,v_{j} }} \right)^{-1} - \varepsilon}. 
 \end{align}
According to Lemma~\ref{7lem_meilleur_espaces}, there exists $N \in (\mathbb{N}^*)^k$ such that $C = C^J_{N}$ if $H(C)$ is large enough. Furthermore, Lemma~\ref{7lem_prox_AJ_CN_H(CN)} gives:
\begin{align}\label{7min_dautrepart}
 \psi_{k-g}(A_J, C^J_{N}) \geq c_{\ref{7cons_prox_CN_AJ}} H(C^J_{N}) ^{ \frac{ - \min\limits_{j = f+1}^k \alpha_{j,M_j+1}} {\sum\limits_{j = f+1 }^k \alpha_{j,N_j} } }
\end{align}
 with $ \cons \label{7cons_prox_CN_AJ} >0 $ independent of $N $. 
 The inequalities $(\ref{7maj_dunepart}) $ and $ (\ref{7min_dautrepart})$ together yield
 \begin{align}\label{7ineg_comp_exposant}
 \frac{\min\limits_{j = f+1}^k \alpha_{j,N_j+v_j}} {\sum\limits_{j = f+1 }^k \alpha_{j,N_j} } \geq \left(\sum\limits_{i = f+1}^k \frac{1}{K_{j,v_{j} }} \right)^{-1}+ \frac{\varepsilon}{2} 
 \end{align}
 as the height tends to $+ \infty$. Let $j \in \llbracket f+ 1, k \rrbracket$. We have $ \left(\sum\limits_{j= f+1}^k \frac{1}{K_{j,v_{j} }} \right)^{-1}\leq \frac{\alpha_{j,N_j+v_j}}{K_{j,v_j}} \leq \frac{ \alpha_{j,N_j} \beta_{j,N_j + 1} \ldots \beta_{j,N_j+v_j}}{K_{j,v_j}} \leq \alpha_{j,N_j} $ since $K_{j,v_j}$ is the maximum of products of $v_j$ consecutive terms among the $\beta_{j,\ell}$. Summing over $j$, we have 
$ (\min\limits_{j=f+1}^k {\alpha_{j,N_j +v_j} } )\sum\limits_{j = f+1}^k \frac{1}{K_{j,v_j}} \leq \sum\limits_{j = f+1}^k \alpha_{j,N_j}. $
Thus $\frac{\min\limits_{j = f+1}^k \alpha_{j,N_j+v_j}} {\sum\limits_{j = f+1 }^k \alpha_{j,N_j} } \leq \left(\sum\limits_{j= f+1}^k \frac{1}{K_{j,v_{j} }} \right)^{-1}$, contradicting $(\ref{7ineg_comp_exposant})$ since $\varepsilon >0$. This completes the proof of the corollary.

\end{proof}

The corollaries~\ref{7cor_min_expos} and~\ref{7cor_maj_exposant} conclude the proof of Proposition~\ref{7prop_int}.

\subsection{Final computation of the exponents}\label{7section_calcul_expo}

We proceed with the proof of Theorem~\ref{1theo_construction} using Proposition~\ref{7prop_int}, and Theorem~\ref{2theo_somme_sev} to derive the exponents $\mu_n(A|e)_{k -g(A,e)}$. For $j \in \llbracket 1,d \rrbracket$, we consider the subspace 
$ R_j = \{0\}^{(j-1)(m+1)} \times \R^{m+1} \times \{0 \}^{(d-j)(m+1)} \subset \R^n.$ The subspaces $R_j$ are rational and in direct sum. Thus, we have $\bigoplus\limits_{j=1}^d R_j \subset \R^n$, and we note that for all $j \in \llbracket 1,d \rrbracket$,
$ \Span(Y_j) \subset R_j.$
We apply Theorem~\ref{2theo_somme_sev} with the $R_j$ and the lines $A_j = \Span(Y_j)$. 

\bigskip

Let $e \in \llbracket 1, n-1 \rrbracket$ and $k \in \llbracket 1 + g(A,e), \min(d,e) \rrbracket$ such that $e < k(m+1)$. For any subset $J \subset \llbracket 1,d \rrbracket$ with cardinality $k$, by Lemma~\ref{7lem_AJ_irr} the subspace 
$ A_J = \Span_{j \in J}(Y_j) $ is $(e, k-g(A_J,e))$-irrational.
The first part of Theorem~\ref{2theo_somme_sev} applied with $k' = k - g(A,e)$ then gives $A \in \II_n(d,e)_{k-g(A,e)}.$
Now we calculate the exponent. The same theorem gives
\begin{align*}
 \mu_n(A|e)_{k - g(A,e)} = \mu_n(A|e)_{k'} = \max\limits_{J \in P(k' + g(A,e),d) } \mu_n(A_J|e)_{k' + g(A,e) - g(A_J,e)}
\end{align*}
with $P(\ell, d)$ being the set of subsets of cardinality $\ell$ of $ \llbracket 1,d \rrbracket$, and therefore
\begin{align}\label{7calcul_exposants_presque_fin}
 \mu_n(A|e)_{k - g(A,e)} = \max\limits_{J \in P(k,d) } \mu_n(A_J|e)_{k - g(A_J,e)} = \max\limits_{J \in P(k,d) } \left(\sum\limits_{q = 1+ f}^{k} \frac{1}{K_{j_q,v_q }} \right)^{-1}
\end{align}
by Proposition~\ref{7prop_int} applied with $J = \{ j_1 < \ldots< j_{k} \}$. It remains to prove the following lemma to complete the proof of Theorem~\ref{1theo_construction}.

\begin{lem}
 We have 
 $$\max\limits_{J = \{ j_1 < \ldots< j_{k} \} } \left(\sum\limits_{q = 1+ f}^{k} \frac{1}{K_{j_q,v_q }} \right)^{-1} = \left(\sum\limits_{q = 1 + f}^{k} \frac{1}{K_{q + d- k +1,v_q }} \right)^{-1} . $$
\end{lem}

\begin{proof}
 Let $J$ be the set $\{d-k+1, \ldots, d\}$, which corresponds to $j_q = q + d- k $ for $q \in \llbracket 1, k \rrbracket$. We have 
$ \max\limits_{J = \{ j_1 < \ldots< j_{k} \} } \left(\sum\limits_{q = 1+ f}^{k} \frac{1}{K_{j_q,v_q }} \right)^{-1} \geq \left(\sum\limits_{q = 1 + f}^{k} \frac{1}{K_{q + d- k +1,v_q }} \right)^{-1} .$
Moreover, the $K_{i,v}$ are increasing with respect to $i$ by $(\ref{7hypothèse_prop_princi3})$, so in particular,
\begin{align*}
\forall 1 \leq j_1 < \ldots < j_k \leq d, \quad \forall q \in \llbracket 1,k \rrbracket, \quad \forall v \in \llbracket 1, m \rrbracket, \quad K_{j_q, v } \leq K_{q +d-k, v}.
\end{align*}
In particular,$ \max\limits_{J = \{ j_1 < \ldots< j_{k} \} }\left(\sum\limits_{q = 1+ f}^{k} \frac{1}{K_{j_q,v_q }} \right)^{-1} \leq \left(\sum\limits_{q = 1 + f}^{k} \frac{1}{K_{q + d- k +1,v_q }} \right)^{-1} $
which completes the proof of the lemma.

\end{proof}

\section{Smooth independence of the exponents}\label{sect_indep}
This section is devoted to proving Theorems~\ref{1theo_min} and~\ref{1theo_d}. It corresponds to the chapter 8 of \cite{Guillot_these}. We establish these proofs by showin that the image of the joint spectrum, as presented in each theorem, contains a subset with a non-empty interior.

\bigskip
Let $ n \in \mathbb{N} $ and $ d \in \llbracket 1, n - 1 \rrbracket $ be fixed such that $d$ divides $n$. \\We study the family of functions $(\mu_n(\cdot|e)_{k-g(d,e,n)})_{(e,k) \in U}$
for $ U $ a subset of $V_{d,n} = \lbrace (e,k) \mid e \in \llbracket 1, n-1 \rrbracket, k \in \llbracket 1 + g(d,e,n), \min(d,e) \rrbracket \rbrace.
$ In this section, we always work with $ d $ and $ n $ fixed. The functions under consideration, denoted by $ \mu_n(\cdot|e)_{k-g(d,e,n)} $, are thus defined on $ \II_n(d, e)_{k-g(d,e,n)} $, the set of $ (e,k-g(d,e,n)) $-irrational subspaces of $ \mathbb{R}^n $ of dimension $d$. For $ U \subset V_{d,n} $, we define the set $ \II_U = \bigcap\limits_{(e,k)\in U} \II_n(d, e)_{k-g(d,e,n)} $ and the mapping
\begin{align*}
M_U: \left| \begin{array}{cccc}
\II_U &\longrightarrow & (\mathbb{R} \cup\{+ \infty \})^U \\
A & \longmapsto & (\mu_n(A|e)_{k-g(d,e,n)})_{(e,k) \in U}
\end{array}\right. .
\end{align*}
If the image of $ M_U $ contains a non-empty open set in $ \mathbb{R}^U $, then the family $ (\mu_n(\cdot|e)_{k-g(d,e,n)})_{(e,k) \in U} $ is smoothly independent on $\II_U$. Indeed, otherwise, there would exist a submersion $\R^U \to \R$ that vanishes on these functions, and the image of $ M_U $ would be contained in an differential hypersurface of $ \mathbb{R}^U $ that has an empty interior. Subsequently, the proofs focus on showin that the images of the considered functions contain a non-empty open set. 

Theorems~\ref{1theo_min} and~\ref{1theo_d} present examples of subfamilies $U$ of $V_{d,n}$ for which $ (\mu_n(\cdot|e)_{k-g(d,e,n)})_{(e,k) \in U} $ is smoothly independent on $\II_U$. It is worth noting that, with Theorem~\ref{1theo_construction}, we can also establish smooth independence for other families. For instance, such a result holds for $ U =\{ (e,k) \mid k \in \llbracket 1, d \rrbracket, e \in \llbracket k , km \rrbracket, k |e \}$, as shown in \cite{Guillot_these}, chapter 8, section 8.4.

\subsection{Sufficient condition}
First we develop two technical lemmas that provide sufficient conditions on $ U \subset V_{d,n} $ for $ M_U(\II_U) $ to contain a non-empty open set and therefore for the associated family of exponents to be smoothly independent on $\II_U$. We write $ n = d(m+1) $ with $ m \in \mathbb{N} $. 

\begin{lem}\label{9lem_technique_faible}
Let $ U \subset V_{d,n} $ with $ \#U \leq dm $ satisfying for all $ (e,k) \in U $, $e < k(m+1).$
For $ (e,k) \in U $ and $ \beta = (\beta_{i,\ell})_{i \in \llbracket 1,d \rrbracket, \ell \in \llbracket 1, m \rrbracket} \in (\mathbb{R}_+^{*})^{dm} $, we define 
\begin{align*}
\Omega_{(e,k)}(\beta) = \sum\limits_{q = 1 + f(e, mk)}^k \beta_{q+d-k,1} \ldots \beta_{q+d-k,v_q(e,k)} = \sum\limits_{q = 1 + f(e, mk) +d - k}^d \beta_{q,1} \ldots \beta_{q,v_{q+k-d}(e,k)},
\end{align*}
which defines a function $ \Omega: (\mathbb{R}_+^{*})^{dm} \to \mathbb{R}^U $, $ \beta \mapsto (\Omega_{(e,k)}(\beta))_{(e,k)\in U} $. Suppose that the Jacobian matrix of $ \Omega $ at $ \beta $, $ J_{\Omega}(\beta) \in \MM_{\#U, dm}(\mathbb{R}) $, has rank $ \#U $ for all $ \beta \in (\mathbb{R}_+^{*})^{dm} $. Then the image of the function $ M_U $ contains a non-empty open set in $ \mathbb{R}^U $. 

\end{lem}

\begin{proof} Let us define the following function 
\begin{align*}
\Omega': \left| \begin{array}{ccc}
(\mathbb{R}_+^{*})^{dm} &\longrightarrow & \mathbb{R}^{U}\\
(\beta_{i, \ell})_{i \in \llbracket 1, d \rrbracket, \ell \in \llbracket 1, m \rrbracket} &\mapsto & \left(\frac{1}{\sum\limits_{q = 1 + f(e, mk)}^k (\beta_{q+d-k,1} \ldots \beta_{q+d-k,v_q(e,k)})^{-1}} \right)_{(e,k) \in U}
\end{array}\right. .
\end{align*}

We will apply Theorem~\ref{1theo_construction} to construct subspaces $ A $ of dimension $ d $ whose family of exponents $ (\mu_n(A|e)_{k } )_{(e,k) \in U} $ range through an open subset of the image of $ \Omega' $. For this, we define $ H' \subset (\mathbb{R}_+^{*})^{dm} $ as the set of $ (\beta_{i, \ell})_{i \in \llbracket 1, d \rrbracket, \ell \in \llbracket 1, m \rrbracket} $ satisfying the hypotheses 
\begin{align}
\forall i \in \llbracket 1,d \rrbracket,& \quad \beta_{i,1} > \ldots > \beta_{i, m}, \label{9hyp_croissance_beta} 
\end{align}
$\min\limits_{\ell \in \llbracket 1, m \rrbracket}(\beta_{1,\ell}) > {(3d)^{\frac{c_{\ref{7cons_petite_hyp_theoc2}}}{c_{\ref{7cons_petite_hyp_theoc2}}-1}}}$, $\min\limits_{\ell \in \llbracket 1, m \rrbracket}(\beta_{1,\ell})^{c_{\ref{7cons_petite_hyp_theoc2c1}}} > \max\limits_{\ell \in \llbracket 1, m \rrbracket}(\beta_{1,\ell})^{c_{\ref{7cons_petite_hyp_theoc2}}}$ and for all $ i \in \llbracket 1,d-1 \rrbracket $: 
$\min\limits_{\ell \in \llbracket 1, m \rrbracket}(\beta_{i,\ell})^{c_{\ref{7cons_petite_hyp_theoc2c1}}} > \max\limits_{\ell \in \llbracket 1, m \rrbracket}(\beta_{i+1,\ell}) $ and $\min\limits_{\ell \in \llbracket 1, m \rrbracket}(\beta_{i+1,\ell}) > \max\limits_{\ell \in \llbracket 1, m \rrbracket}(\beta_{i,\ell})^{c_{\ref{7cons_petite_hyp_theoc2}}}$ 
with $ c_{\ref{7cons_petite_hyp_theoc2c1}}, c_{\ref{7cons_petite_hyp_theoc2}} $ the constants defined in Theorem~\ref{1theo_construction}. The set $ H' $ thus defined is a non-empty open set in $ (\mathbb{R}_+^{*})^{dm} $.

\bigskip
For $ \beta \in H' $, Theorem~\ref{1theo_construction} yields a space $ A $ of dimension $ d $ in $ \mathbb{R}^{n} $ such that for all $ e \in \llbracket 1, n-1\rrbracket $ and $ k \in \llbracket 1 + g(d,e, n), \min(d,e) \rrbracket $ satisfying $ e < k (m+1) $, we have 
\begin{align*}
A \in \II_n(d,e)_{k-g(d,e,n)} \text{ and } \mu_n(A|e)_{k-g(d,e,n)}
&= \left(\sum\limits_{q = 1 + f(e, mk)}^k \frac{1}{K_{q+d-k,v_q(e,k)} } \right)^{-1}
\end{align*}
where, due to hypothesis \eqref{9hyp_croissance_beta}: $\forall i \in \llbracket 1, d \rrbracket, \forall v \in \llbracket 1, m \rrbracket, \quad K_{i,v} = \beta_{i,1} \ldots \beta_{i,v}.$

This result then gives $ \Omega'(H') \subset M_U(\II_U) $. To prove Lemma~\ref{9lem_technique_faible}, we will show that $ \Omega'(H') $ contains a non-empty open set. By composing $ \Omega' $ on both ends with the application that inverts each coordinate, it suffices to show that the smooth function $ \Omega $, defined on $ H = \lbrace (\beta_{i, \ell})_{i \in \llbracket 1, d \rrbracket, \ell \in \llbracket 1, m \rrbracket}, (\frac{1}{\beta_{i, \ell}})_{i \in \llbracket 1, d \rrbracket, \ell \in \llbracket 1, m \rrbracket} \in H' \rbrace $ by 
\begin{align*}
\Omega: \left| \begin{array}{ccc}
H &\longrightarrow & \mathbb{R}^{U}\\
(\beta_{i, \ell})_{i \in \llbracket 1, d \rrbracket, \ell \in \llbracket 1, m \rrbracket} &\mapsto & \left(\sum\limits_{q = 1 + f(e, mk)}^k \beta_{q+d-k,1} \ldots \beta_{q+d-k,v_q(e,k)} \right)_{(e,k) \in U}
\end{array}\right.
\end{align*}
has a non-empty open set in its image. Since by assumption $ J_{\Omega}(\beta) $ has rank $ \#U $ for all $ \beta \in H \subset (\mathbb{R}_+^{*})^{dm} $, the open mapping theorem ensures that $ \Omega(H) $ contains a non-empty open set, which concludes the proof of the lemma.

\end{proof}

We set the notation $ P(\llbracket 1,d \rrbracket \times \llbracket 1, m \rrbracket) $ for the set of subsets of $ \llbracket 1,d \rrbracket \times \llbracket 1, m \rrbracket $, and we define
\begin{align*}
\chi: \left| \begin{array}{ccc}
V_{d,n} &\longrightarrow & P(\llbracket 1,d \rrbracket \times \llbracket 1, m \rrbracket )\\
(e,k) & \longmapsto & (\llbracket 1 + f + d - k, u+d-k \rrbracket \times \llbracket 1, v+1 \rrbracket) \cup (\llbracket u+ 1 + d - k, d \rrbracket \times \llbracket 1, v \rrbracket )
\end{array}\right.
\end{align*}
where $ v $ and $ u $ are respectively the quotient and remainder of the Euclidean division of $ e $ by $ k $, and $ f = f(e,mk)= \max(0, e-mk) $. Furthermore, we recall a definition made in section~\ref{5sect_cas_general}: we associate to $ (e,k) \in V_{d,n} $ the quantities $ v_1(e,k), \ldots, v_k(e,k) $ defined by:
\begin{align}\label{9eq_def_vj}
v_q(e,k) &= \left\{
\begin{array}{lll}
v + 1 &\text{ if } q \in \llbracket 1, u \rrbracket \\
v &\text{ if } q \in \llbracket u+1, k \rrbracket
\end{array}
\right. .
\end{align}
A pair $ (q, \ell) \in \llbracket 1 ,d \rrbracket \times \llbracket 1,m \rrbracket $ belongs to $ \chi(e,k) $ if and only if we have
\begin{align}\label{9apparaisant_dans_omega}
q \geq 1 +f +d -k \text{ and } \ell \leq v_{q+k-d}(e,k).
\end{align}

We state Lemma~\ref{9lem_technique} which is a consequence of the Lemma ~\ref{9lem_technique_faible}.

\begin{lem}\label{9lem_technique}
Let $ U \subset V_{d,n} $ with $ \#U \leq dm $ satisfying for all $ (e,k) \in U $, $e < k(m+1)$.
Suppose furthermore that there exists an order $ <_U $ on $ U $ such that 
\begin{align}\label{9hypothese_sur_U} 
\forall (e,k) \in U, \quad \chi(e,k) \setminus \bigcup\limits_{(e ',k') <_U (e,k)} \chi(e',k') \neq \emptyset. 
\end{align}
Then the image of the function $ M_U $ contains a non-empty open set in $ \mathbb{R}^U $. 
\end{lem}

\begin{proof}
Each $ \Omega_{(e,k)} $, defined in the assumptions of Lemma~\ref{9lem_technique_faible}, is a polynomial in the variables $ \beta_{i, \ell} $. We note using \eqref{9apparaisant_dans_omega} that $ \chi(e,k) $ is in fact the set of indices of the $ \beta_{i,\ell} $ "appearing" in $ \Omega_{(e,k)}(\beta) $. This implies in particular that 
\begin{align}\label{9rel_polynome_nul_ounon}
\frac{\partial \Omega_{(e,k)}}{\partial \beta_{q,\ell}} = \left\{ \begin{array}{cll}
 \prod\limits_{\underset{p \neq \ell}{1 \leq p \leq v_{q +k -d}(e,k)}} \beta_{q,p} & \text{ if } (q,\ell) \in \chi(e,k)\\
0 & \text{ otherwise } 
\end{array} \right. .
\end{align}
According to hypothesis \eqref{9hypothese_sur_U}, we can order $ U =\lbrace \theta_1, \ldots, \theta_{\#U} \} $ in such a way that these quantities respect the order $ <_U $: 
$\forall i < j, \quad \theta_i <_U \theta_j$
and so that, for every $ j \in \llbracket 2, \#U \rrbracket $, there exists $ (q_j,\ell_j) \in \chi(\theta_j) \setminus \bigcup\limits_{i= 1}^{j-1} \chi(\theta_i) $. 
In particular, using \eqref{9rel_polynome_nul_ounon}, for any $ \beta \in (\mathbb{R}_+^{*})^{dm} $
\begin{align}\label{9valeur_pour_trianguler}
\forall i < j, \quad \frac{\partial \Omega_{\theta_i}}{\partial \beta_{q_j,\ell_j}} =0 \text{ and } \frac{\partial \Omega_{\theta_j}}{\partial \beta_{q_j,\ell_j}} = \prod\limits_{\underset{p \neq \ell_j}{1 \leq p \leq v_{q_j +k -d}(\theta_j)}} \beta_{q_j,p} \neq 0. 
\end{align}
Let $ J_{\Omega}(\beta) = \left(\frac{\partial \Omega_{\theta_i}}{\partial \beta_{j,\ell}}\right)_{i \in \llbracket 1, \#U \rrbracket, (j,\ell) \in \llbracket 1, d \rrbracket \times \llbracket 1,m \rrbracket} \in \MM_{\#U, dm}(\mathbb{R}) $ be the Jacobian matrix of $ \Omega $ at $ \beta \in (\mathbb{R}_+^{*})^{dm} $. Here the index $ i $ corresponds to the rows of the matrix and $ (j,\ell) $ to the columns, ordering $ \llbracket 1, d \rrbracket \times \llbracket 1,m \rrbracket $ according to the usual lexicographic order. We will show that $ J_{\Omega}(\beta) $ has maximal rank $ \#U $, and then apply Lemma~\ref{9lem_technique_faible}.
We extract from $ J_{\Omega}(\beta) $ (by possibly interchanging columns) the matrix
\begin{align*}
G_U = \left(\frac{\partial \Omega_{\theta_i}}{\partial \beta_{q_j,\ell_j}}\right)_{i \in \llbracket 1, \#U \rrbracket, j \in \llbracket 1, \#U \rrbracket} \in \MM_{\#U}(\mathbb{R}).
\end{align*}
According to \eqref{9valeur_pour_trianguler}, we have 
\begin{align*}
G_U = \begin{pmatrix}
\frac{\partial \Omega_{\theta_1}}{\partial \beta_{q_1,\ell_1}} & 0 & 0 & \cdots & 0 \\ 
\frac{\partial \Omega_{\theta_2}}{\partial \beta_{q_1,\ell_1}} & \frac{\partial \Omega_{\theta_2}}{\partial \beta_{q_2,\ell_2}} & 0 & \cdots & 0 \\
\vdots & & \ddots & \ddots & \vdots \\
\frac{\partial \Omega_{\theta_{\#U -1 }}}{\partial \beta_{q_1,\ell_1}} & \frac{\partial \Omega_{\theta_{\#U -1 }}}{\partial \beta_{q_2,\ell_2}} & \cdots &\frac{\partial \Omega_{\theta_{\#U -1 }}}{\partial \beta_{q_{\#U -1},\ell_{\#U -1 }}} & 0 \\
\frac{\partial \Omega_{\theta_{\#U }}}{\partial \beta_{q_1,\ell_1}} & \frac{\partial \Omega_{\theta_{\#U }}}{\partial \beta_{q_2,\ell_2}} & & \cdots &\frac{\partial \Omega_{\theta_{\#U }}}{\partial \beta_{q_{\#U },\ell_{\#U }}} 
\end{pmatrix}.
\end{align*}
The matrix $ G_U $ is thus lower triangular, and its diagonal coefficients are non-zero according to \eqref{9valeur_pour_trianguler}, thus it is invertible. 
Therefore, we have shown that $ J_{\Omega}(\beta) $ has rank $ \#U $ for all $ \beta \in H $. Lemma~\ref{9lem_technique_faible} then implies that the image of the function $ M_U $ contains a non-empty open set in $ \mathbb{R}^U $, which concludes the proof of Lemma~\ref{9lem_technique}. 

\end{proof}

\subsection{Results for the Last Angle}\label{9section_dernier_angle_theo}

\bigskip
In this section, we prove Theorems~\ref{1theo_min} and~\ref{1theo_d} using the results from the previous section. We write $ n = d(m+1) $ with $ m \in \mathbb{N} $.

\begin{proofe}[Theorem~\ref{1theo_min}]
We set $U = \{ (e,\min(d,e)) \mid e \in \llbracket 1, n-d \rrbracket \}.$
We will show that $ U $ satisfies the hypothesis \eqref{9hypothese_sur_U} of Lemma~\ref{9lem_technique}. First, note that for all $ (e,k) \in U $, we have $ g(d,e,n) = 0 $ and $ f = f(e,mk) = 0 $ because $ e \in \llbracket 1, dm \rrbracket $ and $ dm = n-d $. We define the following order on $ U $:
 \begin{align*}
 (e',k') <_U (e,k) \Longleftrightarrow e' < e.
 \end{align*}
This indeed defines an order since the first coordinates of the elements of $ U $ are pairwise distinct. Let $ (e,k) \in U $. If $ e \leq d $ then $ k =e $ and 
$ \chi(e,k) = \llbracket 1+d-e, d \rrbracket \times \{ 1 \}.
$ 
 For any $ (e',k') <_U (e,k) $, we have $ \chi(e',k') = \llbracket 1+d-e', d \rrbracket \times \{ 1 \} $ and thus $ (1+d-e, 1) \in \chi(e,k) \setminus \bigcup\limits_{(e',k') <_U (e,k)} \chi(e',k') $.
 
 \bigskip
 If $ e > d $, then $ k =d $ and we write $ e = d v + u $ for the Euclidean division of $ e $ by $ d $. We have 
$ \chi(e,k) = \llbracket 1, u \rrbracket \times \llbracket 1, v+1 \rrbracket \cup \llbracket u+ 1, d \rrbracket \times \llbracket 1, v \rrbracket. $
 For $ (e',k') <_U (e,k) $, we write $ e' = k'v' + u' $ for the Euclidean division of $ e' $ by $ k' $. We distinguish vetween two cases based on whether $ u = 0 $ or $ u \neq 0 $. 
 \\ \textbullet \, \underline{If $ u =0 $.} We will show that $ (d, v) \notin \chi(e',k') $. 
Since $ e > d $ and $ u =0 $, we necessarily have $ v > 1 $. If $ k' =e' $ then $ v' = 1 $ and $ u' = 0 $, so $ \chi(e',k') = \llbracket 1+d-e', d \rrbracket \times \{ 1 \} $ and $ (d,v) \notin \chi(e',k') $ because $ v >1 $.
 \\Otherwise $ k' = d $ and then $ 1 \leq v' < v $ since $ e' < e $. We have 
 $ \chi(e',k') =\llbracket 1, u' \rrbracket \times \llbracket 1, v'+1 \rrbracket \cup \llbracket u'+ 1, d \rrbracket \times \llbracket 1, v' \rrbracket $.
 Then $ (d,v) \notin \chi(e',k') $ because $ u' \leq d-1 $ and $ v' < v $.
 \\ \textbullet \, \underline{If $ u > 0 $.} We will show that $ (u, v+1) \notin \chi(e',k') $. If $ k' =e' $ then $ v' = 1 $, $ u'= 0 $, and $ \chi(e',k') = \llbracket 1+d-e', d \rrbracket \times \{ 1 \} $. Since $ v+1 > 1 $, we have $ (u, v+1) \notin \chi(e',k') $ in this case.
 \\Otherwise $ k' = d $ and $ v' \leq v $ because $ e' <e $. We have 
 $ \chi(e',k') =\llbracket 1, u' \rrbracket \times \llbracket 1, v'+1 \rrbracket \cup \llbracket u'+ 1, d \rrbracket \times \llbracket 1, v' \rrbracket $.
 Therefore, 
 $\left( (u,v+1) \in \chi(e',k') \Longleftrightarrow [ v' = v \text{ and } u' \geq u ]\right).$
Since $ e' <e $, this property is not satisfied and thus $ (u,v+1) \notin \chi(e',k') $.

 \bigskip
 We have shown that for all $ (e,k) \in U $, there exists $ (q,\ell) $ such that $ (q,\ell) \in \chi(e,k) \setminus \bigcup\limits_{(e',k') <_U (e,k)}\chi(e',k').
$

Thus, the hypotheses of Lemma~\ref{9lem_technique} are satisfied: $ M_U(\II_U) $ contains a non-empty open set, hence the functions $(\mu_n(\cdot|e)_{k})_{(e,k) \in U}$ are smoothly independent on $\II_U$ .

\end{proofe}

Lemma~\ref{9lem_technique} is not enough to show Theorem~\ref{1theo_d}. We will use Lemma~\ref{9lem_technique_faible}.

\bigskip

\begin{proofe}[Theorem~\ref{1theo_d}]
We define $U = \llbracket d, n-1 \rrbracket \times \{ d \}.$ 
 We will show the invertibility for all $\beta \in (\R_+^{*})^{dm} $ of the Jacobian matrix $ J_{\Omega}(\beta)\in \MM_{dm, dm}(\R)$ where 
 $\Omega_{(e,d)}(\beta) =\sum\limits_{q = 1 + f(e, md) }^d \beta_{q,1} \ldots \beta_{q,v_{q}(e,d)} $ for $(e,d) \in U$ with $v_1(e,d), \ldots, v_d(e,d) $ defined in $(\ref{9eq_def_vj})$ and $f(e,md) = \max(0, e- md)$. We then have $J_{\Omega}(\beta) = \begin{pmatrix}
 D_1 & \cdots & D_d
 \end{pmatrix} $
 with for $i \in \llbracket 1,d \rrbracket$:
 \begin{align*}
 D_i = \left(\dfrac{\partial \Omega_{(e,d)}}{\partial \beta_{i,\ell}}\right)_{e \in \llbracket d, n-1 \rrbracket, \ell \in \llbracket 1,m \rrbracket} \in \MM_{dm,m}(\R).
 \end{align*}
Given the form of $\Omega_{(e,d)}$, for $(i,\ell) \in \llbracket 1, d \rrbracket \times \llbracket 1, m \rrbracket$ one has
\begin{align*}
 \dfrac{\partial \Omega_{(e,d)}}{\partial \beta_{i,\ell}} = \left\{ \begin{array}{cll}
0 \prod\limits_{\underset{p \neq \ell}{1 \leq p \leq v_{i }(e,d)}} \beta_{i,p} & \text{ if } \ell \leq v_{i}(e,d)\\
 0 & \text{ otherwise } 
 \end{array} \right. .
\end{align*}
In particular, for $\ell \in \llbracket 1 , v_{i}(e,d) -1 \rrbracket$ we have 
$\dfrac{\beta_{i, \ell+1}}{\beta_{i, \ell}} \dfrac{\partial \Omega_{(e,d)}}{\partial \beta_{i,\ell+1}} = \dfrac{\partial \Omega_{(e,d)}}{\partial \beta_{i,\ell}}.$
Let $D_{i,1}, \ldots,D_{i,m}$ be the columns of $D_i$. We successively perform the following elementary operations between the columns of $J_\Omega(\beta)$:
\begin{align*}
 \forall i \in \llbracket 1, d \rrbracket, \quad \forall \ell \in \llbracket 1, m-1 \rrbracket, \quad D_{i, \ell} \leftarrow \left(D_{i, \ell} - \dfrac{\beta_{i, \ell+1}}{\beta_{i, \ell}}D_{i, \ell+1} \right), 
\end{align*}
and denote by $G$ the resulting matrix. Thus, $\det(J_\Omega(\beta)) = \det(G)$ where $ G =\left(g_{e,(i,\ell)} \right)_{ e \in \llbracket d, n-1 \rrbracket, (i,\ell) \in \llbracket 1,d \rrbracket \times \llbracket 1,m \rrbracket}$ and 
\begin{align}\label{9def_matrice_G}
 g_{e,(i,\ell)} \left\{ \begin{array}{lll}
 \neq 0 & \text{ if } v_i(e,d) = \ell \\
 = 0 & \text{ otherwise}
 \end{array}\right. .
\end{align}
We then study $G$ and for that, we reorder the columns of $G$, setting the following order on $\llbracket 1,d \rrbracket \times \llbracket 1,m \rrbracket$:
$\left( (i, \ell) < (j,\ell') \Longleftrightarrow [ \ell < \ell' \text{ or } (\ell = \ell' \text{ and } i <j)]
\right)$
which is the lexicographic order on $\llbracket 1,d \rrbracket \times \llbracket 1,m \rrbracket$ with priority comparison of the second element of the pair. We denote by $G'$ the resulting matrix and we have $\det(G) = \pm \det(G')$. With this chosen order for the columns, the diagonal coefficients of $G'$ are those of the form $g_{e,(u+1,v)} $ with $e = dv + u$ the Euclidean division of $e$ by $d$ for $e \in \llbracket d ,n -1 \rrbracket$. 
We will show that the matrix $G'$ is upper triangular with non-zero diagonal coefficients. 

Let $e \in \llbracket d ,n -1 \rrbracket$. We set $e = dv + u$ the Euclidean division of $e$ by $d$. According to the definition of $v_i(e,d)$ in $(\ref{9eq_def_vj})$ we have $v_{u+1}(e,d) = v$. The construction of $G$ in $(\ref{9def_matrice_G})$ then gives 
$g_{e,(u+1,v)} \neq 0 $. Now consider $e' > e $. We will show that $g_{e',(u+1,v)} =0 $. We set $e' = dv' + u'$ the Euclidean division of $e'$ by $d$ and we then have $v' \geq v$. 
\\ \textbullet \, \underline{If $v' > v $} then since $v_{u+1}(e',d) = v' \text{ or } v' +1 $ we have $v_{u+1}(e',d) \neq v$ and thus 
$g_{e',(u+1,v)} =0 $
according to $(\ref{9def_matrice_G})$.
\\ \textbullet \, \underline{If $v' = v $} then $u' > u$ and $u +1 \in \llbracket 1 , u' \rrbracket$. In particular, according to the definition of $v_i(e',d)$ in $(\ref{9eq_def_vj})$, we have $v_{u+1}(e',d) = v' +1 \neq v$ and thus 
$g_{e',(u+1,v)} =0 $
according to $(\ref{9def_matrice_G})$.
\\We have thus shown that all diagonal coefficients of $G'$ are non-zero and that $G'$ is upper triangular.
It follows that $G$ is invertible since $\det(G) = \pm \det(G')$ and thus $J_\Omega(\beta)$ is also invertible. Lemma~\ref{9lem_technique_faible} then allows us to conclude that $M_U(\II_U)$ contains a non-empty open set and hence the functions $(\mu_n(\cdot|e)_{d} )_{e \in \llbracket dm, d(m+1)-1 \rrbracket}$ are smoothly independent on $\II_U$. 
 
\end{proofe}

\textbf{Acknowledgements:} I sincerely thank Stéphane Fischler for his reading and corrections of this paper.

\bibliographystyle{smfplain.bst}
\bibliography{bibliographie.bib}

Gaétan Guillot, Université Paris-Saclay, CNRS, Laboratoire de mathématiques d’Orsay, 91405 Orsay, France.
\\ {Email adress:} guillotgaetan1@gmail.com 
\bigskip

Keywords : Diophantine approximation, Geometry of numbers

MSC: Primary 11J13, Secondary 11J17, 11J25
\end{document}

%% file: commandes_latex.tex
\newcommand{\R}[0]{\mathbb{R}}

\newcommand{\N}[0]{\mathbb{N}}
\newcommand{\Z}[0]{\mathbb{Z}}
\newcommand{\Q}[0]{\mathbb{Q}}
\renewcommand{\P}[0]{\mathbb{P}}
\newcommand{\Nx}[0]{\mathbb{N}\setminus \{0 \}}

\newcommand{\RR}[0]{\mathcal{R}}

\newcommand{\Mat}[0]{\text{Mat}}
\newcommand{\Ind}[0]{\text{Ind}}
\newcommand{\MM}[0]{\mathcal{M}}

\newcommand{\DD}[0]{\mathcal{D}}
\newcommand{\EE}[0]{\mathcal{E}}
\newcommand{\FF}[0]{\mathcal{F}}
\newcommand{\HH}[0]{\mathcal{H}}
\newcommand{\GG}[0]{\mathcal{G}}
\newcommand{\GL}[0]{\text{GL}}

\newcommand{\udots}[0]{\reflectbox{$\ddots$}}
\renewcommand{\mod}[0]{\text{ mod }}
\newcommand{\partfrac}[1]{\left\{ #1 \right\} }

\newcommand{\II}[0]{\mathcal{I}}

\newcommand{\floor}[1]{\left\lfloor #1 \right\rfloor}

\let\oldforall\forall
\renewcommand{\forall}{\oldforall \, }

\newcommand{\Zbasis}[0]{\mathbb{Z}\text{-basis}}
\newcommand{\tir}[0]{\text{-}}

\newcommand{\Span}[0]{\text{Span}}
\newcommand{\rank}[0]{\text{rank}}

\theoremstyle{definition}
\newtheorem{theo}{Theorem}[section]
\newtheorem{req}[theo]{Remark}
\newtheorem{claim}[theo]{Claim}

\newtheorem{cor}[theo]{Corollary}
\newtheorem{prop}[theo]{Proposition}
\newtheorem{lem}[theo]{Lemma}
\newtheorem{defi}[theo]{Definition}

\newcounter{constante}
\newcommand{\cons}[0]{\refstepcounter{constante}c_{\theconstante}}

\makeatletter
\renewcommand\theequation%
{\thesection.\arabic{equation}}
\@addtoreset{equation}{section}
\makeatother

\newenvironment{proofe}[1][]{\textit{Proof of #1.}}{\hfill$\qed$} 

\usepackage{dsfont}